\documentclass[12pt]{amsart}
\usepackage{amssymb, labelfig, 
graphics, amsxtra, amsthm}
\usepackage{hyperref}

\newtheorem{thm}{Theorem}
\newtheorem{prop}[thm]{Proposition}
\newtheorem{lem}[thm]{Lemma}
\newtheorem{sublem}[thm]{Sublemma}
\newtheorem{cor}[thm]{Corollary}

\newtheorem{fact}[thm]{Fact}

\theoremstyle{remark}
\newtheorem{rem}[thm]{Remark}

\theoremstyle{definition}

\newcommand{\C}{\mathbb C}
\newcommand{\R}{\mathbb R}
\newcommand{\Z}{\mathbb Z}

\newcommand{\SL}{\mathrm{SL}_n(\R)}

\newcommand{\GL}{\mathrm{GL}_n(\R)}
\newcommand{\PGL}{\mathrm{PGL}_n(\R)}
\newcommand{\Flag}{\mathrm{Flag}(\R^n)}

\newcommand{\F}{\mathcal F}
\newcommand{\B}{\mathcal B}
\newcommand{\E}{\mathrm{e}}

\newcommand{\sln}{\mathfrak{sl}_n(\R)}
\newcommand{\Ad}{\mathop{\mathrm{Ad}}}

\newcommand{\bdry}{\partial_\infty \widetilde S }
\newcommand{\wt}{\widetilde}
\newcommand{\wh}{\widehat}
\newcommand{\Left}{\mathrm{left}}
\newcommand{\Right}{\mathrm{right}}
\newcommand{\Out}{\mathrm{out}}

\newcommand{\Hit}{\mathrm{Hit}_n}
\newcommand{\Id}{\mathrm{Id}}

\DeclareMathOperator*{\Bigcirc}{\bigcirc}
\newcommand{\cupp}{{\smallsmile}}

\newcommand{\db}{/\kern -4pt/}

\renewcommand{\leq}{\leqslant}
\renewcommand{\geq}{\geqslant}
\renewcommand{\phi}{\varphi}
\renewcommand{\epsilon}{\varepsilon}

\title[The symplectic structure
of the $\PGL$--Hitchin component]
{The symplectic structure
of 
\\
the $\PGL$--Hitchin component}
\author{Francis Bonahon}
\address{University of Southern California, Los Angeles, CA 90089-2532, U.S.A.}
\address{Michigan State University, East Lansing, MI 48824, U.S.A.}
\email{fbonahon@usc.edu, bonahonf@msu.edu}
\author{Ya\c sar S\"ozen}
\address{Hacettepe University, 06800 Ankara, T\"urk\.{\i}ye}
\email{ysozen@hacettepe.edu.tr}
\author{Hat\.{\i}ce Zeybek}
\address{Hacettepe University, 06800 Ankara, T\"urk\.{\i}ye}
\email{haticekayazeybek@gmail.com}

\date{\today}

\thanks{This work was partially supported by  the National Science Foundation of the United States (grant DMS-2005656), and by the Scientific and Technological Research Council of T\"urkiye, T\"UB\.{I}TAK (B\.IDEB-2221-(1059B211402137) and  B\.{I}DEB-2214/A)}

\begin{document}

\maketitle

\begin{abstract}

The $\PGL$--Hitchin component $\Hit(S)$ of a closed oriented surface is a preferred component of the character variety consisting of homomorphisms from the fundamental group of the surface to the projective linear group $\PGL$. It admits a symplectic structure, defined by the Atiyah-Bott-Goldman symplectic form. The main result of the article is an explicit computation of this symplectic form in terms of certain global coordinates for $\Hit(S)$. A remarkable feature of this expression is that its coefficients are constant. 
 
\end{abstract}
\tableofcontents

\section*{Introduction}

For an oriented closed  surface $S$ of negative Euler characteristic, the \emph{$\PGL$--Hitchin component} $\Hit(S)$ is a preferred component of the character variety
$$
\mathcal X_{\PGL}(S) = \{ \rho\colon \pi_1(S) \to \PGL \} \db \PGL
$$
consisting of group homomorphisms $\rho\colon \pi_1(S) \to \PGL$ from the fundamental group $\pi_1(S)$ of the surface to the projective linear group $\PGL$, considered modulo the action of $\PGL$ by conjugation. 

Fundamental results of Hitchin \cite{Hit}, Labourie \cite{Lab} and Fock-Goncharov \cite{FocGon1}, as well as subsequent work by many researchers triggered by these breakthroughs, have showed that the \emph{Hitchin representations}, namely the homomorphisms $\rho\colon \pi_1(S) \to \PGL$ representing elements of $\Hit(S)$, have very strong geometric properties. The general theme is that these properties are very similar to the ones observed in the case $n=2$ where, for the identification between $\mathrm{PGL}_2(\R)$ and the isometry group of the hyperbolic plane $\mathbb H^2$, Hitchin representations correspond to the monodromies of hyperbolic metrics on the surface $S$.  For instance, Hitchin representations are injective with discrete image, the action of the mapping class group on $\Hit(S)$ is discontinuous, and the Hitchin component $\Hit(S)$ is smooth and diffeomorphic to $\R^{-\chi(S)(n^2-1)}$, where $\chi(S)<0$ is the Euler characteristic of the surface. 

The Hitchin diffeomorphism between $\Hit(S)$ and $\R^{-\chi(S)(n^2-1)}$ was based on the theory of Higgs bundles. A different  parametrization, combining the dynamical properties uncovered in \cite{Lab} with the representation theoretic arguments of \cite{FocGon1}, was introduced in \cite{BonDre2}. 
The Hitchin component $\Hit(S)$ and, more generally, the smooth part of $\mathcal X_{\PGL}(S) $, comes with a preferred symplectic structure, the \emph{Atiyah-Bott-Goldman symplectic form} \cite{AtiBot, Gol}. The main result of the current article is an explicit computation of this symplectic form in terms of the global coordinates for $\Hit(S)$ developed in \cite{BonDre2}.

The coordinates of  \cite{BonDre2} were inspired by earlier work of Fock and Goncharov \cite{FocGon1}  for surfaces with at least one puncture. For such a punctured surface $S$, Fock and Goncharov developed  a theory of positive framed representations $\rho\colon \pi_1(S) \to \PGL$  that is parallel to that of Hitchin representations. In particular, they  proved that the space $\mathrm P_n(S)$ of positive framed characters is also diffeomorphic to $\R^{-\chi(S)(n^2-1)}$, by constructing explicit global coordinates for this space, associated to an ideal triangulation of the surface $S$. Because of the punctures, the space $\mathrm P_n(S)$ does not admit a symplectic structure, but it is endowed with a Poisson structure and Fock and Goncharov explicitly compute this Poisson structure in their coordinates. 

\begin{figure}[htbp]

\includegraphics{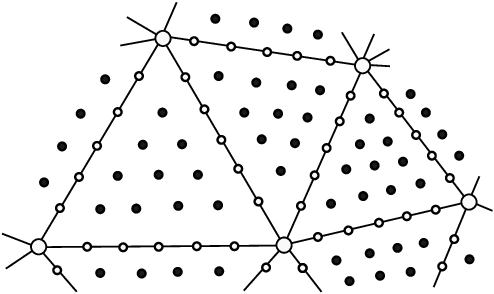}

\caption{The dots indexing the Fock-Goncharov coordinates, for $n=6$.}
\label{fig:FockGonDots}
\end{figure}

Before describing the generalized Fock-Goncharov coordinates of \cite{BonDre2} it may be useful, for perspective, to remind the reader of the original Fock-Goncharov coordinates of \cite{FocGon1} for punctured surfaces. Incidentally, we are here considering the so-called $X$--coordinates for positive framed $\PGL$--characters, as opposed to the $A$--coordinates of the dual $\SL$--based theory; see \cite{FocGon1}. Let $S$ be a punctured surface of negative Euler characteristic, obtained  by removing a nonempty set of punctures $\{p_1, p_2, \dots, p_k\}$ from a closed oriented surface $\bar S$. The Fock-Goncharov coordinates for the space $\mathrm P_n(S)$ of positive framed characters use an \emph{ideal triangulation} of the surface $S$, namely a triangulation of the closed surface $\bar S$ whose vertices are exactly the punctures $p_i$.  The Fock-Goncharov coordinates are then indexed by a certain number of dots drawn on the surface. More precisely, there are $n-1$ hollow (white) dots on each edge of the ideal triangulation, and the interior of each face $T$ is endowed with $\frac12(n-1)(n-2)$ solid (black) dots arranged as an embedded copy of the \emph{discrete triangle}
$$
\Theta_n = \{ (a,b,c)\in\Z^3; a,b,c\geq1 \text{ and } a+b+c=n\}, 
$$
in such a way that the vertices of the discrete triangle point towards the corners of the triangle $T$. See Figure~\ref{fig:FockGonDots}. An elementary Euler characteristic argument shows that there is a total of $-\chi(S)(n^2-1)$ such solid and hollow dots, a number that is also the dimension of  the space $\mathrm P_n(S)$ of positive framed characters. 

Given  a positive framed character $[\rho] \in\mathrm P_n(S)$ , Fock and Goncharov associate a \emph{triangle invariant} $\tau_\delta(\rho)\in \R$ to each solid dot $\delta$ in a face of the triangulation, and a \emph{shearing invariant} $\sigma_\delta(\rho)\in \R$ to each hollow dot $\delta$ in an edge. To be completely accurate, the original Fock-Goncharov $X$--coordinates are positive numbers, and we are taking here their logarithms. Fock and Goncharov then show that these invariants provide a diffeomorphism from  the space $\mathrm P_n(S)$ of positive framed characters to $\R^{-\chi(S)(n^2-1)}$.

\begin{figure}[htbp]

\includegraphics{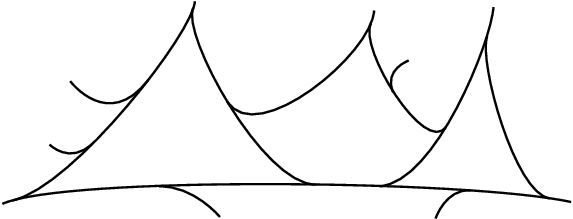}

\caption{A train track.}
\label{fig:TrainTrack}
\end{figure}

A closed surface has no ideal triangulation, since there is no puncture where to put the vertices. The parametrization developed in \cite{BonDre2} replaces the ideal triangulation by the data of  a maximal geodesic lamination $\lambda$ and of a train track carrying $\lambda$; see \S \ref{subsect:GeodLamTrainTracks} for precise definitions. A geodesic lamination usually has a very intricate dynamical structure, but a train track is a very combinatorial object. More precisely, a  \emph{train track} $\Psi$ in the surface $S$ is a trivalent graph embedded in $S$, whose edges are smoothly embedded and such that, at each vertex, the three adjacent edges are tangent to a single line with two edges going in one direction and the remaining edge in the other direction. See Figure~\ref{fig:TrainTrack} for an example. This property implies that the completion of the complement $S-\Psi$ has sharp corners occurring at the vertices of $\Psi$ and, because the geodesic lamination is maximal, there is the additional property that each component of $S-\Psi$ is a triangle, namely a disk with 3 corners in its boundary.

In accordance with the american railroad terminology, the vertices of the train track $\Psi$ are called its \emph{switches}, and its edges are its \emph{branches}. 

As in the original framework of \cite{FocGon1}, the generalized Fock-Goncharov coordinates of \cite{BonDre2} are associated to dots drawn on the surface, except that there are many more of them. More precisely, there are $n-1$ hollow dots on each branch of the train track $\Psi$, and $\frac12(n-1)(n-2)$ solid dots in each triangle component $U$ of the complement $S-\Psi$, arranged as a copy of the discrete triangle $\Theta_n$ in such a way that the vertices of $\Theta_n$ point toward the corners of $U$. In particular, there  is a total of $-\chi(S)(n-1)(n+7)$ such dots. See Figure~\ref{fig:BonDreDots} for an illustration. A Hitchin character $[\rho] \in \Hit(S)$ then determines a \emph{triangle invariant} $\tau_\delta(\rho) \in \R$ associated to each solid dot $\delta$ in the complement $S-\Psi$, and a \emph{shearing invariant} $\sigma_\delta(\rho)\in \R$ associated to each hollow dot on a branch of $\Psi$. We should probably emphasize that these invariants depend, not just on the train track $\Psi$, but also on the maximal geodesic lamination $\lambda$ carried by $\Psi$.

\begin{figure}[htbp]

\includegraphics{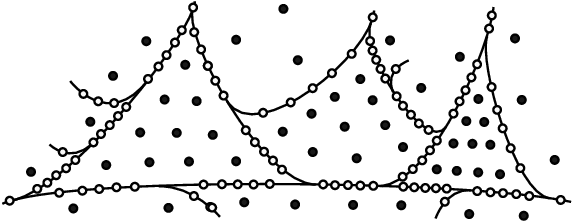}

\caption{The dots indexing the generalized Fock-Goncharov coordinates, for $n=6$.}
\label{fig:BonDreDots}
\end{figure}

The total number $- \chi(S)(n-1)(n+7)$ of dots  is significantly more than the dimension $- \chi(S)(n^2-1)$ of $\Hit(S)$. Indeed, whereas the original Fock-Goncharov coordinates were independent,  the generalized Fock-Goncharov coordinates satisfy a certain number of linear relations associated to the switches of the train track $\Psi$. More precisely, by our definition of train tracks, a switch $s$ of $\Psi$ has two incoming branches $e_s^\Left$ and $e_s^\Right$ and a single outgoing branch $e_s^\Out$, with $e_s^\Left$ to the left of $e_s^\Right$ when oriented toward $s$. We orient $e_s^\Out$ away from $s$. The switch also determines a component $U_s$ of $S-\Psi$, located between $e_s^\Left$ and $e_s^\Right$ near $s$.  Then, for $a=1$, $2$, \dots, $n-1$, the \emph{$a$--th switch relation} associated to $s$ states that
$$
\sigma_{\delta_s^\Out(a)}(\rho) = \sigma_{\delta_s^\Left(a)}(\rho) + \sigma_{\delta_s^\Right(a)}(\rho) - \sum_\delta \tau_\delta(\rho)
$$
where $\delta_s^\Out(a)$, $\delta_s^\Left(a)$ and $\delta_s^\Right(a)$ are the $a$--th dots in the branches $e_s^\Out$, $e_s^\Left$ and $e_s^\Right$ for the orientations specified above, and where the sum is over the $n-a-1$ solid dots $\delta$ in the component $U_s$ of $S-\Psi$ that, in the corresponding discrete triangle, are in the $a$--th row parallel to the side of this discrete triangle that is opposite  the vertex facing $s$.  See Figure~\ref{fig:SwitchRelation} for a graphical illustration  of this switch relation, and \S \ref{subsect:GenFocGonInvariants} for a more formal description. 

\begin{figure}[htbp]

\SetLabels\small
( .32 * .57 ) $\delta_s^\Out(a)$  \\
\R(  .72* .8 )  $\delta_s^\Right(a)$ \\
\R( .72 *  .13)  $\delta_s^\Left(a)$ \\
( .88 * .23) $\delta$\\
( .88 * .4) $\delta$\\
( .88 * .56) $\delta$\\
( .88 * .76) $\delta$\\
\endSetLabels
\centerline{\AffixLabels{\includegraphics{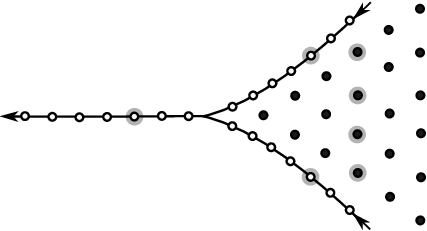}}}

\caption{The $a$--th switch relation for $n=8$ and $a=3$}
\label{fig:SwitchRelation}
\end{figure}
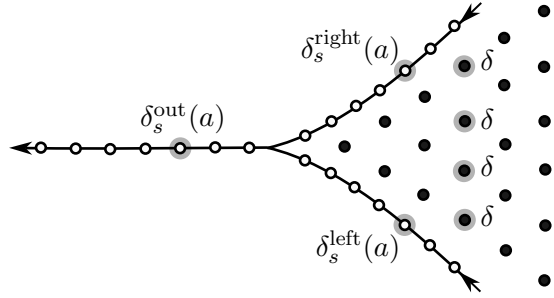

The switch relations constrain the generalized Fock-Goncharov invariants $\tau_\delta(\rho)$ and $\sigma_\delta(\rho)$ to the linear subspace $L$ of $\R^{-(n-1)(n+7) \chi(S)}$ defined by these relations. In addition, these invariants must satisfy a certain \emph{positivity condition} defined by finitely many strict linear inequalities, involving the transverse measures for  the geodesic lamination $\lambda$ and consequently based on the dynamical structure of the lamination. 

The main result of \cite{BonDre2} is that the map assigning its generalized Fock-Goncharov invariants to a Hitchin character $[\rho] \in \Hit(S)$ defines a diffeomorphism from $\Hit(S)$ to the polytope in $\R^{-(n-1)(n+7) \chi(S)}$ defined by the switch relations and the positivity condition. 

Going back to the Atiyah-Bott-Goldman symplectic form $\omega$, each dot $\delta$ defines a function $\tau_\delta \colon [\rho] \mapsto \tau_\delta(\rho)$ or $\sigma_\delta \colon [\rho] \mapsto \sigma_\delta(\rho)$ on the Hitchin component $\Hit(S)$. Our main result provides an explicit computation of $\omega$ in terms of the differentials $d\tau_\delta$ and $d\sigma_\delta$ of these coordinate maps.

\begin{thm}
\label{thm:MainThm}

The Atiyah-Bott-Goldman symplectic form $\omega$ of the Hitchin component $\Hit(S)$ can be expressed in terms of the generalized Fock-Goncharov invariants as 
\begin{align*}
 \omega &=
 \sum_{U \text{ component of } S-\Psi}\ \sum_{\delta, \, \delta' \in U} \tfrac12 C(\delta,\delta') \,d\tau_\delta \wedge d\tau_{\delta'}
 \\
&\qquad
+ \sum_{s \text{ switch of }\Psi}\ \sum_{\delta \in e_s^\Left \!, \,\delta'\in e_s^\Right} C(\delta,\delta') \,d\sigma_\delta \wedge d\sigma_{\delta'}
 \\
&\qquad
- \sum_{s \text{ switch of }\Psi}\ \sum_{\delta \in e_s^\Left\!, \, \delta'\in U_s} C(\delta,\delta') \,d\sigma_\delta \wedge d\tau_{\delta'}
 \\
&\qquad
+ \sum_{s \text{ switch of }\Psi}\ \sum_{\delta \in e_s^\Right\!,  \,\delta'\in U_s} C(\delta,\delta') \,d\sigma_\delta \wedge d\tau_{\delta'}
\end{align*}
where the coefficients $C(\delta,\delta')$ are explicit integers associated to the dots $\delta$ and $\delta'$. In particular, these coefficients are constant. 
\end{thm}

More precisely, 
$$
C(\delta,\delta')=
\begin{cases}
a'b-ab' &\text{if } a\leq a', b\leq b' \text{ and } c\geq c'\\
ac'-a'c &\text{if } a\leq a', b\geq b' \text{ and } c\leq c'\\
b'c-bc' &\text{if } a\leq a', b\geq b' \text{ and } c\geq c'\\
b'c-bc' &\text{if } a\geq a', b\leq b' \text{ and } c\leq c'\\
ac'-a'c &\text{if } a\geq a', b\leq b' \text{ and } c\geq c'\\
a'b -ab'&\text{if } a\geq a', b\geq b' \text{ and } c\leq c'
\end{cases}
$$
if the solid dots $\delta$, $\delta'\in U$ are respectively associated to the points $(a,b,c)$, $(a', b', c')$ of the discrete triangle $\Theta_n$, and
$$
C(\delta,\delta')=
\begin{cases}
a'(n-a)  &\text{if } a\geq a'
\\
 a(n-a') &\text{if } a\leq a',
\end{cases}
$$
if: $\delta$ is the $a$--th hollow dot in $e_s^\Left$ and  $\delta'$ is the $a'$--th hollow dot in  $e_s^\Right$; or  $\delta$ is the $a$--th hollow dot in $e_s^\Left$ and $\delta'\in U_s$ is a solid dot which, in the discrete triangle of $U_s$, is in the $a'$--th row parallel to the side that is opposite $s$; or  $\delta$ is the $a$--th hollow dot in $e_s^\Right$ and $\delta'\in U_s$ is a solid dot which, in the discrete triangle of $U_s$, is in the $a'$--th row parallel to the side  opposite $s$. 

As an illustration, the formula is relatively simple when $n=3$. In this case, let $\delta_s$ be the unique solid 
dot contained in the component $U_s$ of $S-\Psi$ associated to the switch $s$. Also, let  $\delta_s^\Left(a)$ and $\delta_s^\Right(a)$ be the $a$--th dots in the incoming branches $e_s^\Left$ and $e_s^\Right$ of $s$, as in our statement of the Switch Condition and in Figure~\ref{fig:SwitchRelation}. 

\begin{cor}
 When $n=3$, the symplectic form of the Hitchin component $\mathrm{Hit}_3(S)$ is equal to
 \pushQED{\qed}
\begin{align*}
\omega
&=  \sum_{s \text{ switch of }\Psi} 2\left( 
d\sigma_{\delta_s^\Left(1)} \wedge  d\sigma_{\delta_s^\Right(1)} 
+d\sigma_{\delta_s^\Left(2)} \wedge  d\sigma_{\delta_s^\Right(2)}
\right)
\\
&\qquad+  \sum_{s \text{ switch of }\Psi} \left( 
d\sigma_{\delta_s^\Left(1)} \wedge  d\sigma_{\delta_s^\Right(2)}
+d\sigma_{\delta_s^\Left(2)} \wedge  d\sigma_{\delta_s^\Right(1)}
\right)
\\
&\qquad+  \sum_{s \text{ switch of }\Psi} \left( 
2 \sigma_{\delta_s^\Right(1)} +\sigma_{\delta_s^\Right(2)} - 2\sigma_{\delta_s^\Left(1)} - \sigma_{\delta_s^\Left(2)}  
\right) \wedge d\tau_{\delta_s}.
\qedhere
\end{align*}
\end{cor}

\begin{rem}
 We have chosen to restrict attention to trivalent train tracks for the sake of the exposition. The reader familiar with this concept should have no problem extending Theorem~\ref{thm:MainThm} (and the definition of the switch relations) to  general train tracks carrying the maximal geodesic lamination $\lambda$, where some switches are allowed to  be adjacent to more than three branches. This would result in a significant increase in notational complexity for the statements, but tends to be more convenient for explicit computations. 
\end{rem}

\begin{rem}
 The article \cite{BonDre1}, intended as an easier version of \cite{BonDre2} for the case when the geodesic lamination $\lambda$ has finitely many leaves, develops a set of coordinates for the Hitchin component $\Hit(S)$ that looks very much like that of \cite{BonDre2}. In fact, these two sets of coordinates coincide for the most part, but they differ very substantially in the way they measure the shear along a closed leaves of $\lambda$. The outcome is that, as already observed in \cite{SunZha}, the coordinates of \cite{BonDre1} are not well adapted to the symplectic form $\omega$. Indeed, the expression of $\omega$ in these coordinates involves coefficients that are not constant and depend on the point of $\Hit(S)$ considered. 
\end{rem}

Theorem~\ref{thm:MainThm} is proved as Theorem~\ref{thm:MainThm2} in \S \ref{bigsect:ProofMainThm}, using a notation scheme that is more formula-oriented and less graphical than the one used in this Introduction (see \S \ref{subsect:GenFocGonInvariants} for the correspondence). 
The Atiyah-Bott-Goldman form is defined \cite{Gol} as a cup-product, using the Weil identification of the tangent space $T_{[\rho]} \Hit(S)$ with the cohomology space $H^1(S; \sln_{\Ad \rho})$. The main technical part of the article, requiring subtle analytic estimates to guarantee convergence, is concentrated in \S\S \ref{bigsect:VariationFlagMap}-\ref{bigsect:CohomologyClassTgtVector} and is devoted to the explicit construction of a simplicial cocycle $c_V \in C^1(S; \sln_{\Ad \rho})$ representing the cohomology class associated to a tangent vector $V \in T_{[\rho]} \Hit(S)$. This part involves a reconstruction of a Hitchin representation from its generalized Fock-Goncharov invariants, which is more explicit than the one used in \cite{BonDre2} and may be of independent interest. After this, relatively straightforward combinatorial and algebraic computations in \S \ref{bigsect:ComputeEstimateCupProduct} provide an estimate for the value of the cup-product. In a final step, the estimate is shown to be invariant under the operation of zipper-opening for train tracks described in \S \ref{subsect:OpeningZippers}, and is therefore exactly equal to the value of the cup-product.

Theorem~\ref{thm:MainThm} is the natural extension of the results of \cite{SozBon} for the case where $n=2$, and of \cite{Zey} for the restriction of $\omega$ to the slices of $\Hit(S)$ where the triangle invariants $\tau_\rho(d)$ are constant. A  version of Theorem~\ref{thm:MainThm} was obtained by Zhe Sun and Tengren Zhang \cite{SunZha} for the case  where the geodesic lamination is restricted to have only finitely many leaves and where the convergence arguments are easier. We actually rely on several tools developed by these authors, in particular the eruptions introduced in \cite{SunWieZha} (see \S \ref{subsect:Eruptions}),  the algebraic computation of Lemma~\ref{lem:ComputeKillingFormLeftRight}, and the barriers of \S\ref{subsect:Barriers} which enabled us to simplify an earlier approach of ours.

\section{ The $\PGL$--Hitchin component}
\label{bigsect:HitchinComp}

\subsection{Hitchin characters and Hitchin representations}
For a closed  surface $S$ of negative Euler characteristic, the \emph{$\PGL$--Hitchin component} $\Hit(S)$ is a preferred component of the character variety
$$
\mathcal X_{\PGL}(S) = \{ \rho\colon \pi_1(S) \to \PGL \} \db \PGL
$$
consisting of group homomorphisms $\rho\colon \pi_1(S) \to \PGL$ from the fundamental group $\pi_1(S)$ of the surface to the projective linear group $\PGL$, considered modulo the action of $\PGL$ by conjugation. More precisely, the monodromy $ \pi_1(S) \to \mathrm{PSL}_2(\R)$ of a hyperbolic metric on $S$ can be extended to a homomorphism  $\rho\colon \pi_1(S) \to \PGL$ by composing it with the unique $n$--dimensional irreducible representation $\mathrm{SL}_2(\R) \to \mathrm{SL}_n(\R)$ of $\mathrm{SL}_2(\R)$. The Hitchin component is the component of $\mathcal X_{\PGL}(S)$ that contains the characters thus associated to hyperbolic metrics. 

We will refer to elements of $\Hit(S)$ as \emph{Hitchin characters}, while homomorphisms $\rho\colon \pi_1(S) \to \PGL$ representing a Hitchin character $[\rho] \in \Hit(S)$ will be \emph{Hitchin representations}. 
As indicated in the Introduction, Hitchin representations have very strong geometric properties. 
We will particularly rely on two such features. 

One is a dynamical property of Labourie \cite{Lab} which provides, for each Hitchin representation $\rho\colon \pi_1(S) \to \PGL$, a continuous $\rho$--equivariant \emph{flag map} $\F_\rho \colon \bdry \to \Flag$ from the boundary at infinity $\bdry$ of the universal cover $\wt S$ to the space $\Flag$ of flags in $\R^n$. This flag map is actually the unique continuous $\rho$--equivariant map $ \bdry \to \Flag$. 

The second property, due to Fock-Goncharov \cite{FocGon1}, is more Lie theoretic and asserts that the flag map $\F_\rho$ sends any finite family of distinct points $x_1$, $x_2$, \dots, $x_k \in\bdry$, occurring in this order in the circle $\bdry$, to a family of flags $\F_\rho(x_1)$, $\F_\rho(x_2)$, \dots, $\F_\rho(x_k) \in \Flag$ that is positive, in a sense that we will discuss in a few more details in \S \ref{subsect:TripleDoubleRatios}.  

\subsection{Geodesic laminations and train tracks}
\label{subsect:GeodLamTrainTracks}

The parametrization of $\Hit(S)$ developed in \cite{BonDre2} uses maximal geodesic laminations as  a substitute for the ideal triangulations used by Fock and Goncharov \cite{FocGon1} for punctured surfaces, since closed surfaces admit no ideal triangulation. The next section \S \ref{subsect:GenFocGonInvariants} will give a  brief description of the coordinates of  \cite{BonDre2}, but we first give some background on geodesic laminations. We refer to, for instance, \cite{ThuNotes, CanEpsGre, CasBle, PenHar} for details about these facts. 

Since we are restricting attention to surfaces of negative Euler characteristic, we can endow the closed surface $S$ with an arbitrary riemannian metric $m_0$ of negative curvature. A \emph{geodesic lamination} (or \emph{$m_0$--geodesic lamination} if we want to emphasize the dependence on the auxiliary metric $m_0$) is a closed subset $\lambda \subset S$ that can be decomposed as the union of a family of disjoint simple geodesics, where a simple geodesic is a complete $m_0$--geodesic that does not intersect itself (but may be closed). These geodesics are the \emph{leaves} of the geodesic lamination. 

A simple example is provided by a family of finitely many disjoint simple closed geodesics. This example can be expanded by adding a few infinite geodesics that spiral around the closed geodesics. However, a typical geodesic lamination is usually much more complicated, and can locally consist of a Cantor set of disjoint geodesics. This is often the case with those geodesic laminations that arise from a geometric context, such as the stable and unstable measured geodesic laminations of a pseudo-Anosov diffeomorphism, or the bending lamination of a kleinian group. 

Although the definition given here depends on the choice of a negatively curved metric $m_0$ on the surface $S$, geodesic laminations can actually be described in an intrinsic way that make them independent of this choice, for instance in terms of the boundary at infinity $\bdry$ of the universal cover $\wt S$. As such, geodesic lamination  intrinsically are topological objects associated to the surface. 

A geodesic lamination is \emph{maximal} if it is maximal for inclusion among all geodesic laminations. 

\begin{prop}
$ $
\begin{enumerate}
 \item Every geodesic lamination is contained in a maximal geodesic lamination. 
 \item A geodesic lamination $\lambda$ is maximal if and only if its complement $S-\lambda$ consists of finitely many ideal triangles, each bounded by three infinite geodesics in $\lambda$. In addition, the number of these triangles is equal to $-2\chi(S)$, where $\chi(S)$ is the Euler characteristic of $S$. \qed
\end{enumerate}
\end{prop}

In practice, exactly as a Cantor set looks like a finite set at finite resolution, a geodesic lamination drawn on a surface looks like a  train track as defined in the Introduction. This can be formalized as follows. 

A (trivalent) \emph{train track neighborhood} for the $m_0$--geodesic lamination $\lambda$ is a closed neighborhood $\Phi$ of $\lambda$ which can be decomposed as a union of finitely many embedded rectangles $R_i$, called its \emph{branches}, such that
\begin{enumerate}
\item  each rectangle $R_i \cong [0,1] \times [0,1]$ is foliated by arcs $\{x \} \times [0,1]$, called its \emph{ties}, and the leaves of $\lambda$ are transverse to these ties; 

\item  the points of a branch $R_i \cong [0,1] \times [0,1]$ that belong to  another branch $R_j$ are exactly the points of the boundary ties $\{0,1\} \times [0,1]$;

\item the points of $S$ locally belonging to at least two branches form a family of disjoint arcs, called the \emph{switch ties} of the train track neighborhood;

\item each switch tie is locally adjacent to one branch $R_i$ on one side, and to exactly two branches $R_j$, $R_k$ on the other side.
\end{enumerate}

It follows from this general description that the closure of the complement $S-\Phi$ has corners, corresponding to points where three distinct branches meet. We add one more global condition. 

\begin{enumerate}
\setcounter{enumi}{4}
\item no component of $S-\Phi$ is a disk with 0, 1 or 2 corners in its boundary. 
\end{enumerate}

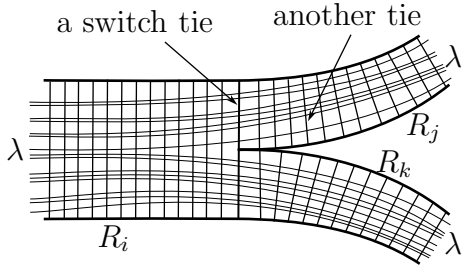
\begin{figure}[htbp]
\vskip 15pt

\SetLabels
( 1 * .85 )  $\lambda$ \\
( 1 *  .05)  $\lambda$ \\
\R( - 0.01 * .45  )  $\lambda$ \\
( .25 * 1 )  a switch tie \\
( .75 * 1.03 ) another tie  \\
( .2 * .08 )  $R_i$ \\
( .93 * .56 )  $R_j$ \\
( .86 * .4 )  $R_k$ \\
\endSetLabels
\centerline{\AffixLabels{\includegraphics{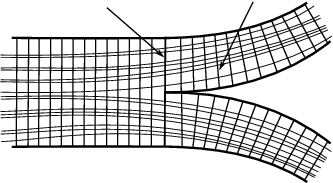}}}

\caption{The local model of a train track neighborhood near a switch tie}
\label{fig:TrainTrackLocal}
\end{figure}

A fundamental property is that every geodesic lamination $\lambda$ admits arbitrarily small train track neighborhoods. 

\begin{prop}
\label{prop:ComplementsTrainTrackGeodLamination}

Let $U$ be a component of the complement $S-\Phi$ of a train track neighborhood $\Phi$ of the maximal geodesic lamination $\lambda$, and let $T$ be the component of $S-\lambda$ that contains it. Then, the foliation of $\Phi\cap T$ induced by the ties of $\Phi$ is of the type illustrated in Figure~{\upshape\ref{fig:TriangleAndTrainTrack}}. In particular, $U$ is a triangle, namely a disk with three corners in its boundary, and each corner of $U$ points toward a unique vertex of the ideal triangle $T$.

This establishes a one-to-one correspondence between the components $U$ of  $S-\Phi$ and the components $T$ of $S-\lambda$, as well as a one-to-one correspondence between the corners of the boundary of $U$ and the vertices of the ideal triangle $T$. 
\qed
 \end{prop}

 \begin{figure}[htbp]
 
 \SetLabels
( .92 *  .7) $T$  \\
\E( .65 * .5 )  $U$ \\
(  *  )   \\
\endSetLabels
\centerline{\AffixLabels{\includegraphics{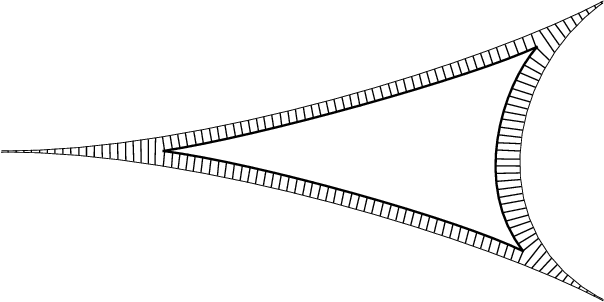}}}

\caption{The train track neighborhood $\Phi$ and a component $T$ of $S-\lambda$}
\label{fig:TriangleAndTrainTrack}
\end{figure}

From a train track neighborhood $\Phi$ of the $m_0$--geodesic lamination $\lambda$, one easily obtains a train track $\Psi$ in the sense of the Introduction, by collapsing each tie of $\Phi$ to a single point. In this situation, we will say that $\lambda$ is \emph{(strongly) carried} by the train track $\Psi$. 

We will often use the same letter to denote a branch $e$ of the train track neighborhood $\Phi$ and the corresponding branch $e$ of the train track $\Psi$ associated to $\Phi$, or for the switch tie $s$ of $\Phi$ and the corresponding switch $s$ of $\Psi$.

\begin{rem}
The above definitions heavily depend on the auxiliary negatively curved metric $m_0$ chosen to describe geodesic laminations. There is a more topological point of view, independent of the choice of the metric $m_0$, in which a geodesic lamination is \emph{weakly carried} by the train track $\Psi$ if, for every leaf $g$ of $\lambda$, there is a train route in $\Psi$ that is homotopic to $g$ by a homotopy moving points by a uniformly bounded amount (the Fellow Traveller Property). All the results of the current article (as well as those of \cite{BonDre2}) could be rephrased in this more topological framework, but at the expense of increased complexity in the exposition. We have chosen to restrict attention to the more restricted (metric dependent) viewpoint,  but experts should have no problem extending our arguments to the more general setup. 
\end{rem}

\subsection{Generalized Fock-Goncharov invariants}
\label{subsect:GenFocGonInvariants}

Let $\lambda$ be  a maximal geodesic lamination in the surface $S$ that is carried by a  trivalent train track $\Psi$, in the sense of \S \ref{subsect:GeodLamTrainTracks}. By Proposition~\ref{prop:ComplementsTrainTrackGeodLamination}, each component of $S-\Psi$ is a triangle, namely a disk with three boundary corners (the vertices of the triangle) located at switches of $\Psi$.

The  parametrization of $\Hit(S)$  associated to $\lambda$ and $\Psi$ that is developed in \cite{BonDre2}  is based on two distinct types of invariants for a Hitchin representation $\rho \colon \pi_1(S) \to \PGL$. 

The first type are the \emph{triangle invariants} $\tau^{abc}_\rho (U,s) \in \R$. These are indexed by the components $U$ of the complement $S-\Psi$, the vertices $s$ of the triangle $U$, and the points $(a,b,c)$ in the \emph{discrete triangle}
$$
\Theta_n = \{ (a,b,c)\in\Z^3; a,b,c\geq1 \text{ and } a+b+c=n\}. 
$$
These triangle invariants satisfy the \emph{Rotation Condition} that, if $s_1$, $s_2$, $s_3$ are the three vertices of the triangle $U$, indexed counterclockwise in this order around $U$, 
$$
\tau_\rho^{abc}(U, s_1) = \tau_\rho^{bca}(U, s_2) = \tau_\rho^{cab}(U, s_3). 
$$

The second type of invariant consists of the \emph{shearing invariants} $\sigma_\rho^{ab}(e)\in \R$, indexed by the oriented branches $e$ of the train track $\Psi$ and by the points $(a,b)$ in the \emph{discrete interval}
$$
\mathrm I_n = \{ (a,b) \in \Z^2; a,b \geq1 \text{ and } a+b=n\}. 
$$
These are constrained by the \emph{Orientation Reversal Relation} that, if $\overline e$ is obtained by reversing the orientation of the branch $e$, 
$$
\sigma_\rho^{ab}(\overline e) = \sigma_\rho^{ba}( e) . 
$$

There are two additional constraints for the invariants $\tau^{abc}_\rho (U,s)$ and $\sigma_\rho^{ab}(e)$. The main one consists of  the \emph{Switch Relations} for the shearing invariants $\sigma_\rho^{ab}$, defined as follows. By our assumption that the train track $\Psi$ is trivalent, there are three branches $e_s^\Out$, $e_s^\Left$, $e_s^\Right$ of $\Psi$ that are adjacent to each switch $s$, with $e_s^\Out$ on one side of $s$ while $e_s^\Left$ and $e_s^\Right$ are on the other side. We orient $e_s^\Out$ away from $s$ and $e_s^\Left$ and $e_s^\Right$ toward $s$, with $e_s^\Left$ coming in on the left of $e_s^\Right$ (hence the names). See Figure~\ref{fig:SwitchRelation}. Finally, let $U_s$ be the triangle component of $S-\Psi$ that has $s$ as a vertex. The Switch Condition then asserts that 
$$
\sigma_\rho^{ab} (e_s^\Out) = \sigma_\rho^{ab} (e_s^\Left) + \sigma_\rho^{ab} (e_s^\Right) - \sum_{b'+c'=n-a} \tau_\rho^{ab'c'}(U_s,s)
$$
where the sum is over all $b'$, $c'$ such that $(a, b', c')$ is in the discrete triangle $\Theta_n$. 

There is an additional \emph{Positivity Condition} satisfied by the invariants, which we alluded to in the Introduction. However, it is here irrelevant  as it is an open condition and we are only interested in local properties. 

\begin{rem}
 The shearing invariant $\sigma_\rho^{ab}(e)$ is called $\sigma^\rho_{a}(e)$ in \cite{BonDre2}, with no loss of information since $b=n-a$. 
\end{rem}

The reader should have no difficulty connecting these invariants $\tau^{abc}_\rho (U,s)$ and $\sigma_\rho^{ab}(e)$ to the invariants $\tau_\rho (\delta)$ and $\sigma_\delta(\rho)$ of the Introduction. Namely, if we embed the discrete triangle $\Theta_n$ in the component $U$ of $S-\Psi$ in such a way that the vertices $(n-2, 1, 1)$, $(1, n-2, 1)$, $(1,1,n-2)$ of $\Theta_n$ respectively point toward corners $s_1$, $s_2$, $s_3$ of $U$ occurring counterclockwise in this order around $U$, then $\tau_\rho^{abc}(U,s_1)$ is equal to  $\tau_\delta(\rho)$ for the solid dot $\delta \in \Theta_n \subset U$ corresponding to $(a,b,c)$. Similarly, if $e$ is an oriented branch of the train track $\Psi$ and if $\delta$ is the hollow dot drawn on $e$ that occurs in $a$--th position for the orientation of $e$, the shearing invariant $\sigma_\delta(\rho)$ is equal to $\sigma_\rho^{a(n-a)}(e)$. The Rotation Condition and the Orientation Reversal Condition above guarantee the consistency of these identifications. Then, the Switch Relations of this section are easily seen to be equivalent to those described in the Introduction.

\section{The symplectic structure of the Hitchin component $\Hit(S)$}
\label{bigsect:SymplecticStructure}

\subsection{The Atiyah-Bott-Goldman form}
\label{subsect:Atiyah-Bott-Goldman}
The starting point in the definition of the Atiyah-Bott-Goldman form is  Weil's cohomological interpretation of tangent vectors to character varieties, pioneered in \cite{Wei}. In our context, this correspondence is the following. 

Let $V\in T_{[\rho]} \Hit(S)$ be a tangent vector of the Hitchin component at the point $[\rho]\in \Hit(S)$. Realize it as the tangent vector 
$$
V= \frac{d}{dt} [\rho_t]_{|t=0}
$$ 
for a curve $t\mapsto [\rho_t] \in \Hit(S)$ with $[\rho_0] = [\rho]$, namely for a family of group homomorphisms $\rho_t \colon \pi_1(S) \to \PGL $ depending smoothly on a real parameter $t$ and such that $\rho_0=\rho$. 

For every $\gamma \in \pi_1(S)$, the derivative
$$
c_V(\gamma) = \frac{d}{dt} \rho_t(\gamma)\rho_0(\gamma)^{-1} {}_{|t=0}
$$
is an element of the Lie algebra $\mathfrak{pgl}_n(\R)$ of $\PGL$. We will identify $\mathfrak{pgl}_n(\R)$ to the more familiar Lie algebra $\sln$ of $\SL$, where computations will also be more explicit. 

It can then be shown that the map $c_V \colon \pi_1(S) \to \sln$ so defined is a group cocycle twisted by the adjoint representation $\Ad \rho \colon \pi_1(S) \to \mathrm{GL}\left( \sln \right)$, in the sense that
$$
c_V(\gamma \gamma') = c_V(\gamma) + \Ad\rho(\gamma) \left( c_V(\gamma') \right) =  c_V(\gamma) + \rho(\gamma) c_V(\gamma') \rho(\gamma)^{-1} .
$$

This cocycle $c_V$  defines a group cohomology class $[c_V] \in H^1\left( \pi_1(S); \sln_{\Ad\rho}\right)$ twisted by the adjoint representation. Weil shows (in a more general setup) that this cohomology class $[c_V]$ depends only on the tangent vector $V$, and that this establishes a one-to-one correspondence between the tangent space  $T_{[\rho]} \Hit(S)$ and the cohomology space  $H^1\left( \pi_1(S); \sln_{\Ad\rho}\right)$. 

Since the universal cover of $S$ is contractible, the cohomology spaces $H^1\left( \pi_1(S); \sln_{\Ad\rho}\right)$ and $H^1\left( S; \sln_{\Ad\rho}\right)$ are naturally isomorphic.

Now, two tangent vectors $V_1$, $V_2 \in T_{[\rho]} \Hit(S)$ give two cohomology classes  $[c_{V_1}]$, $[c_{V_2}] \in H^1(S; \sln_{\Ad\rho})$ and, inspired by an earlier construction of Atiyah-Bott \cite{AtiBot} for compact groups, Goldman \cite{Gol} considers their cup-product
$$
[c_{V_1}] \cupp [c_{V_2}] \in H^2\left(S; \left(\sln\otimes \sln \right)_{\Ad\rho}\right).
$$
The Killing form $K \colon \sln\otimes \sln \to \R$ of $\sln$ is invariant under the adjoint representation, and therefore defines a homomorphism $K^* \colon  H^2\left(S; \left(\sln\otimes \sln \right)_{\Ad\rho}\right) \to H^2\left(S; \R \right) $ valued in the real untwisted cohomology. We can then evaluate the corresponding cohomology class on the fundamental class $[S] \in H_2(S; \R)$, and define
$$
\omega_\rho (V_1, V_2) =\left\langle  K^* \left( [c_{V_1}] \cupp [c_{V_2}] \right), [S] \right\rangle.
$$

This construction provides an antisymmetric bilinear form $\omega_\rho$ on each tangent space $T_{[\rho]} \Hit(S)$, which depends differentiably on the point $[\rho] \in \Hit(S)$. In other words, this defines a differential form $\omega$ of degree 2 on the Hitchin component $\Hit(S)$. 

It follows from Poincar\'e duality  that  this differential form $\omega \in \Omega^2\big( \Hit(S)\big)$ is non-degenerate at every point $[\rho] \in \Hit(S)$, and Goldman shows that it is also closed. As a consequence, $\omega$ is a symplectic form on $\Hit(S)$. 

\subsection{A simplicial description of the Weil cohomology class}
\label{subsect:SimplicialWeil}

The group cohomology description of the Weil class $[c_V] \in H^1\left( \pi_1(S); \sln_{\Ad\rho}\right)= H^1\left( S; \sln_{\Ad\rho}\right)$ associated to a tangent vector $V\in T_{[\rho]}\Hit(S)$ is not very convenient for our purposes. In particular, to compute the cup-product of two such classes, we want a description of these classes in terms of the simplicial cohomology associated to a triangulation of the surface. We now describe a scheme that provides such a description, and will  implement this scheme in  later sections. 

The group $\R^\times$ of nonzero real numbers acts by multiplication on the set of all bases of $\R^n$. Let a \emph{projective basis} be an element of the corresponding quotient, namely a basis of $\R^n$ considered only up to scalar multiplication.

Let a tangent vector $V\in T_{[\rho]}\Hit(S)$ be given, and represented by the tangent vector $\dot\rho = \frac{d}{dt} \rho_t{}_{|t=0}$ of a 1--parameter family of homomorphisms $\rho_t \colon \pi_1(S) \to \PGL$ with $\rho_0=\rho$. Suppose that we are given the additional data, for each point $\wt x$ of  the universal cover $\wt S$, of a projective basis $ \B_t(\wt x)$ such that
\begin{enumerate}
 \item the choice of $\mathcal B_t(\wt x)$ is $\rho_t$--equivariant, in the sense that $ \B_t(\gamma \wt x)= \rho_t(\gamma) \left(  \B_t(\wt x) \right)$ for every $\gamma \in \pi_1(S)$ and every point $\wt x$ of $\wt S$;
 \item for every  $\wt x\in \wt S$, the projective basis $\mathcal B_t(\wt x)$ depends differentiably on $t$.
\end{enumerate}

Note that there is no requirement that the projective basis $\mathcal B_t(\wt x)$ depends continuously on $\wt x$, so that it can just be constructed orbit by orbit for the action of $\pi_1(S)$ on $\wt S$. In practice, we will only need these projective bases $\mathcal B_t(\wt x)$ over finitely many orbits, corresponding to the preimage in $\wt S$ of the vertex set of a triangulation of $S$.

By considering the coordinates of its elements, we can interpret the projective basis $\B_t(\wt x)$ as a projective matrix $\B_t(\wt x)\in \PGL$ and we will use the same notation to represent the projective basis and the projective matrix. Namely, $\B_t(\wt x)\in \PGL$ is the unique projective map sending the standard basis of $\R^n$ to the projective basis $\B_t(\wt x)$. 
 For every $\wt x \in \wt S$, we now consider  the derivative 
 $$
 c^0 (\wt x) =  \frac d{dt} \B_t(\wt x) \B_0(\wt x)^{-1}{}_{|t=0} \in \mathfrak{pgl}_n(\R)=\sln.
 $$
 
 We can interpret $c^0$ as an $\sln$--valued singular 0--cochain $c^0 \in C^0 \big(\widetilde S;\sln \big)$. We then consider its coboundary $c^1 = dc^0 \in C^1\big(\widetilde S; \sln \big)$. Namely, if $\widetilde k$ is a singular 1--simplex in $\widetilde \Sigma$, considered as a path going from $\wt x_-$ to $\wt x_+$, then
 $$
 c^1(\widetilde k) = c^0(\wt x_+) - c^0 (\wt x_-).
 $$

 \begin{lem}
 \label{lem:ProjectiveBasisDefinesCocyle}
 The $\sln$--valued cochain $c^1$ is closed and  $\Ad\rho$--equivariant. 
\end{lem}

\begin{proof} The fact that $c^1$ is closed is an immediate consequence of the fact that $c^1 = dc^0$. 

To prove the $\Ad\rho$--equivariance, let us first compute how $c^0$ behaves under the action of $\pi_1(S)$. For $\wt x \in \wt S$ and  $\gamma \in \pi_1(S)$, 
\begin{align*}
 c^0 ( \gamma \wt x) 
 &= \frac d{dt} \B_t(\gamma \wt x) \B_0(\gamma \wt x)^{-1}{}_{|t=0}
= \frac d{dt} \rho_t(\gamma) \B_t( \wt x) \B_0( \wt x)^{-1} \rho_0(\gamma)^{-1} {}_{|t=0}
  \\
 &= \frac d{dt} \rho_t(\gamma) \B_0( \wt x)  \B_0( \wt x)^{-1} \rho_0(\gamma)^{-1} {}_{|t=0}
 + \frac d{dt} \rho_0(\gamma) \B_t( \wt x)  \B_0( \wt x)^{-1} \rho_0(\gamma)^{-1} {}_{|t=0}  
\\
&= \frac d{dt} \rho_t(\gamma) \rho_0(\gamma)^{-1} {}_{|t=0}
+  \rho_0(\gamma) \left( \frac d{dt}  \B_t( \wt x)  \B_0( \wt x)^{-1} {}_{|t=0} \right) \rho_0(\gamma)^{-1} 
\\
&= \frac d{dt} \rho_t(\gamma) \rho_0(\gamma)^{-1} {}_{|t=0}
+   \Ad\rho(\gamma) \big(c^0(\wt x) \big).
\end{align*}

Therefore, if the 1--simplex  $\widetilde k$ goes from $\wt x_-$ to $\wt x_+$,
\begin{align*}
 c^1(\gamma \widetilde k) &=  c^0( \gamma \wt x_+) - c^0 (\gamma \wt x_-)
 \\
 &=  \Ad\rho(\gamma) \big(c^0(\wt x_+) \big) -  \Ad\rho(\gamma) \big(c^0(\wt x_-) \big)
 \\
 &=  \Ad\rho(\gamma) \big(c^1(\widetilde k) \big).
\end{align*}
This proves that $c^1$ is $\Ad\rho$--equivariant. Note that this is in general not true for $c^0$. 
\end{proof}

As a consequence, $c^1$ is a cocycle in the twisted chain complex $C^\bullet(S; \sln_{\Ad\rho})$ by definition of this  chain complex, and defines a singular cohomology class  $[c^1] \in H^1(S; \sln_{\Ad\rho})$.

\begin{lem}
 \label{lem:ProjectiveBasisDefinesWeilCocyle}
 The singular cohomology class $[c^1] \in H^1(S; \sln_{\Ad\rho})$ is equal to the Weil class $[c_V] \in H^1\left( \pi_1(S); \sln_{\Ad\rho}\right)= H^1\left( S; \sln_{\Ad\rho}\right)$ associated to the same tangent vector $V \in T_{[\rho]} \Hit(S)$. 
\end{lem}

\begin{proof}
To compare $[c^1]$ to the Weil class $[c_V]$, we need to be more explicit with the correspondence between the group cohomology space $H^1 \left( \pi_1(S); \sln_{\Ad \rho} \right)$ and the singular cohomology space $H^1 \left( S; \sln_{\Ad \rho} \right)$. 

The group cohomology of $\pi_1(S)$ is defined as the simplicial cohomology of a complex $K$ whose $n$--dimensional simplices $\sigma_{\gamma_1,\gamma_2, \dots, \gamma_n}$ are indexed by the elements $(\gamma_1,\gamma_2, \dots, \gamma_n)$ of the power $\pi_1(S)^n$, and where the boundary operator is defined by
$$
\partial \sigma_{\gamma_1,\gamma_2, \dots, \gamma_n} =\sigma_{\gamma_2,\gamma_3, \dots, \gamma_n}  + \sum_{i=1}^{n-1} (-1)^i \sigma_{\gamma_1, \dots, \gamma_{i-1}, \gamma_i\gamma_{i+1}, \gamma_{i+2}, \dots, \gamma_n} + (-1)^n \sigma_{\gamma_1,\gamma_2, \dots, \gamma_{n-1}} .
$$
See for instance \cite[II.3]{Brown}. 

The twisted cohomology  $H^* \left( \pi_1(S); \sln_{\Ad \rho} \right)$ is defined as the homology of the chain complex of $\Ad \rho$--equivariant cochains on the universal cover $\wt K$ of $K$. More precisely, the universal cover $\wt K$ comes with a base vertex $\wt \sigma_0$ in the preimage of the unique 0--simplex $\sigma_0$ of $K$, and the 0--skeleton of $\wt K$ is the orbit of $\wt \sigma_0$ under the action of $\pi_1(K)=\pi_1(S)$. Then, the $\Ad \rho$--equivariant  1--cochain $c_V \in C^1(\wt K; \sln)$ representing the Weil class $[c_V] \in H^1 \left( \pi_1(S); \sln_{\Ad \rho} \right)$ is uniquely determined by the property that, if $\wt \sigma_\gamma$ is the lift of $\sigma_\gamma$ starting at the base point $\wt\sigma_0$, 
$$
c_V(\wt\sigma_\gamma) = \frac d{dt} \rho_t(\gamma) \rho_0(\gamma)^{-1} {}_{|t=0}  \in \sln . 
$$
Namely, $c_V(\wt\sigma_\gamma)$ is what we earlier called $c_V(\gamma)$ in the group cocycle interpretation of $c_V$. For a general $1$--simplex  $\wt \sigma$ of $\wt K$ going from the 0--simplex $\alpha_- \wt\sigma_0$ to $\alpha_+ \wt\sigma_0$, with $\alpha_-$, $\alpha_+ \in \pi_1(S)$, the $\Ad\rho$--equivariance of $c_V$ implies that
$$
c_V(\wt \sigma)   = \Ad\rho(\alpha_-) \big(c_V(\wt\sigma_{\alpha_-^{-1}\alpha_+}) \big)  = \Ad\rho(\alpha_-) \Big( \frac d{dt} \rho_t(\alpha_-^{-1}\alpha_+) \rho_0(\alpha_-^{-1}\alpha_+)^{-1} {}_{|t=0}  \Big). 
$$

As a topological space, $K$ is homotopy equivalent to $S$. Lift an arbitrary homotopy equivalence $f \colon K \to S$ to a $\pi_1(S)$--equivariant map $\wt f \colon \wt K \to \wt S$, sending the base vertex $\wt\sigma_0 \in \wt K$ to a base point $\wt x_0 = \wt f(\wt \sigma_0) \in \wt S$. The map $\wt f$ then induces an isomorphism $H^1(f) \colon H^1(S; \sln_{\Ad\rho}) \to  H^1(K; \sln_{\Ad\rho})$. 

To determine the image $H^1(f)\big([c^1]\big) \in H^1(K; \sln_{\Ad\rho})$ of $[c^1] \in H^1(S; \sln_{\Ad\rho})$, we first evaluate $c^1$ on the  1--simplex  $\wt \sigma_\gamma$ associated to $\gamma\in \pi_1(S)$.
\begin{align*}
 c^1 \big( \wt f( \wt \sigma_\gamma ) \big) 
 &= c^0 \big(\wt f (\gamma \wt \sigma_0) \big) - c^0 \big(\wt f (\wt \sigma_0) \big) 
 = c^0 (\gamma \wt x_0) - c^0 (\wt x_0)
 \\
 &=  \frac d{dt} \rho_t(\gamma) \rho_0(\gamma)^{-1} {}_{|t=0}
+   \Ad\rho(\gamma) \big( c^0(\wt x_0) \big)
- c^0 (\wt x_0)
 \\
 &=  c_V(\wt\sigma_\gamma) 
+   \Ad\rho(\gamma) \big( c^0(\wt x_0) \big)
- c^0 (\wt x_0),
\end{align*}
using our earlier computation of $c^0 (\gamma \wt x_0)$ in the proof of Lemma~\ref{lem:ProjectiveBasisDefinesCocyle}. If we introduce an $\Ad\rho$--equivariant 0--cocycle  $c^{0\,\prime} \in C^0 (\wt S; \sln_{\Ad\rho})$ defined by the property that
$$
c^{0\, \prime}(\wt x) = 
\begin{cases}
 \Ad\rho(\alpha) \big( c^0(\wt x_0) \big)
 &\text{if } \wt x = \alpha \wt x_0 \text{ for some } \alpha \in \pi_1(S)
 \\
 0 &\text{if } \wt x \text{ is not in the orbit of } \wt x_0,
\end{cases}
$$
this can be rewritten as
\begin{align*}
  c^1 \big( \wt f( \wt \sigma_\gamma ) \big)
 &=  c_V(\wt\sigma_\gamma)  +  c^{0\,\prime}(\gamma \wt x_0) - c^{0\,\prime} (\wt x_0)
 \\
 &=  c_V(\wt\sigma_\gamma)  +  dc^{0\,\prime}\big( \wt f(\wt\sigma_\gamma) \big).
\end{align*}
By $\Ad\rho$--equivariance, this implies that
$$
 c^1 \big( \wt f( \wt \sigma ) \big) - dc^{0\,\prime}\big( \wt f(\wt\sigma) \big) = c_V(\wt\sigma) 
$$
for every $1$--simplex $\wt\sigma$ of $\wt K$. As a consequence, the isomorphism $H^1(f) \colon H^1(S; \sln_{\Ad\rho}) \to  H^1(K; \sln_{\Ad\rho})$ sends $[c^1] \in H^1(S; \sln_{\Ad\rho})$ to
$$
 H^1(f)\big([c^1]\big) =H^1(f)\big([c^1 -dc^{0\prime}]\big) = [c_V] 
$$
in $H^1(K; \sln_{\Ad\rho}) = H^1\big(\pi_1(S); \sln_{\Ad\rho}\big)$. 
\end{proof}

The cocycle $c^1 \in H^1 ( S;\sln_{\Ad\rho} )$ developed in this section will enable us to localize the computation of the Weil class $[c_V]\in H^1\big( \pi_1(S);\sln_{\Ad\rho} \big)$ associated to a tangent vector $V\in T_{[\rho]}\Hit(S)$, in order to compute the cup product $c_{V_1} \cupp\, c_{V_2}$ corresponding to two tangent vectors. For this purpose, we will not need to determine the full singular cocycle $c^1$, but just its evaluation on the 1--simplices of a suitably chosen triangulation of the surface $S$.


\section{Eruptions and shears}
\label{bigsect:EruptionsShears}

\subsection{Flags, triple ratios and double ratios}
\label{subsect:TripleDoubleRatios}

Recall that a (complete) \emph{flag} $F$ in $\R^n$ is a nested family of linear subspaces $0=F^{(0)} \subset F^{(1)} \subset F^{(2)} \subset \dots \subset F^{(n-1)} \subset F^{(n)} = \R^n$ such that each $F^{(a)}$ has dimension $a$. We denote the space of flags in $\R^n$ as $\Flag$. 

The standard action of the linear group $\GL$ on $\R^n$ induces an action of $\GL$ on $\Flag$, which descends to an action of the projective group $\PGL$ on $\Flag$. 
The key ingredients for the results of \cite{FocGon1} are certain invariants for the action of $\PGL$ on sufficiently generic finite families of flags. 

A flag triple  $(E,F,G) \in \Flag^3$ satisfies the \emph{Maximum Span Property} if
$$
\dim ( E^{(a)} + F^{(b)} + G^{(c)} ) = \min \{ a+b+c, n\}
$$
for every $a$, $b$, $c\in \{0,1,2,\dots, n\}$. In other words, this means that the span $ E^{(a)} + F^{(b)} + G^{(c)} $ of the three subspaces $ E^{(a)}$, $ F^{(b)}$, $ G^{(c)} $ is as large as possible. When this property is satisfied, we will also say that $(E,F,G)$ is a \emph{maximum-span flag triple}. Elementary considerations show that this is equivalent to the property that
$$
\R^n = E^{(a)} \oplus F^{(b)} \oplus G^{(c)} 
$$
for every $a$, $b$, $c$ with $a+b+c=n$. 

This definition is exactly what is needed to define, for every $a$, $b$, $c\geq1$ with $a+b+c=n$, the \emph{$(a,b,c)$--triple-ratio invariant} of a maximum-span flag triple $(E,F,G) \in \Flag^3$, which is the quantity
$$
X_{abc}(E,F,G)= 
\frac
{ e^{(a+1)} \wedge  f^{(b)} \wedge  g^{(c-1)}}
{ e^{(a-1)} \wedge  f^{(b)} \wedge  g^{(c+1)}}
\frac
{ e^{(a)} \wedge  f^{(b-1)} \wedge  g^{(c+1)}}
{ e^{(a)} \wedge  f^{(b+1)} \wedge  g^{(c-1)}}
\frac
{ e^{(a-1)} \wedge  f^{(b+1)} \wedge  g^{(c)}}
{ e^{(a+1)} \wedge  f^{(b-1)} \wedge  g^{(c)}}
 \in \R-\{0\}
$$
where all $e^{(a')}\in \mathsf\Lambda^{a'}E^{(a')}\cong \R$, $f^{(b')}\in \mathsf\Lambda^{b'}F^{(b')}\cong \R$ and $g^{(c')}\in \mathsf\Lambda^{c'}G^{(c')}\cong \R$ are arbitrary nonzero elements of the corresponding exterior powers, and where the quotients are taken for an arbitrary isomorphism $\mathsf\Lambda^n\R^n \cong \R$. 

\begin{prop}
[{\cite{FocGon1}}]
\label{prop:TripleRatioClassifyFlagTriples}
 Given two maximum-span flag triples $(E,F,G)$ and $(E', F', G') \in \Flag^3$, there exists a projective map $\phi\in \PGL$ sending $(E,F,G)$ to $(E', F', G')$ if and only if
 $$
 X_{abc}(E,F,G)= X_{abc}(E',F',G')
 $$
 for every $a$, $b$, $c\geq1$ with $a+b+c=n$. 
 
 In addition, the projective map $\phi$ is unique when it exists. 
 \qed
\end{prop}

Similarly, let $(E,F;G,H) \in \Flag^4$ be a flag quadruple such that the flag triples $(E,F,G)$ and $(E,F,H)$ satisfy the Maximum Span Property. For $a$, $b\geq1$ with $a+b=n$, the \emph{$(a,b)$--double-ratio invariant} of $(E,F;G,H)$ is defined as
$$
X_{ab} (E,F;G,H) =
- \frac
{ e^{(a)} \wedge  f^{(n-a-1)}\wedge  g^{(1)}} 
{ e^{(a)} \wedge  f^{(n-a-1)}\wedge  h^{(1)}}
\frac
{ e^{(a-1)} \wedge  f^{(n-a)}\wedge  h^{(1)}}
{ e^{(a-1)}  \wedge  f^{(n-a)}\wedge  g^{(1)}} 
 \in \R - \{0\}
$$
where all $e^{(a')}\in \mathsf\Lambda^{a'}E^{(a')}\cong \R$, $f^{(b')}\in \mathsf\Lambda^{b'}F^{(b')}\cong \R$, $g^{(1)}\in \mathsf\Lambda^{1}G^{(1)}\cong \R$, $h^{(1)}\in \mathsf\Lambda^{1}H^{(1)}\cong \R$  are arbitrary nonzero elements in the corresponding exterior powers. 

Note that $X_{ab} (E,F;G,H) $ does not depend on the whole flags $G$ and $H$, but only on the lines $G^{(1)}$, $H^{(1)}$. In particular, the requirement that $(E,F,G)$ and $(E,F,H)$ satisfy the Maximum Span Property is an overkill if we only want $X_{ab} (E,F;G,H) $ to be defined.  

\begin{prop}
[{\cite{FocGon1}}]
\label{prop:DoubleTripleRatioClassifyFlagQuadruples}
 Given two flag quadruples $(E,F;G, H)$ and $(E', F'; G', H') \in \Flag^4$ such that $(E,F,G)$, $(E,F,H)$, $(E',F',G')$ and $(E',F',H')$ satisfy the Maximum Span Property, there exists a projective map $\phi\in \PGL$  sending  $(E,F;G, H)$ to $(E', F'; G', H')$ if and only if
\begin{align*}
 X_{abc}(E,F,G)&= X_{abc}(E',F',G'),
 \\
 X_{abc}(E,F,H)&= X_{abc}(E',F',H'),
 \\
\text{and } X_{a'b'}(E,F;G, H)&= X_{a'b'}(E', F'; G', H') 
\end{align*}
 for every $a$, $b$, $c$, $a'$, $b'\geq1$ with $a+b+c=n$ and $a'+b'=n$. 
 
  In addition, the projective map $\phi$ is unique when it exists. 
 \qed
\end{prop}

A basis $\mathcal B=\{e_1, e_2, \dots, e_n\}$ for $\R^n$ specifies two preferred flags: the \emph{ascending flag} $E$ defined by the property that each subspace $E^{(a)}$ is spanned by the first $a$ vectors $e_1$, $e_2$, \dots, $e_a$ of $\mathcal B$; and the \emph{descending flag} $F$ for which each $F^{(b)}$ is spanned by the last $b$ vectors $e_{n-b+1}$, $e_{n-b+2}$, \dots, $e_n$. 

We will frequently make use of the following elementary result, which was also the main ingredient in the proof of the uniqueness part of Propositions~\ref{prop:TripleRatioClassifyFlagTriples} and \ref{prop:DoubleTripleRatioClassifyFlagQuadruples}.
\begin{lem}
\label{lem:ProjectiveMapBetweenFlagTriples}
 Let $(E,F,G)$, $(E', F', G') \in \Flag^3$ be two maximum-span flag triples. Then, there is a unique projective map $\phi \in \PGL$ sending the flag $E$ to $E'$, the flag $F$ to $F'$, and the line $G^{(1)}$ to $G^{\prime(1)}$. 
 \end{lem}
\begin{proof}
 An elementary argument provides a basis $\mathcal B=\{e_1, e_2, \dots, e_n\}$ whose ascending flag is equal to $E$, whose descending flag is equal to $F$, and such that the line $G^{(1)}$ is spanned by the sum $e_1+e_2+ \dots +e_n$ of the elements of $\mathcal B$. In addition, this basis is unique up to multiplication of the elements of $\mathcal B$ by the same scalar. The result then follows by consideration of the basis $\mathcal B'$ similarly associated to the flag triple $(E', F', G')$. 
\end{proof}

\subsection{Connection with Fock-Goncharov invariants}
\label{subsect:DoubleTripleRatiosFockGoncharov}

Let $\wt S$ be the universal cover of the surface $S$, and let $\bdry$ denote its circle at infinity.  An important consequence of the Anosov property of Hitchin representations developed by Labourie \cite{Lab} is the following.

\begin{prop}
 [{\cite{Lab}}]
 \label{prop:FlagMap}
 Let $\rho \colon \pi_1(S) \to \PGL$ be a Hitchin representation. Then there exists a unique continuous map $\F_\rho \colon \bdry \to \Flag$ that is $\rho$--equivariant, in the sense that
 $$
 \F_\rho(\gamma x) = \rho(\gamma) \F_\rho(x)
 $$
 for every $x\in \bdry$ and $\gamma \in \pi_1(S)$. \qed
\end{prop}

The map $\F_\rho \colon \bdry \to \Flag$ is the \emph{flag map} of the Hitchin representation $\rho \colon \pi_1(S) \to \PGL$. 

\begin{prop}
 [{\cite{Lab, FocGon1}}]
 \label{prop:FlagMapPositive}
 The flag map $\F_\rho \colon \bdry \to \Flag$ of a Hitchin representation $\rho \colon \pi_1(S) \to \PGL$ is \emph{positive}, in the sense that:
\begin{enumerate}
 \item For every distinct three points $x$, $y$, $z\in \bdry$, and every integers $a$, $b$, $c\geq 1$ with $a+b+c=n$, the triple-ratio
 $
 X_{abc} \big( \F_\rho(x),  \F_\rho(y),  \F_\rho(z) \big)
 $
 is well-defined and positive. 
 \item For every four points $x$, $y$, $u$, $v\in \bdry$ such that $u$ and $v$ are in different components of $\bdry - \{x,y\}$, and for every integers $a$, $b\geq 1$ with $a+b=n$, the double-ratio $X_{ab}\big ( \F_\rho(x),  \F_\rho(y);  \F_\rho(u),  \F_\rho(v) \big)$ is  positive. 
 \qed
\end{enumerate}
\end{prop}

We can now hint at the way the generalized Fock-Goncharov invariants $\tau_\rho^{abc}(U,s)$ and $\sigma_\rho^{ab}(e)$ of a Hitchin representation are defined. These are based on the triple-ratio and double-ratio invariants of \S \ref{subsect:TripleDoubleRatios}. 

Let $\Psi$ be a train track carrying the geodesic lamination $\lambda$, associated to a train track neighborhood $\Phi$ of $\lambda$. Let $U$ be a component of $S-\Phi$, and let $s$ be one of the three corners of the closure of $U$. By the correspondence of Proposition~\ref{prop:ComplementsTrainTrackGeodLamination}, the corner $s$ determines a vertex $x$ of the component $T$ of $S-\lambda$ that contains the component of $S-\Phi$ corresponding to $U$. For the preimage $\wt\lambda$ of $\lambda$ in the universal cover $\wt S$, lift $T$ to a component $\wt T$ of $\wt S-\wt\lambda$ and let $\wt x\in \bdry$ be the vertex of $\wt T$ corresponding to $x$. Index the other vertices of $\wt T$ as $\wt y$ and $\wt z $ in such a way that $\wt x$, $\wt y$, $\wt z$ occur in this order counterclockwise around $\wt T$, and consider the flags $ \F_\rho(\wt x)$, $ \F_\rho(\wt y)$, $ \F_\rho(\wt z) \in \Flag$ associated to these vertices by the flag map $ \F_\rho$. Then, the invariant $\tau_\rho^{abc}(U,s)$ is defined as 
$$
\tau_\rho^{abc}(U,s) = \log  X_{abc} \big( \F_\rho(\wt x),  \F_\rho(\wt y),  \F_\rho(\wt z) \big). 
$$
Note that the positivity of $X_{abc} \big( \F_\rho(\wt x),  \F_\rho(\wt y),  \F_\rho(\wt z) \big)$ guaranteed by Proposition~\ref{prop:FlagMapPositive} is needed for the logarithm to be defined. 

The shearing invariant $\sigma_\rho^{ab}(e)$ is easily described in the very special case when there is a unique leaf $g$ of $\lambda$ that passes through the branch of $\Phi$ corresponding to $e$. The orientation of the branch $e$ determines an orientation for $g$, and  $g$ separates two components $T$ and $T'$ of $S-\lambda$ with $T$ on the right of $g$ and $T'$ on its left. Lift $g$ to an oriented leaf $\wt g$ of  $\wt\lambda$, and let $\wt T$ and $\wt T'$ be the two components of $\wt S-\wt\lambda$ that are adjacent to $\wt g$ and respectively lift $T$ and $T'$. Let $x \in \bdry$ be the positive endpoint of $\wt g$, let $y$ be its negative endpoint, let $u$ be the third vertex of $\wt T$, and let $v$ be the third vertex of $\wt T'$. Then,
$$
\sigma_\rho^{ab}(e) = \log  X_{ab} \big( \F_\rho( x),  \F_\rho( y);  \F_\rho( u),  \F_\rho( v)  \big) 
$$
in this special case. Note that the positivity of $X_{ab} \big( \F_\rho( x),  \F_\rho( y);  \F_\rho( u),  \F_\rho( v)  \big) $ is again critical. 

This definition of $\sigma_\rho^{ab}(e)$ in this special case, as well as the definition of the triangle invariants  $\tau_\rho^{abc}(U,s) $, are exactly those of the (non-generalized) Fock-Goncharov invariants of \cite{FocGon1}. The definition of the shearing invariant $\sigma_\rho^{ab}(e)$ is much more elaborate in the generic case where there are infinitely (and possibly uncountably) many  leaves of $\lambda$ passing through the branch of $\Phi$ corresponding to $e$, and requires the use of the shearing map discussed in \S \ref{subsect:SlitheringMaps}; in particular, see Fact~\ref{fact:ShearingUsingSlithering}.

\subsection{Left and right eruption maps} 
\label{subsect:Eruptions}
The notion of eruption, as developed in \cite{SunWieZha}, provides a very convenient tool to modify a   map $\F \colon \bdry \to \Flag$ in such a way that one triple-ratio invariant $X_{abc} \big( \F(x), \F(y), \F(z) \big)$ can be arbitrarily adjusted, while many other invariants $X_{a'b'c'} \big( \F(x'), \F(y'), \F(z') \big)$ and $X_{a'b'} \big( \F(x'), \F(y'); \F(u'), \F(v') \big)$ remain unchanged. 

We will only need a very weak version of these eruptions. Given a maximum-span flag triple $(E,F,G)\in \Flag^3$ and integers $a$, $b$, $c\geq 1$ with $a+b+c=n$, we have a direct sum decomposition
$$
\R^n = E^{(a)} \oplus F^{(b)} \oplus G^{(c)} .
$$
The \emph{left and right $(a,b,c)$--eruptions of amplitude $t\in \R$ along $(E,F,G)$} are the two projective maps $L^{abc}_{EFG}(t)$ and $R^{abc}_{EFG}(t)\in \PGL$ respectively defined as:
\begin{align*}
L^{abc}_{EFG}(t) &= \E^{- t}\Id_{E^{(a)}} \oplus  \Id_{F^{(b)}} \oplus  \Id_{G^{(c)}}
 \\
R^{abc}_{EFG}(t) &=  \Id_{E^{(a)}} \oplus \E^{ t} \Id_{F^{(b)}} \oplus  \Id_{G^{(c)}}.
\end{align*}

We list the more important properties of these eruption maps. 

\begin{lem}
\label{lem:EruptionsMainProperties}
 For every maximum-span flag triple $(E,F,G)\in \Flag^3$,  for every integers $a$, $b$, $c \geq 1$ with $a+b+c=n$, and for every $t \in \R$, 
\begin{enumerate}
 \item $ L_{EFG}^{abc}(t)(E) = E$,  $R_{EFG}^{abc}(t)(F)=F$;
 \item $L_{EFG}^{abc}(t)(G) = R_{EFG}^{abc}(t)(G)$;
 \item $L_{EFG}^{abc}(t)$ and $R_{EFG}^{abc}(t)$ respect the lines $E^{(1)}$, $F^{(1)}$ and $G^{(1)}$;
 \item for every integers $a'$, $b'$, $c' \geq 1$ with $a'+b'+c'=n$, 
 $$
 X_{a'b'c'}\big(E,F, L_{EFG}^{abc}(t)(G) \big)=
\begin{cases}
 X_{a'b'c'}(E,F, G) &\text{if } (a',b',c') \neq (a,b,c)
 \\
 \E^t X_{abc}(E,F,G) &\text{if } (a',b',c') = (a,b,c)
\end{cases}
$$
\item for every integers $a'$, $b'$, $c' \geq 1$ with $a'+b'+c'=n$ and for every $t'\in \R$, $ L_{EFG}^{abc}(t)$ commutes with  $ L_{EFG}^{a'b'c'}(t')$;

\item for every integers $a'$, $b'$, $c' \geq 1$ with $a'+b'+c'=n$ and for every $t'\in \R$, $ R_{EFG}^{abc}(t)$ commutes with  $ R_{EFG}^{a'b'c'}(t')$.
\end{enumerate}
\end{lem}
 
\begin{proof}
Properties (1) and (2) are the content of \cite[ Lem. 3.1]{SunWieZha}, after composition of the two sides of each of these equalities with the map $ \E^{\frac t3}\Id_{E^{(a)}} \oplus \E^{-\frac t3} \Id_{F^{(b)}} \oplus  \Id_{G^{(c)}}$.

Property (3) is obvious from definitions.

Property (4) is proved in \cite[ Prop. 3.5(1)]{SunWieZha}.

Properties (5--6) are the content of \cite[ Prop. 3.14(1)]{SunWieZha}.
\end{proof}

To explain our ``left'' and ``right'' terminology for eruptions, we can say that we will use them in situations where the flags $E$, $F$, $G\in \Flag$  are associated to points $x$, $y$, $z\in \bdry$ where $x$ sits to the left of $y$ as seen from a base point, while $z$ sits behind the geodesic $xy$ as seen from the same base point. Then $L_{EFG}^{abc}(t)$ and $R_{EFG}^{abc}(t)$ will be applied to flags associated to points $w\in \bdry$ that sit behind $xy$ and respectively are to the left and right of $z$. See also the more general version of eruptions in \cite{SunWieZha}.

\subsection{Shearing maps}
\label{subsect:Shearing}

Shearing maps similarly modify the double-ratio invariants $X_{ab}(E,F; G,H)$, but are much simpler than eruptions. 

Consider a flag pair $(E,F) \in \Flag^2$ that is \emph{transverse}, in the sense that every linear subspace $E^{(a)}$ is transverse to every $F^{(b)}$. Given $a$, $b\geq 1$ with $a+b=n$ and $t\in \R$, the \emph{$(a,b)$--shearing map of amplitude $t$ along $(E,F)$} is the projective map $S^{ab}_{EF}(t) \in \PGL$ defined as
$$
 S^{ab}_{EF}(t) = \E^{ t}\Id_{E^{(a)}} \oplus  \Id_{F^{(b)}}.
$$

\begin{lem}
\label{lem:ShearingMainProperties}
 For every transverse flag pair  $(E,F) \in \Flag^2$,  for every integers $a$, $b \geq 1$ with $a+b=n$, and for every $t\in \R$,  
\begin{enumerate}
 \item $S_{EF}^{ab}(t)$ respects the flags $E$ and $F$;
 \item for every $a'$, $b'\geq 1$ with $a'+b'=n$, and for every $G$, $H\in \Flag$ such that the flag triples $(E,F,G)$ and $(E,F,H)$ satisfy the Maximum Span Property, 
$$
 X_{a'b'}\big(E,F;G, S_{EF}^{ab}(t)(H)  \big)=
\begin{cases}
  X_{a'b'}\big(E,F;G, H  \big) &\text{if } (a',b') \neq (a,b)
 \\
 \E^t X_{ab}\big(E,F; G, H  \big) &\text{if } (a',b') = (a,b)
\end{cases}
$$
\item for every $a'$, $b'\geq 1$ with $a'+b'=n$ and every $t' \in \R$, $S_{EF}^{a'b'}(t')$ commutes with $S_{EF}^{ab}(t)$.
\end{enumerate}
\end{lem}

\begin{proof}
By transversality, there is a basis $\mathcal B$ for  $\R^n$ in which $E$ is the ascending flag and $F$ is the descending flag. Using this basis for computations easily provides a proof of the above properties.
\end{proof}

\section{The flag map of a variation of a Hitchin representation}
\label{bigsect:VariationFlagMap}

We consider a variation of a Hitchin character $[\rho] \in \Hit(S)$, namely a nearby character $[\wh\rho] \in \Hit(S)$ associated to small variations $ \Delta \tau^{abc}$ and $\Delta \sigma^{ab}$ of the generalized Fock-Goncharov invariants of $[\rho]$, in the sense that
\begin{align*}
 \tau^{abc}_{\wh\rho}(U,s) &= \tau^{abc}_\rho(U,s) + \Delta \tau^{abc}(U,s)
 \\
 \sigma^{a'b'}_{\wh\rho}(e) &= \sigma^{a'b'}_\rho(e) + \Delta \sigma^{a'b'}(e)
\end{align*}
for every component $U$ of the complement $S-\Psi$ of a train track $\Psi$ carrying the geodesic lamination $\lambda$, for every corner $s$ of the triangle $U$, for every oriented branch $e$ of $\Psi$, and for every $a$, $b$, $c$, $a'$, $b' \geq 1$ with $a+b+c=a'+b'=n$. 

This section is devoted to an explicit comparison between the  flag maps $\F_\rho$, $\F_{\wh\rho} \colon \bdry \to \Flag$, in terms of these variations $ \Delta \tau^{abc}$ and $\Delta \sigma^{ab}$. This will lead us to a practical implementation of the scheme of \S \ref{subsect:SimplicialWeil} to compute the Weil class $[c_V] \in H^1(S;\sln_{\Ad\rho})$ associated to the tangent vector $V\in T_{[\rho]}\Hit(S)$, in terms of the corresponding infinitesimal variations  $ \dot \tau^{abc}$ and $\dot \sigma^{ab}$ of the generalized Fock-Goncharov invariants of $[\rho]$.

\subsection{Lifting invariants from the train track $\Psi$ to the geodesic lamination $\wt\lambda$}
\label{subsect:LiftInvariants}

To compare the  flag maps $\F_\rho$, $\F_{\wh\rho} \colon \bdry \to \Flag$, we need to lift our data to the universal cover $\wt S$. By construction, the train track $\Psi$ is obtained by collapsing the ties of a train track neighborhood $\Phi$ of $\lambda$. 
 Let $\wt \lambda \subset \wt S$ be the preimage of the geodesic lamination $\lambda$, and let $\wt \Phi$ and $\wt \Psi$ be the respective preimages of  $\Phi$ and $\Psi$. 

The one-to-one correspondence of Proposition~\ref{prop:ComplementsTrainTrackGeodLamination} provides us with the following rephrasing of the triangle invariants. 
 
 \begin{lem}
 \label{lem:LiftTriangleInvariantsToUniversalCover}
 The  triangle invariant  map
 $$
  \tau_\rho^{abc} \colon \left\{ ( U,  s);\  U \text{ component of }  S -  \Psi \text{ and }  s \text{ corner of }  U \right\} \to \R
 $$
uniquely determines a map
$$
  \tau_\rho^{abc} \colon \left\{ (\wt T,  x);\ \wt T \text{ component of } \wt S - \wt \lambda \text{ and }  x \text{ vertex of } \wt T \right\} \to \R,
$$
denoted by the same symbol $  \tau_\rho^{abc}  $, such that 
$$
  \tau_\rho^{abc}  (\wt T,  x) =   \tau_\rho^{abc} ( U,  s)
$$
whenever  the one-to-one correspondences of Proposition~{\upshape\ref{prop:ComplementsTrainTrackGeodLamination}} associate $U$ to the projection $T$ of $\wt T$ in $S$, and the corner $s$ of $U$ to the vertex of $T$ that is the image of $x$ under the projection. 
\qed
\end{lem}
 
The interpretation of shear invariants in terms of components of $\wt S - \wt\lambda$ is more elaborate, but will play a critical role in our construction.

\begin{lem}
\label{lem:LiftShearsToUniversalCover}
 The  shear invariant map 
 $$
  \sigma_\rho^{ab} \colon \left\{\text{oriented branches of }  \Psi \right\} \to \R
 $$
 uniquely determines a $\pi_1(S)$--equivariant map
 $$
 \sigma_\rho^{ab} \colon \left\{ (\wt T, \wt T'); \ \wt T, \wt T' \text{ distinct components of }  \wt S-\wt\lambda \right\} \to \R,
 $$
 denoted by the same symbol $\sigma_\rho^{ab}$, such that the following two conditions are satisfied. 
\begin{enumerate}
 \item For every oriented branch $e$ of the train track $\Psi$ and for every lift of $e$ to a branch $\wt e$ of the preimage $\wt \Psi$ of $\Psi$ in $\wt S$, let $\wt U$ and $\wt U'$ be the two components of $\wt S - \wt \Psi$ that are adjacent to $\wt e$, indexed so that $\wt U$ sits to the left of $\wt e$ for the orientation of $\wt e$. Then, if $\wt T$ and $\wt T'$ are the two components of $\wt S - \wt\lambda$  that contain the components of $\wt S- \wt\Phi$ respectively associated to $\wt U$ and $\wt U'$, 
 $$
 \sigma_\rho^{ab} ( \wt T, \wt T') = \sigma_\rho^{ab}( e). 
 $$
 
 \item Let $\wt T$, $\wt T'$, $\wt T''$ be three components of $\wt S - \wt \lambda$ such that $\wt T'$ separates $\wt T$ from $\wt T''$, in the sense that $\wt T$ and $\wt T''$ sit in different components of $\wt S - \wt T'$. In particular, there exist exactly two sides of the triangle $\wt T'$ that separate $\wt T$ from $\wt T''$; let $x$ be the vertex of $\wt T'$ that is common to these two sides. Then: 
\begin{enumerate}
 \item if $x$ sits to the left of $\wt T'$ as seen from $\wt T$
 $$ \sigma_\rho^{ab} ( \wt T, \wt T'') =  \sigma_\rho^{ab} ( \wt T, \wt T') +  \sigma_\rho^{ab} ( \wt T', \wt T'')  - \sum_{b'+c'=n-a} \tau_\rho^{ab'c'}(\wt T', x);$$ 
  \item if $x$ sits to the right of $\wt T'$ as seen from $\wt T$
 $$ \sigma_\rho^{ab} ( \wt T, \wt T'') =  \sigma_\rho^{ab} ( \wt T, \wt T') +  \sigma_\rho^{ab} ( \wt T', \wt T'')  - \sum_{a'+c'=n-b} \tau_\rho^{bc'a'}(\wt T', x).$$ 
\end{enumerate}
\end{enumerate}
\end{lem}

\begin{proof}
 This is an immediate consequence of the alternative interpretation of relative tangent cycles developed in \cite[\S 4.7]{BonDre2}. 
\end{proof}

Actually, the invariants $\tau_\rho^{abc}$ and $\sigma_\rho^{ab}$ were originally defined in \cite{BonDre2} in the form of Lemmas~\ref{lem:LiftTriangleInvariantsToUniversalCover} and \ref{lem:LiftShearsToUniversalCover}, before being expressed in terms of data associated to a train track carrying the geodesic lamination $\lambda$. 

We will mostly rely on the natural extension of  Lemmas~\ref{lem:LiftTriangleInvariantsToUniversalCover} and \ref{lem:LiftShearsToUniversalCover} to the variations $\Delta\tau^{abc}$ and $\Delta\sigma^{ab}$ of the generalized Fock-Goncharov invariants. 

\subsection{Slithering maps} 
\label{subsect:SlitheringMaps}

The slithering maps are fundamental objects in the definition of the shearing invariants $\sigma_\rho^{ab}(e)$ in \cite{BonDre2}. 

Consider two oriented leaves $xy$ and $x'y'$ of the geodesic lamination $\wt \lambda$ in the universal cover $\wt S$, oriented in parallel in the sense that their endpoints $x$, $y$, $y'$, $x' \in \bdry$ occur in this order (clockwise or counterclockwise) around the circle at infinity $\bdry$. The flag map $\F_\rho \colon \bdry \to \Flag$ of a Hitchin representation $\rho \colon \pi_1(S) \to \PGL$ associates flags $E=\F_\rho(x)$, $F=\F_\rho(y)$, $E'=\F_\rho(x')$, $F'=\F_\rho(y')\in \Flag$ to these endpoints. By, for instance, Lemma~\ref{lem:ProjectiveMapBetweenFlagTriples} applied to $(E,F,G)$ and $(E', F', G')$ with $G= \F_\rho(z)$ and $G'= \F_\rho(z')$ associated to additional points $z$, $z'\in \bdry$,  there  exists a linear isomorphism $\R^n \to \R^n$ sending $E'$ to $E$ and $F'$ to $F$. In fact, there exist many such linear isomorphisms, since the stabilizer of the pair $(E,F)$ in $\GL$ has dimension $n$. The slithering construction uses the maximality of the  geodesic lamination $\wt \lambda$ to single out a preferred isomorphism $\Sigma_{xy, x'y'} \colon \R^n \to \R^n$ sending $(E',F')$ to $(E, F)$. 

More precisely, first consider  the case where the two leaves $xy$ and $x'y'$ have a common endpoint $x=x'$; in particular, $E=E'$. Then $\Sigma_{xy, x'y'} \colon \R^n \to \R^n$ is the unique linear map that sends $(E',F')$ to $(E, F)$ and is \emph{unipotent} in the sense that all its eigenvalues are equal to 1. The unipotent property is equivalent to saying that, if we choose an arbitrary basis $e_1$, $e_2$, \dots, $e_n$ for $\R^n$ such that each $E^{(a)}$ is spanned by $\{e_1, e_2, \dots, e_a\}$ and each $F^{(b)}$ is spanned by $\{e_{n-b+1}, e_{n-b+2}, \dots, e_n\}$, the matrix of  $\Sigma_{xy, x'y'}$ in this basis is upper triangular with all diagonal terms equal to 1.

The slithering map $\Sigma_{xy, x'y'} $ is similarly defined when $y=y'$. 

In the general case, let $\wt T_1$, $\wt T_2$, \dots, $\wt T_{i_0}$ be a family of components of $\wt S - \wt \lambda$ separating the leaf $xy$ from $x'y'$, indexed in this order from $xy$ to $x'y'$ in the sense that each $\wt T_i$ separates $xy$ from $\wt T_{i+1}$. Let $x_iy_i$ be the side of $\wt T_i$ that faces $xy$, and let $x_i'y_i'$ be the side that faces $x'y'$, with the parallelisms between $xy$, $x_iy_i$, $x_i'y_i'$ and $x'y'$ compatible with orientations. Note that, for each $i$, either $x_i=x_i'$ or $y_i=y_i'$, so that the slithering map $\Sigma_{x_iy_i, x_i'y_i'}$ is defined by the above special cases. Then, the  \emph{slithering map} $\Sigma_{xy, x'y'} \colon \R^n \to \R^n$ is defined as the limit of
$$
\Sigma_{x_1y_1, x_1'y_1'} \circ \Sigma_{x_2y_2, x_2'y_2'} \circ \dots \circ \Sigma_{x_{i_0}y_{i_0}, x_{i_0}'y_{i_0}'}
$$ 
as the family $\{ \wt T_1, \wt T_2, \dots, \wt T_{i_0} \}$ tends to the set of all components of $\wt S - \wt \lambda$ separating $xy$ from $x'y'$. See \cite[\S 5.1]{BonDre2} for a proof that the limit exists,  and for a proof of the following fact.

\begin{lem}
\label{lem:SlitheringSendsFlagToFlag}
The slithering map $\Sigma_{xy, x'y'}$  sends the flag $E'=\F_\rho(x')$ to $E=\F_\rho(x)$, and the flag $F'=\F_\rho(y')$ to $F=\F_\rho(y)$. 
\qed
\end{lem}

We note the following elementary properties. 

\begin{lem}
\label{lem:CompositionSlitheringMaps}
 Let $xy$, $x'y'$, $x''y''$ be three oriented leaves of the geodesic lamination $\wt \lambda$ that are parallel to each other. Then
 \pushQED{\qed}
\begin{align*}
\Sigma_{yx, y'x'} &= \Sigma_{xy, x'y'}
&
\Sigma_{xy, x''y''} &= \Sigma_{xy, x'y'} \circ \Sigma_{x'y', x''y''}
&
\Sigma_{x'y', xy}&=\Sigma_{xy, x'y'}^{-1}.
\qedhere
\end{align*}
\end{lem}

The slithering maps play a critical role in the definition of the shearing invariants $\sigma_\rho^{ab}$.  Let $\wt T$ and $\wt T'$ be two components of the complement $\wt S - \wt \lambda$. Index the vertices of $\wt T$ as $x$, $y$, $z\in \bdry$ \emph{clockwise} in this order around $\wt T$, and in such a way that $xy$ is the side facing $\wt T'$. Then index the vertices of $\wt T'$ as $x'$, $y'$, $z'\in \bdry$ this time \emph{counterclockwise} in this order, and so that $x'y'$ is the side facing $\wt T$. Using the correspondence of Lemma~\ref{lem:LiftShearsToUniversalCover},   the shearing invariant is defined in \cite[\S 5.2]{BonDre2} as follows. 

\begin{fact}
\label{fact:ShearingUsingSlithering}
With the vertex labelling conventions above, 
\begin{align*}
 \sigma^{ab}_\rho(\wt T, \wt T') &= \log X_{ab} \big( \F_\rho(x),  \F_\rho(y);  \F_\rho(z),  \Sigma_{xy, x'y'}(\F_\rho(z')) \big)
 \\
 &= \log X_{ab} \big( \F_\rho(x'),  \F_\rho(y');   \Sigma_{x'y', xy} (\F_\rho(z)), \F_\rho(z') \big), 
\end{align*}
 where the double-ratio $X_{ab}(E,F;G,H)$ is defined in {\upshape\S \ref{subsect:TripleDoubleRatios}}.
 \end{fact}

\subsection{Variation of slithering maps}
\label{subsect:VariationSlitheringMap}

We now return to the situation considered at the beginning of this \S \ref{bigsect:VariationFlagMap}, with two Hitchin representations $\rho$, $\wh\rho  \colon \pi_1(S) \to \PGL$ such that
\begin{align*}
 \tau^{abc}_{\wh\rho}(U,s) &= \tau^{abc}_\rho(U,s) + \Delta \tau^{abc}(U,s)
 \\
 \sigma^{a'b'}_{\wh\rho}(e) &= \sigma^{a'b'}_\rho(e) + \Delta \sigma^{a'b'}(e)
\end{align*}
for every component $U$ of the complement $S-\Psi$ of a train track $\Psi$ carrying the geodesic lamination $\lambda$, for every corner $s$ of the triangle $U$, for every oriented branch $e$ of $\Psi$, and for every $a$, $b$, $c$, $a'$, $b' \geq 1$ with $a+b+c=a'+b'=n$. 

Let $\wt T_0$ be a component of $\wt S-\wt \lambda$, which we will use as a base component, and let $xy$ be a leaf of the geodesic lamination $\wt\lambda$. We choose the indexing so that $x$ is to the left of $y$ in $\bdry$, as seen from $\wt T_0$. Index the vertices of $\wt T_0$ \emph{clockwise} as $u_0$, $v_0$, $w_0 \in \bdry$ so that the side $u_0v_0$ faces $xy$. The clockwise convention is designed so that the orientations of $u_0v_0$ is parallel to that of $xy$. See Figure~\ref{fig:VariationSlithering}. 

 We want to compare the slithering maps $\Sigma_{xy,u_0v_0}$, $\wh\Sigma_{xy,u_0v_0} \colon \R^n \to \R^n$ respectively associated to $\rho$ and $\wh\rho$.
 
 \begin{figure}[htbp]

\SetLabels
(  .54* .24 )  $\wt T_0$ \\
( .86 *  .52)  $\wt T_i$ \\
(  .12* .71 )  $\wt T_j$ \\
\R( 0 * .5 )  $x_i$ \\
\L(  1* .44 )  $y_i$ \\
\L(1  * .66 )  $z_i$ \\
\R( 0 * .67 )  $x_j$ \\
\L(  1*  .82)  $y_j$ \\
\R( 0 * .83 )  $z_j$ \\
\R( 0 *  .32)  $u_0$ \\
\L(  1*  .34)  $v_0$ \\
\T( .7 * 0 )  $w_0$ \\
\R( 0 *  .98)  $x$ \\
\L( 1 * .93 )  $y$ \\
\endSetLabels
\centerline{\AffixLabels{\includegraphics{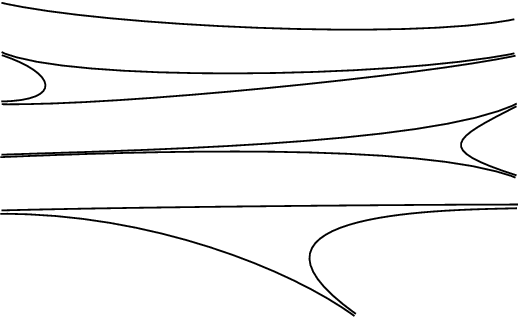}}}

\caption{The setup of Proposition~\ref{prop:VariationSlitheringMap}.}
\label{fig:VariationSlithering}
\end{figure}
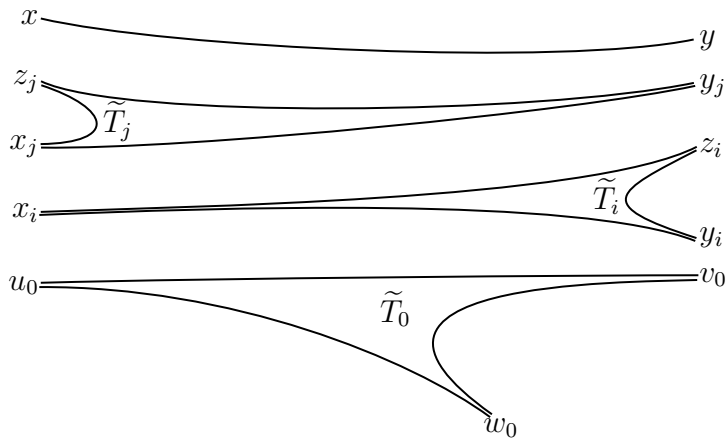

As before, consider a family $\wt T_1$, $\wt T_2$, \dots, $\wt T_{i_0}$ of components of $\wt S - \wt \lambda$ separating $u_0v_0$ from $xy$, indexed in this order from $u_0v_0$ to $xy$ in the sense that each $\wt T_i$ separates $u_0v_0$ from $\wt T_{i+1}$. Index the vertices of $\wt T_i$ as $x_i$, $y_i$, $z_i \in \bdry$ in such a way that they occur counterclockwise in this order around $\wt T_i$ and that $x_iy_i$ is the side of $\wt T_i$ facing $\wt T_0$. See Figure~\ref{fig:VariationSlithering}. For $i=1$, $2$, \dots, $i_0$, let $E_i = \F_\rho(x_i)$, $F_i = \F_\rho(y_i)$, $G_i = \F_\rho(z_i) \in \Flag$ be the flags associated to the vertices of $\wt T_i$  by the flag map $\F_\rho \colon \bdry \to \Flag$, and similarly define $E_0 = \F_\rho(u_0)$, $F_0 = \F_\rho(v_0)$, $G_0 = \F_\rho(w_0) \in \Flag$. 

 By Lemma~\ref{lem:ProjectiveMapBetweenFlagTriples}, there is a unique projective map $A_0 \in \PGL$ sending the flag $E_0=\F_\rho(u_0)$ to $\F_{\wh\rho}(u_0)$, the flag $F_0=\F_\rho(v_0)$ to $\F_{\wh\rho}(v_0)$, and the line $G_0^{(1)}=\F_\rho(w_0)^{(1)}$ to $\F_{\wh\rho}(w_0)^{(1)}$. 

\begin{prop}
\label{prop:VariationSlitheringMap}
The slithering map $\wh\Sigma_{xy, u_0v_0}$ defined by $\wh\rho$  is the limit of
 $$
 A_0 \circ A(\wt T_1) \circ A(\wt T_2) \circ \dots \circ A(\wt T_{i_0}) \circ \Sigma_{xy, u_0v_0} \circ A_0^{-1} \in \GL
 $$
 as the  family $\{ \wt T_1, \wt T_2, \dots, \wt T_{i_0}\}$ tends to the set of all components of $\wt S - \wt \lambda$ separating $\wt T_0$ from $\wt T$, where
 $$
A(\wt T_i)=
\begin{cases}
\displaystyle
\Bigcirc_{a+b=n}  S^{ab}_{E_iF_i} (\Delta\sigma^{ab}(\wt T_0, \wt T_{i}))
\circ
 \Bigcirc_{a+b+c=n} L^{abc}_{E_iF_iG_i} \left( \Delta \tau^{abc} (\wt T_i, x_i) \right)
 \\ \qquad\qquad\displaystyle
 \circ \Bigcirc_{a+b+c=n} S^{a(b+c)}_{E_i G_i} ( \Delta\tau^{abc}(\wt T_i, x_i))
 \circ 
 \Bigcirc_{a+b=n} S^{ab}_{E_i G_i} (-\Delta\sigma^{ab}(\wt T_0, \wt T_i))
\\ \qquad \qquad\qquad\qquad\qquad\qquad\qquad\qquad\qquad\qquad
\text{if } xy \text{ faces the side } x_iz_i \text{ of }\wt T_i 
\\
\displaystyle
\Bigcirc_{a+b=n}  S^{ab}_{E_iF_i} (\Delta\sigma^{ab}(\wt T_0, \wt T_{i}))
\circ
\Bigcirc_{a+b+c=n} R^{abc}_{E_iF_iG_i} \left( \Delta \tau^{abc} (\wt T_i, x_i) \right)
 \\ \qquad\qquad\displaystyle
\circ \Bigcirc_{a+b+c=n} S^{(a+c)b}_{G_iF_i} ( \Delta\tau^{abc}(\wt T_i, x_i))
\circ 
\Bigcirc_{a+b=n} S^{ab}_{G_iF_i} (-\Delta\sigma^{ab}(\wt T_0, \wt T_i))
\\ \qquad \qquad\qquad\qquad\qquad\qquad\qquad\qquad\qquad\qquad
\text{if } xy \text{ faces the side } y_iz_i \text{ of }\wt T_i ,
\end{cases}
$$
where $A_0 \in \PGL$ is defined as above, where
\begin{align*}
L^{abc}_{EFG}(t) &= \E^{- t}\Id_{E^{(a)}} \oplus  \Id_{F^{(b)}} \oplus  \Id_{G^{(c)}} \in \GL
 \\
R^{abc}_{EFG}(t) &=  \Id_{E^{(a)}} \oplus \E^{ t} \Id_{F^{(b)}} \oplus  \Id_{G^{(c)}} \in \GL
\\
 S^{ab}_{EF}(t) &= \E^{ t}\Id_{E^{(a)}} \oplus  \Id_{F^{(b)}} \in \GL
\end{align*}
are the  left eruption, right eruption and shearing maps of {\upshape\S \ref{subsect:Eruptions}} and {\upshape\S \ref{subsect:Shearing}}, and where the variations $\Delta\sigma^{ab}$, $\Delta\tau^{abc}$ of generalized Fock-Goncharov invariants are expressed using the framework of Lemmas~{\upshape\ref{lem:LiftTriangleInvariantsToUniversalCover}} and {\upshape\ref{lem:LiftShearsToUniversalCover}}. 
\end{prop}

In \S \ref{subsect:Eruptions} and \S \ref{subsect:Shearing}, we only introduced the maps $L^{abc}_{EFG}(t)$, $R^{abc}_{EFG}(t)$ and $S^{ab}_{EF}(t) $ as projective maps, but we are here  considering them as actual linear maps $\R^n \to \R^n$. 

In the definition of $A(\wt T_i) $, note that the commutativity properties of Lemmas~\ref{lem:EruptionsMainProperties}(5--6) and \ref{lem:ShearingMainProperties}(3)  guarantee that the compositions $\displaystyle\Bigcirc_{a+b=n}  S^{ab}_{EF} (\ )$, $\displaystyle \Bigcirc_{a+b+c=n} L^{abc}_{EFG} ( \  )$ and $ \displaystyle\Bigcirc_{a+b+c=n} R^{abc}_{EFG} ( \  )$ do not depend on the order in which their terms are taken. Also, although  the element $A_0 \in \PGL$ is only determined projectively, conjugating an element of $\GL$ by an element of $\PGL$ gives a well-defined element of $\GL$, so that the  formula given in the statement of Proposition~\ref{prop:VariationSlitheringMap} makes sense. 

\begin{proof}
The proof of Proposition~\ref{prop:VariationSlitheringMap} will take a while, and we split it into several steps. 

We first lift some of the mystery behind the formulas defining the terms $A(\wt T_i)$, by discussing the key property motivating these definitions. 

The generalized Fock-Goncharov invariants are defined so that 
\begin{align*}
\exp  \tau_\rho^{abc} (\wt T_i, x_i) &=  X_{abc}(E_i, F_i, G_i)
 \\
\exp  \sigma_\rho^{a'b'} (\wt T_0, \wt T_i) &=  X_{a'b'} \big( E_i, F_i ; \Sigma_{x_iy_i, u_0v_0}(G_0), G_i \big)
\end{align*}
for every $a$, $b$, $c$, $a'$, $b'\geq 1$ with $a+b+c=a'+b'=n$ (see \S \ref{subsect:DoubleTripleRatiosFockGoncharov} and Fact~\ref{fact:ShearingUsingSlithering}). 

Moving from $\rho$ to $\wh\rho$, Lemmas~\ref{lem:EruptionsMainProperties}(4) and \ref{lem:ShearingMainProperties}(2)  show that there exists a  flag $G_i'\in \Flag$ such that
\begin{align}
\label{eqn:ReconstructSlithering1}
 \exp \tau_{\wh\rho}^{abc} (\wt T_i, x_i)  &=  X_{abc}(E_i, F_i, G_i') 
 \\
 \label{eqn:ReconstructSlithering2}
\exp \sigma_{\wh\rho}^{a'b'} (\wt T_0, \wt T_i)  &= X_{a'b'} (E_i, F_i ; \Sigma_{x_iy_i, u_0v_0}(G_0), G_i')
\end{align}
for every $a$, $b$, $c$, $a'$, $b'\geq 1$ with $a+b+c=a'+b'=n$. In addition, the combination of Proposition~\ref{prop:DoubleTripleRatioClassifyFlagQuadruples} and Lemma~\ref{lem:ProjectiveMapBetweenFlagTriples} shows that $G_i'$ is unique.

\begin{lem}
\label{lem:ExplicitSlithering1}
The map $A(\wt T_i) \colon \R^n \to \R^n $ is the unique unipotent isomorphism that sends the flag $G_i$ to $G_i'$, respects the flag $E_i$ if $xy$ faces the side $x_iz_i$ of $\wt T_i$, and respects  $F_i$  if $xy$ faces the side $y_iz_i$ of $\wt T_i$. 

\end{lem}

\begin{proof}
Consider the case where $xy$ faces the side $x_iz_i$ of $\wt T_i$. The other case is essentially identical. 

All the terms $S_{E_iF_i}^{ab}$, $L_{E_iF_iG_i}^{a'b'c'}$, $S_{E_iG_i}^{a''b''}$ involved in the formula
\begin{align*}
A(\wt T_i)& =
 \Bigcirc_{a+b=n}  S^{ab}_{E_iF_i} (\Delta\sigma^{ab}(\wt T_0, \wt T_{i}))
\circ
 \Bigcirc_{a+b+c=n} L^{abc}_{E_iF_iG_i} \left( \Delta \tau^{abc} (\wt T_i, x_i) \right)
 \\ 
 &\qquad\qquad\qquad
 \circ \Bigcirc_{a+b+c=n} S^{a(b+c)}_{E_i G_i} ( \Delta\tau^{abc}(\wt T_i, x_i))
 \circ 
 \Bigcirc_{a+b=n} S^{ab}_{E_i G_i} (-\Delta\sigma^{ab}(\wt T_0, \wt T_i))
\end{align*}
respect the flag $E_i$, by Lemmas~\ref{lem:EruptionsMainProperties}(1) and \ref{lem:ShearingMainProperties}(1). It follows that their composition $A(\wt T_i)$ also respects $E_i$. 

Similarly, 
$$
 \Bigcirc_{a+b+c=n} S^{a(b+c)}_{E_i G_i} ( \Delta\tau^{abc}(\wt T_i, x_i))
 \circ 
 \Bigcirc_{a+b=n} S^{ab}_{E_i G_i} (-\Delta\sigma^{ab}(\wt T_0, \wt T_i)) (G_i) = G_i. 
$$

If we set
$$G_i'' =  \Bigcirc_{a+b+c=n} L^{abc}_{E_iF_iG_i} \left( \Delta \tau^{abc} (\wt T_i, x_i) \right)(G_i),$$
Lemma~\ref{lem:EruptionsMainProperties}(4) shows that
\begin{align*}
 X_{abc}(E_i, F_i, G_i'') &= 
 \E^{ \Delta \tau^{abc}(\wt T_i, x_i)}  X_{abc}(E_i, F_i, G_i)
 \\
 &= 
 \exp \left( \tau_{\rho}^{abc} (\wt T_i, x_i)+ \Delta \tau^{abc}(\wt T_i, x_i) \right)
  \\
 &= \exp \tau_{\wh\rho}^{abc} (\wt T_i, x_i)
 \end{align*}
 for every $a$, $b$, $c \geq 1$ with $a+b+c=n$. Also, by Lemma~\ref{lem:EruptionsMainProperties}(3), $G_i$ and $G_i''$ have the same line $G_i^{(1)}=G_i^{\prime\prime(1)}$. Therefore, 
 \begin{align*}
 X_{a'b'}(E_i, F_i ; \Sigma_{x_iy_i, u_0v_0}(G_0), G_i'' ) &= X_{a'b'}(E_i, F_i ; \Sigma_{x_iy_i, u_0v_0}(G_0), G_i) 
 \\
 &= \exp \sigma_{\rho}^{a'b'} (\wt T_0, \wt T_i)
 \end{align*}
 for every  $a'$, $b'\geq 1$ with $a'+b'=n$.

 Finally, the last term $ \displaystyle\Bigcirc_{a+b=n}  S^{ab}_{E_i F_i} (\Delta\sigma^{ab}(\wt T_0, \wt T_{i}))$ respects $E_i$ and $F_i$, and sends $G_i''$ to a flag $G_i'''$ such that
 \begin{align*}
 X_{abc}(E_i, F_i, G_i''') &=  X_{abc}(E_i, F_i, G_i'')  =\exp  \tau_{\wh\rho}^{abc} (\wt T_i, x_i)
 \\
 X_{a'b'}(E_i, F_i ; \Sigma_{x_iy_i, u_0v_0}(G_0), G_i''' ) 
 &= \E^{ \Delta\sigma^{a'b'} (\wt T_0, \wt T_i)} X_{a'b'}(E_i, F_i ; \Sigma_{x_iy_i, u_0v_0}(G_0), G_i'' )
 \\
 &=  \exp \left( \sigma_{\rho}^{a'b'} (\wt T_0, \wt T_i)  + \Delta\sigma^{a'b'} (\wt T_0, \wt T_i) \right)
 \\
 &= \exp \sigma_{\wh\rho}^{a'b'} (\wt T_0, \wt T_i)
 \end{align*}
 for every $a$, $b$, $c$, $a'$, $b'\geq 1$ with $a+b+c=a'+b'=n$, using  Proposition~\ref{prop:TripleRatioClassifyFlagTriples} and Lemma~\ref{lem:ShearingMainProperties}(2). Then $G_i'''=G_i'$ by Proposition~\ref{prop:DoubleTripleRatioClassifyFlagQuadruples} and  Lemma~\ref{lem:ProjectiveMapBetweenFlagTriples}, and by definition of $G_i'$ in (\ref{eqn:ReconstructSlithering1}--\ref{eqn:ReconstructSlithering2}). 
 
 This completes the proof that $A(\wt T_i)$ respects $E_i$ and sends $G_i$ to $G_i'$. 
 
 There remains to prove that $A(\wt T_i) $ is unipotent. For this, we will use a basis  $\{ e_1, e_2, \dots, e_n\}$ of $\R^n$ that is adapted to the transverse flag pair $(E_i, F_i)$, in the sense that each space $E_i^{(a)}$ is generated by the first $a$ vectors $e_1$, $e_2$, \dots, $e_a$, and each $F_i^{(b)}$ is spanned by the last $b$ vectors $e_{n-b+1}$, $e_{n-b+2}$, \dots, $e_n$. 

By its definition, the shearing map 
$$ 
S^{ab}_{E_iF_i} (t) = \E^t \Id_{E_i^{(a)}} \oplus \Id_{F_i^{(b)}} 
$$ 
is diagonal in this basis, with its first $a$ entries equal to $\E^t$ and the remaining ones equal to~$1$. 

The left eruption map
$$
L^{abc}_{E_iF_iG_i} \left(t \right) = \E^{-t} \Id_{E_i^{(a)}} \oplus \Id_{F_i^{(b)}}  \oplus \Id_{G_i^{(c)}} 
$$
sends $e_j$ to $\E^{-t} e_j$ if $j\leq a$ and sends $e_j$ to $e_j$ if $j\geq n-c+1$. For $a<j \leq n-c$, use the direct sum decomposition $\R^n= E_i^{(a)} \oplus F_i^{(b)}  \oplus G_i^{(c)}$ and write $e_j$ as $e_j = v_a+v_b+v_c$ with 
$v_a \in E_i^{(a)} $, $v_b \in F_i^{(b)}$ and $v_c \in G_i^{(c)}$. Then
$$
L^{abc}_{E_iF_iG_i} (t ) (e_j) = \E^{-t} v_a + v_b + v_c = (\E^{-t}-1) v_a + e_j.
$$
Since $v_a \in E_i^{(a)} $ is a linear combination of $e_k$ with $k\leq a<j$, this proves that the matrix of $L^{abc}_{E_iF_iG_i} \left(t \right)$ in the basis $\{ e_1, e_2, \dots, e_n\}$ is upper triangular, with its first $a$ diagonal entries equal to $\E^{-t}$ and the remaining diagonal entries equal to 1.

Finally, the shearing map 
$$ 
S^{ab}_{E_iG_i} (t) = \E^t \Id_{E_i^{(a)}} \oplus \Id_{G_i^{(b)}} 
$$ 
sends $e_j$ to $\E^t e_j$ when $j\leq a$. When $j>a$, we again decompose $e_j$ as $e_j=v_a + v_b$ with $v_a \in E_i^{(a)} $ and $v_b \in G_i^{(b)}$. Then
$$
S^{ab}_{E_{i} G_{i}} (t)(e_j) =  \E^{t} v_a + v_b = (\E^{t}-1) v_a + e_j.
$$
This again proves that $S^{ab}_{E_{i} G_{i}} (t)$ is upper triangular in the basis $\{ e_1, e_2, \dots, e_n\}$, with its first $a$ diagonal entries equal to $\E^t$, the remaining diagonal entries equal to $1$. 

If we apply these computations to each term in
\begin{align*}
A(\wt T_i) &= 
\Bigcirc_{a+b=n}  S^{ab}_{E_iF_i} (\Delta\sigma^{ab}(\wt T_0, \wt T_{i}))
\circ
 \Bigcirc_{a+b+c=n} L^{abc}_{E_iF_iG_i} \big( \Delta \tau^{abc} (\wt T_i, x_i) \big)
 \\ 
 &\qquad\qquad
 \circ \Bigcirc_{a+b+c=n} S^{a(b+c)}_{E_iG_i} ( \Delta\tau^{abc}(\wt T_i, x_i))
 \circ 
 \Bigcirc_{a+b=n} S^{ab}_{E_iG_i} (-\Delta\sigma^{ab}(\wt T_0, \wt T_i))
\end{align*}
we see that the matrix of $A(\wt T_i)$ in the basis $\{ e_1, e_2, \dots, e_n\}$ is upper diagonal, with its $j$--th diagonal entry equal to
$$
\prod_{\substack{a+b=n\\ a\geq j}} \E^{\Delta\sigma^{ab}(\wt T_0, \wt T_{i})}
 \prod_{\substack{a+b+c=n\\ a\geq j}} \E^{- \Delta \tau^{abc} (\wt T_i, x_i) } 
 \prod_{\substack{a+b+c=n\\ a\geq j}} \E^{ \Delta \tau^{abc} (\wt T_i, x_i) }
 \prod_{\substack{a+b=n\\ a\geq j}} \E^{-\Delta\sigma^{ab}(\wt T_0, \wt T_{i})}
 =1.
$$

Therefore, $A(\wt T_i)$ is a unipotent isomorphism that respects the flag $E_i$ and sends $G_i$ to $G_i'$. The unipotent condition guarantees that this map is unique. This completes the  proof of Lemma~\ref{lem:ExplicitSlithering1} when $xy$ faces the side $x_iz_i$  of $\wt T_i$. 

The proof in the other case, where  $xy$ faces the side $y_iz_i$  of $\wt T_i$, is essentially identical. 
\end{proof}

As in the construction of the slithering maps, let $x_i'y_i'$ be the side of $\wt T_i$ that faces $xy$. In particular, either $x_i'=x_i$ and $y_i'=z_i$, or $x_i'=z_i$ and $y_i'=y_i$. 

The next step in the proof of Proposition~\ref{prop:VariationSlitheringMap} is the following computation. 
\begin{lem}
\label{lem:ExplicitSlithering2}
 $$
 \wh\Sigma_{x_i'y_i', x_iy_i} = \left( \Sigma_{u_0v_0, x_iy_i}^{-1} \circ A_0^{-1} \circ \wh\Sigma_{u_0v_0, x_iy_i} \right)^{-1} \circ A(\wt T_i) \circ \Sigma_{x_i'y_i', x_iy_i}  \circ  \left( \Sigma_{u_0v_0, x_iy_i}^{-1}  \circ  A_0^{-1} \circ \wh\Sigma_{u_0v_0, x_iy_i} \right) .
 $$
\end{lem}
\begin{proof}
Let us restrict attention to the case where $xy$ faces the side $x_iz_i$  of $\wt T_i$, namely where $x_i'=x_i$ and $y_i'=z_i$. As usual, the other case will be almost identical.

Let 
$$
\Sigma = \left( \Sigma_{u_0v_0, x_iy_i}^{-1} \circ A_0^{-1} \circ \wh\Sigma_{u_0v_0, x_iy_i} \right)^{-1} \circ A(\wt T_i) \circ \Sigma_{x_i'y_i', x_iy_i}  \circ  \left( \Sigma_{u_0v_0, x_iy_i}^{-1}  \circ  A_0^{-1} \circ \wh\Sigma_{u_0v_0, x_iy_i} \right) 
 $$
  denote the right hand side of the equation. Lemma~\ref{lem:ExplicitSlithering1} shows that $A(\wt T_i)$ respects the flag $E_i$ and is unipotent, and   $\Sigma_{x_i'y_i', x_iy_i} $   also  respects $E_i$ and is unipotent by definition of slithering maps. It follows that $A(\wt T_i)\circ \Sigma_{x_i'y_i', x_iy_i}$  respects $E_i$ and is unipotent. Since $\Sigma$ is conjugate to $A(\wt T_i) \circ \Sigma_{x_i'y_i', x_iy_i}$, it therefore is also unipotent. 

In an obvious adaptation of the notation to $\wh\rho$, set $\wh E_i= \F_{\wh\rho}(x_i)$, $\wh F_i= \F_{\wh\rho}(y_i)$, $\wh G_i= \F_{\wh\rho}(z_i)\in \Flag$ for $i=1$, $2$, \dots, $i_0$, and $\wh E_0= \F_{\wh\rho}(u_0)$, $\wh F_0= \F_{\wh\rho}(v_0)$, $\wh G_0= \F_{\wh\rho}(w_0)$. Similarly, let $ E_i'= \F_{\rho}(x_i')$, $ F_i'= \F_{\rho}(y_i')$, $\wh E_i'= \F_{\wh\rho}(x_i')$, $\wh F_i'= \F_{\wh\rho}(y_i')$. 
Then,
\begin{align*}
 \left( \Sigma_{u_0v_0, x_iy_i}^{-1}  \circ  A_0^{-1} \circ \wh\Sigma_{u_0v_0, x_iy_i} \right)  (\wh E_i)
 &= \Sigma_{u_0v_0, x_iy_i}^{-1}  \circ  A_0^{-1} (\wh E_0) 
 = \Sigma_{u_0v_0, x_iy_i}^{-1} ( E_0)
 = E_i
 \\
  \left( \Sigma_{u_0v_0, x_iy_i}^{-1}  \circ  A_0^{-1} \circ \wh\Sigma_{u_0v_0, x_iy_i} \right)  (\wh F_i)
 &= \Sigma_{u_0v_0, x_iy_i}^{-1}  \circ  A_0^{-1} (\wh F_0) 
 = \Sigma_{u_0v_0, x_iy_i}^{-1} ( F_0)
 = F_i.
\end{align*}
 Since we are in the case where $xy$ faces the side $x_iz_i$  of $\wt T_i$,   $E_i'=E_i$ and  it follows that
\begin{align*}
 \Sigma(\wh E_i) 
 &= \left( \Sigma_{u_0v_0, x_iy_i}^{-1} \circ A_0^{-1} \circ \wh\Sigma_{u_0v_0, x_iy_i} \right)^{-1} \circ A(\wt T_i) \circ \Sigma_{x_i'y_i', x_iy_i} (E_i)
 \\
 &= \left( \Sigma_{u_0v_0, x_iy_i}^{-1} \circ A_0^{-1} \circ \wh\Sigma_{u_0v_0, x_iy_i} \right)^{-1} \circ A(\wt T_i) (E_i) 
 \\
 &= \left( \Sigma_{u_0v_0, x_iy_i}^{-1} \circ A_0^{-1} \circ \wh\Sigma_{u_0v_0, x_iy_i} \right)^{-1} (E_i) 
 =\wh E_i. 
\end{align*}
Finally, since $F_i'=G_i$ in the case considered,
\begin{align*}
 \Sigma(\wh F_i) 
 &= \left( \Sigma_{u_0v_0, x_iy_i}^{-1} \circ A_0^{-1} \circ \wh\Sigma_{u_0v_0, x_iy_i} \right)^{-1} \circ A(\wt T_i) \circ \Sigma_{x_i'y_i', x_iy_i} (F_i)
 \\
 &= \left( \Sigma_{u_0v_0, x_iy_i}^{-1} \circ A_0^{-1} \circ \wh\Sigma_{u_0v_0, x_iy_i} \right)^{-1} \circ A(\wt T_i) (G_i) 
 \\
 &= \left( \Sigma_{u_0v_0, x_iy_i}^{-1} \circ A_0^{-1} \circ \wh\Sigma_{u_0v_0, x_iy_i} \right)^{-1} (G_i')
\end{align*}
where the flag $G_i'$ is defined as in Lemma~\ref{lem:ExplicitSlithering1}. 

Therefore, by application of the linear isomorphism $\Sigma_{u_0v_0, x_iy_i}^{-1} \circ A_0^{-1} \circ \wh\Sigma_{u_0v_0, x_iy_i}$ and remembering that the double-ratios of four flags depend only on the lines of the last two flags, the flag $\Sigma(\wh F_i)$ is such that 
\begin{align*}
 X_{abc}\big(\wh E_i, \wh F_i, \Sigma(\wh F_i)\big)
 &=  X_{abc}( E_i,  F_i, G_i')
= \exp \tau_{\wh\rho}^{abc} (\wt T_i, x_i)
 \\
 &=  X_{abc}\big(\wh E_i, \wh F_i, \wh G_i \big)
 \\
 X_{a'b'}\big(\wh E_i, \wh F_i ; \wh \Sigma_{x_iy_i, u_0v_0}(\wh G_0), \Sigma(\wh F_i) \big) 
 &=  X_{a'b'} \big( E_i,  F_i ; \Sigma_{u_0v_0, x_iy_i}^{-1} \circ A_0^{-1}( \wh G_0), G_i' \big) 
 \\
 &=  X_{a'b'} \big( E_i,  F_i ; \Sigma_{x_iy_i, u_0v_0} ( G_0), G_i' \big) 
= \exp \sigma_{\rho}^{a'b'} (\wt T_0, \wt T_i)
  \\
  &= X_{a'b'}\big(\wh E_i, \wh F_i ; \wh \Sigma_{x_iy_i, u_0v_0}(\wh G_0), \wh G_i \big) 
 \end{align*}
 for every $a$, $b$, $c$, $a'$, $b'\geq 1$ with $a+b+c=a'+b'=n$. By Proposition~\ref{prop:DoubleTripleRatioClassifyFlagQuadruples} and Lemma~\ref{lem:ProjectiveMapBetweenFlagTriples}, this proves that $\Sigma(\wh F_i) =\wh G_i = \wh F_i'$. 
 
 Therefore, $\Sigma$ is a unipotent map respecting $\wh E_i = \wh E_i'$ and sending $\wh F_i$ to $\wh F_i'$. By uniqueness of such a map, $\Sigma$ is equal to the slithering map $ \wh \Sigma_{x_i'y_i', x_iy_i}$. 
 
 This concludes the proof of Lemma~\ref{lem:ExplicitSlithering2} in the case where $xy$ faces the side $x_iz_i$  of $\wt T_i$. The other case is essentially identical.
\end{proof}

Our next step is to borrow two estimates from \cite{BonDre2}. To quantify the ``width'' of a component $\wt T$ of $\wt S-\wt\lambda$ separating $xy$ from $\wt T_0$, choose a geodesic arc $k \subset \wt S$ joining a point of $\wt T_0$ to a point of $xy$. Then, for every component $\wt T$ of $\wt S-\wt\lambda$ separating $xy$ from $\wt T_0$, we can consider the length $\ell(k\cap \wt T)$ of its intersection with $k$. Let $\Sigma_{x_{\wt T}^{\phantom{\prime}} y_{\wt T}^{\phantom{\prime}} , x_{\wt T}'y_{\wt T}'} \in \GL$ be the slithering map associated to such a component, where  $x_{\wt T}y_{\wt T}$ is the side of $\wt T$ that faces $\wt T_0$ and where $x_{\wt T}'y_{\wt T}'$ is the side facing $xy$. 

\begin{lem}
\label{lem:ElementarySlitheringBoundedByGapLength}
 There exists a number $\nu>0$ such that
 $$
\Sigma_{x_{\wt T}^{\phantom{\prime}} y_{\wt T}^{\phantom{\prime}} , x_{\wt T}'y_{\wt T}'}  = \Id_{\R^n} + O \left( \ell(k\cap \wt T)^\nu \right)
 $$
  for every component $\wt T$ of $\wt S-\wt\lambda$ separating $xy$ from $\wt T_0$. 
  
  In addition, the exponent $\nu$ and the constant hidden in the Landau symbol $O\left( \ \right)$ depend only on a compact subset of $\wt S$ containing the arc $k$, and on a compact subset containing $\rho$ in the space of Hitchin representations (assuming the geodesic lamination $\lambda$  given). 
\end{lem}
\begin{proof}
 See \cite[Lem.~5.2]{BonDre2}. 
\end{proof}

\begin{lem}
\label{lem:HolderSumOfGapLengths}
 As $\wt T$ ranges over all components of $\wt S-\wt\lambda$ that separate $xy$ from $\wt T_0$, the sum
 $$
 \sum_{\wt T}  \ell(k\cap \wt T)^\nu
 $$
 is convergent for every $\nu>0$. 
 
 More precisely, for every $\nu>\nu'>0$,
  $$
 \sum_{\ell(k\cap \wt T) <\epsilon}  \ell(k\cap \wt T)^\nu = O (\epsilon^{\nu'})
 $$
for every $\epsilon>0$, where the sum is taken over all $\wt T$ with $\ell(k\cap \wt T) <\epsilon$ and where the constant hidden in the Landau symbol $O(\ )$ depends only on a compact subset of $\wt S$ containing the arc $k$ and on the exponents $\nu$, $\nu'$. 
\end{lem}
\begin{proof}
 See \cite[Lem.~5.3]{BonDre2}. 
\end{proof}

\begin{lem}
\label{lem:ElementarySlitheringBoundedByGapLength2}
For the exponent $\nu$ of  Lemma~{\upshape\ref{lem:ElementarySlitheringBoundedByGapLength}}, 
 $$
 A(\wt T_i)=   \Id_{\R^n} + O \left( \ell(k\cap \wt T_i)^\nu \right)
 $$
 where the constant hidden in the Landau symbol $O(\ )$ depends only on a compact subset of $\wt S$ containing the arc $k$, and on a compact subset of the space of Hitchin representations containing both $\rho$ and $\wh\rho$.
\end{lem}
\begin{proof}
 We can rewrite the equality of Lemma~\ref{lem:ExplicitSlithering2} as 
 $$
 A(\wt T_i) =  \left( \Sigma_{u_0v_0, x_iy_i}^{-1} \circ A_0^{-1} \circ \wh\Sigma_{u_0v_0, x_iy_i} \right) \circ  \wh\Sigma_{x_i'y_i', x_iy_i} \circ 
 \left( \Sigma_{u_0v_0, x_iy_i}^{-1} \circ A_0^{-1} \circ \wh\Sigma_{u_0v_0, x_iy_i} \right)^{-1} \circ \Sigma_{x_i'y_i', x_iy_i}^{-1} .
  $$
  
  By Lemma~\ref{lem:ElementarySlitheringBoundedByGapLength}, $ \wh\Sigma_{x_i'y_i', x_iy_i} =  \Id_{\R^n} + O \left( \ell(k\cap \wt T_i)^\nu \right)$. 
  
  Also, the slithering map  $\Sigma_{u_0v_0, x_iy_i}$ is defined as the (infinite) composition of the elementary slithering maps $\Sigma_{\wt T}$ associated to the components $\wt T$ separating $x_iy_i$ from $\wt T_0$. By the combination of Lemmas~\ref{lem:ElementarySlitheringBoundedByGapLength} and \ref{lem:HolderSumOfGapLengths}, it follows that  $\Sigma_{u_0v_0, x_iy_i}$ is uniformly bounded in $\GL$. The same holds for $\Sigma_{u_0v_0, x_iy_i}^{-1}$, $\wh\Sigma_{u_0v_0, x_iy_i}$ and $\wh\Sigma_{u_0v_0, x_iy_i}^{-1}$. Therefore, the terms  $ \Sigma_{u_0v_0, x_iy_i}^{-1} \circ A_0^{-1} \circ \wh\Sigma_{u_0v_0, x_iy_i} $ and their inverses are uniformly bounded in $\GL$. it follows that 
 $$
  \left( \Sigma_{u_0v_0, x_iy_i}^{-1} \circ A_0^{-1} \circ \wh\Sigma_{u_0v_0, x_iy_i} \right) \circ  \wh\Sigma_{x_i'y_i', x_iy_i} \circ 
 \left( \Sigma_{u_0v_0, x_iy_i}^{-1} \circ A_0^{-1} \circ \wh\Sigma_{u_0v_0, x_iy_i} \right)^{-1} =  \Id_{\R^n} + O \left( \ell(k\cap \wt T_i)^\nu \right).
 $$
 Since  $ \Sigma_{x_i'y_i', x_iy_i} =  \Id_{\R^n} + O \left( \ell(k\cap \wt T_i)^\nu \right)$  by another application of Lemma~\ref{lem:ElementarySlitheringBoundedByGapLength}, this concludes the proof of Lemma~\ref{lem:ElementarySlitheringBoundedByGapLength2}. 
\end{proof}

We now return to the proof of Proposition~\ref{prop:VariationSlitheringMap}. 
By another manipulation of the equality of Lemma~\ref{lem:ExplicitSlithering2},
$$
A_0^{-1} \circ \wh\Sigma_{x_i'y_i',u_0v_0}  \circ  A_0 \circ \Sigma_{x_i'y_i', u_0v_0}^{-1} = \left(A_0^{-1} \circ  \wh \Sigma_{x_iy_i, u_0v_0} \circ A_0 \circ \Sigma_{x_iy_i, u_0v_0}^{-1} \right) \circ A(\wt T_i) .
 $$
 
 Similarly, let $B_i \in \GL$ be such that
$$
A_0^{-1} \circ \wh\Sigma_{x_{i+1}y_{i+1}, u_0v_0}  \circ  A_0 \circ \Sigma_{x_{i+1}y_{i+1}, u_0v_0}^{-1} = \left(A_0^{-1} \circ  \wh \Sigma_{x_i'y_i',u_0v_0} \circ A_0 \circ \Sigma_{x_i'y_i',u_0v_0}^{-1} \right) \circ B_i .
 $$
 Namely, 
 $$
 B_i =   \Sigma_{x_i'y_i',u_0v_0} \circ A_0^{-1} \circ \wh \Sigma_{x_i'y_i',u_0v_0}^{-1}\circ \wh\Sigma_{x_{i+1}y_{i+1}, u_0v_0}  \circ  A_0 \circ \Sigma_{x_{i+1}y_{i+1}, u_0v_0}^{-1}.
 $$
 
 Then, 
 $$
 A_0^{-1} \circ \wh\Sigma_{xy, u_0v_0}  \circ  A_0 \circ \Sigma_{xy, u_0v_0}^{-1} = B_0 \circ A(\wt T_1) \circ B_1 \circ A(\wt T_2) \circ B_2 \circ \dots\circ A(\wt T_{i_0}) \circ B_{i_0}
 $$
 if we define
\begin{align*}
  B_0 &=    \wh\Sigma_{x_{1}y_{1}, u_0v_0}  \circ  A_0 \circ \Sigma_{x_{1}y_{1}, u_0v_0}^{-1}
  \\
   B_{i_0} &=   \Sigma_{x_{i_0}'y_{i_0}',u_0v_0} \circ A_0^{-1} \circ \wh \Sigma_{x_{i_0}'y_{i_0}',u_0v_0}^{-1}\circ \wh\Sigma_{xy, u_0v_0}  \circ  A_0 \circ \Sigma_{xy, u_0v_0}^{-1}.
\end{align*}

\begin{lem}
\label{lem:SmallSlitheringNegligibleInLimit}
 As the  family $\{ \wt T_1, \wt T_2, \dots, \wt T_{i_0}\}$ tends to the set of all components of $\wt S - \wt \lambda$ separating $xy$ from $\wt T_0$, the two terms
 $$
  B_0 \circ A(\wt T_1) \circ B_1 \circ A(\wt T_2) \circ B_2 \circ \dots\circ A(\wt T_{i_0}) \circ B_{i_0}
 $$
 and
 $$
 A(\wt T_1) \circ A(\wt T_2) \circ \dots \circ A(\wt T_{i_0})
 $$
 have the same limit.  
\end{lem}

\begin{proof}
We first rewrite
\begin{align*}
 B_i &=   \Sigma_{x_i'y_i',u_0v_0} \circ A_0^{-1} \circ \wh \Sigma_{x_i'y_i',u_0v_0}^{-1}\circ \wh\Sigma_{x_{i+1}y_{i+1}, u_0v_0}  \circ  A_0 \circ \Sigma_{x_{i+1}y_{i+1}, u_0v_0}^{-1}
 \\
&=  \left( \Sigma_{x_i'y_i',u_0v_0} \circ A_0^{-1} \circ \wh \Sigma_{x_i'y_i',u_0v_0}^{-1} \right) \circ \wh\Sigma_{x_{i+1}y_{i+1}, x_i'y_i'} 
\\
&\qquad\qquad\qquad\qquad
\circ  \left( \Sigma_{x_i'y_i',u_0v_0} \circ A_0^{-1} \circ \wh \Sigma_{x_i'y_i',u_0v_0}^{-1} \right)^{-1} 
 \circ \Sigma_{x_{i+1}y_{i+1}, x_i'y_i'}^{-1} .
\end{align*}

The same argument as in the proof of Lemma~\ref{lem:ElementarySlitheringBoundedByGapLength2} then gives that
$$
B_i = \Id_{\R^n} + O\left( \sum_{\wt T'} \ell(k \cap \wt T')^\nu \right)
$$
where $\wt T'$ ranges over all components of $\wt S- \wt\lambda$ that separate $\wt T_i$ from $\wt T_{i+1}$, if $i<i_0$, and over all components of $\wt S- \wt\lambda$ that separate $\wt T_{i_0}$ from $xy$ if $i=i_0$.

Now, consider
$$
C_i = A(\wt T_1) \circ A(\wt T_2) \circ \dots \circ A(\wt T_i) \circ B_i \circ A(\wt T_{i+1}) \circ B_{i+1} \circ \dots \circ A(\wt T_{i_0}) \circ B_{i_0},
$$
obtained from 
$$C_0= B_0 \circ A(\wt T_1) \circ B_1 \circ A(\wt T_2) \circ B_2 \circ \dots\circ  A(\wt T_{i_0}) \circ B_{i_0}
$$ by omitting the $B_j$ with $j<i$. Then, using Lemmas~\ref{lem:ElementarySlitheringBoundedByGapLength2} and \ref{lem:HolderSumOfGapLengths} as well as the above estimate for the $B_j$ to bound the terms $ A(\wt T_1) \circ A(\wt T_2) \circ \dots \circ A(\wt T_i)$ and $  A(\wt T_{i+1}) \circ B_{i+1} \circ \dots \circ A(\wt T_{i_0}) \circ B_{i_0}$,
\begin{align*}
 C_{i} -C_{i+1} 
 &= A(\wt T_1) \circ A(\wt T_2) \circ \dots \circ A(\wt T_i) \circ (B_i - \Id_{\R^n}) \circ A(\wt T_{i+1}) \circ B_{i+1} \circ \dots \circ A(\wt T_{i_0}) \circ B_{i_0}
 \\
 &= O\left( \sum_{T} \ell(k \cap \wt T)^\nu \right)
\end{align*}
where the sum is over all the components $\wt T$ of $\wt S- \wt\lambda$ that separate $\wt T_i$ from $\wt T_{i+1}$. 

By induction, we conclude that
$$
 C_0 -C_{i_0}  = O\left( \sum_{T\not\in \mathcal T} \ell(k \cap \wt T)^\nu \right)
$$
where the sum is now over all the components $\wt T$ of $\wt S- \wt\lambda$ that separate $xy$ from  $\wt T_0$ and do not belong to the family $\mathcal T = \{ \wt T_1, \wt T_2, \dots, \wt T_{i_0} \}$. By Lemma~\ref{lem:HolderSumOfGapLengths}, $C_0 -C_{i_0} $ therefore converges to 0 as the family $\mathcal T$ tends to the set of all components of $\wt S- \wt\lambda$ separating $xy$ from $\wt T_0$. Since
\begin{align*}
 C_0 &= B_0 \circ A(\wt T_1) \circ B_1 \circ A(\wt T_2) \circ \dots\circ A(\wt T_{i_0}) \circ B_{i_0}
 \\
 C_{i_0} &=  A(\wt T_1) \circ A(\wt T_2) \circ \dots \circ A(\wt T_{i_0}),
\end{align*}
this completes the proof of Lemma~\ref{lem:SmallSlitheringNegligibleInLimit}. 
\end{proof}

Because 
$$
B_0 \circ A(\wt T_1) \circ B_1 \circ A(\wt T_2) \circ \dots\circ A(\wt T_{i_0}) \circ B_{i_0} = A_0^{-1} \circ \wh\Sigma_{xy, u_0v_0}  \circ  A_0 \circ \Sigma_{xy, u_0v_0}^{-1} 
$$
is actually independent of the family $\mathcal T = \{ \wt T_1, \wt T_2, \dots, \wt T_{i_0} \}$, this shows that $A_0^{-1} \circ \wh\Sigma_{xy, u_0v_0}  \circ  A_0 \circ \Sigma_{xy, u_0v_0}^{-1} $ is the limit of $A(\wt T_1) \circ A(\wt T_2) \circ \dots \circ A(\wt T_{i_0})$ as $\mathcal T$ tends to the set of all components of $\wt S- \wt\lambda$ separating $xy$ from $\wt T_0$.

Equivalently, $\wh\Sigma_{xy, u_0v_0}$ is the limit of 
$$
A_0 \circ A(\wt T_1) \circ A(\wt T_2) \circ \dots \circ A(\wt T_{i_0}) \circ \Sigma_{xy, u_0v_0} \circ A_0^{-1}
$$
 as the family $ \{ \wt T_1, \wt T_2, \dots, \wt T_{i_0} \}$ tends to the set of all components of $\wt S- \wt\lambda$ separating $xy$ from $\wt T_0$. This completes the proof of Proposition~\ref{prop:VariationSlitheringMap}. 
\end{proof}

\subsection{Infinitesimal variation of flag maps}
\label{subsect:InfinitesimalVariationFlagMap}

Let $t\mapsto [\rho_t] \in \Hit(S)$ be a smooth curve in the Hitchin component, tangent to a vector $V \in T_{[\rho]} \Hit(S)$ at the point $[\rho] = [\rho_0]$. Our overall goal is to describe the Weil class $[c_V] \in H^1(S; \sln_{\Ad\rho})$ of this tangent vector $V$ in terms of the directional derivatives
\begin{align*}
 \dot \tau^{abc}_V(T,x) &= \frac{d}{dt} \tau^{abc}_{[\rho_t]}(T,x) _{|t=0}
 \\
 \dot \sigma^{ab}_V(e) &= \frac{d}{dt} \sigma^{ab}_{[\rho_t]}(e) _{|t=0}
\end{align*}
of its generalized Fock-Goncharov invariants. In order to implement the methods of \S \ref{subsect:SimplicialWeil}, we will rely on the computations of \S \ref{subsect:VariationSlitheringMap} to determine the corresponding infinitesimal variation of the flag map $\F_{\rho_t} \colon \bdry \to \PGL$, namely the tangent vector
$$
\frac{d}{dt} \F_{\rho_t} (w)_{|t=0} \in T_{ \F_{\rho} (w)}  \Flag
$$ 
for $w\in \bdry$. Actually, we will not need this computation for all $w\in \bdry$, just for those $w$ in the dense subset consisting of all vertices of the components of $\wt S- \wt\lambda$.

The formulas of \S \ref{subsect:VariationSlitheringMap} involved the eruption and shearing maps 
\begin{align*}
L^{abc}_{EFG}(t) &= \E^{- t}\Id_{E^{(a)}} \oplus  \Id_{F^{(b)}} \oplus  \Id_{G^{(c)}}
 \\
R^{abc}_{EFG}(t) &=  \Id_{E^{(a)}} \oplus \E^{ t} \Id_{F^{(b)}} \oplus  \Id_{G^{(c)}}
\\
 S^{ab}_{EF}(t) &= \E^{ t}\Id_{E^{(a)}} \oplus  \Id_{F^{(b)}} .
\end{align*}
While these were well-defined linear isomorphisms $\R^n \to \R^n$ in  \S \ref{subsect:VariationSlitheringMap}, it is now convenient to resume considering them projectively, as elements of $\PGL$, and to take the derivatives
\begin{align}
\label{eqn:InfinitesimalLeftEruptionDef}
 \dot L^{abc}_{EFG} &= \frac{d}{dt}  L^{abc}_{EFG}(t) _{|t=0}= \tfrac{-b-c}{n} \Id_{E^{(a)}} \oplus \tfrac{a}{n}  \Id_{F^{(b)}} \oplus  \tfrac{a}{n} \Id_{G^{(c)}} \in \mathfrak{pgl}_n(\R)=\sln
 \\
 \label{eqn:InfinitesimalRightEruptionDef}
 \dot R^{abc}_{EFG} &=  \frac{d}{dt}  R^{abc}_{EFG}(t) _{|t=0}= \tfrac{-b}{n} \Id_{E^{(a)}} \oplus\tfrac{a+c}{n} \Id_{F^{(b)}} \oplus  \tfrac{-b}{n} \Id_{G^{(c)}}\in \sln
 \\
 \label{eqn:InfinitesimalShearDef}
 \dot  S^{ab}_{EF} &= \frac{d}{dt}  S^{ab}_{EF}(t)_{|t=0} =  \tfrac{b}{n} \Id_{E^{(a)}} \oplus  \tfrac{-a}{n} \Id_{F^{(b)}} \in \sln.
\end{align}
Note that, in this context, one needs to rescale the original formulas to represent $ L^{abc}_{EFG}(t)$, $ R^{abc}_{EFG}(t)$, $ S^{ab}_{EF}(t) \in \PGL$ by elements of $\SL$, in order to take advantage of the isomorphism between the Lie algebras $\mathfrak{pgl}_n(\R)$ and $\sln$.

For a flag $F\in \Flag$ and $\dot A \in \sln$, let
$$
\dot A F = \frac{d}{dt} A_t F_{|t=0} \in T_F \Flag
$$
denote the vector tangent to the curve $t \mapsto A_t F \in \Flag$ defined by any curve $t \mapsto A_t \in \PGL$ tangent to $\dot A = \frac{d}{dt} A_t{}_{|t=0}$ at $A_0 = \Id_{\R^n}$. 

To simplify the formulas, we normalize the Hitchin representations $\rho_t\colon \pi_1(S) \to \PGL$ as follows. Arbitrarily choose a base component $\wt T_0$ among all components of $\wt S - \wt\lambda$. Index its vertices as $x_{\wt T_0}$, $y_{\wt T_0}$, $z_{\wt T_0} \in \bdry$ counterclockwise in this order around $\wt T_0$. By Lemma~\ref{lem:ProjectiveMapBetweenFlagTriples}, there is a projective map $A_t\in \PGL$ sending the flag $\F_{\rho_t}(x_{\wt T_0})$ to $\F_\rho(x_{\wt T_0})$, the flag $\F_{\rho_t}(y_{\wt T_0})$ to $\F_\rho(y_{\wt T_0})$, and the line $\F_{\rho_t}(z_{\wt T_0})^{(1)}$ to $\F_\rho(z_{\wt T_0})^{(1)}$. After conjugating $\rho_t$ with $A_t$, which replaces the flag map $\F_{\rho_t}$ with $A_t \circ \F_{\rho_t}$, we can therefore arrange that  the flags $\F_{\rho_t}(x_{\wt T_0})$, $\F_{\rho_t}(y_{\wt T_0}) \in \Flag$  and the line $\F_{\rho_t}(z_{\wt T_0})^{(1)}\in \mathbb{RP}^{n-1}$ are independent of $t$, namely respectively equal to $\F_{\rho}(x_{\wt T_0})$, $\F_{\rho}(y_{\wt T_0})$, $\F_{\rho}(z_{\wt T_0})^{(1)}$.

Finally, for every component $\wt T$ of $\wt S-\wt\lambda$ that is different from the base component $\wt T_0$, we counterclockwise index the vertices of $\wt T$ as $x_{\wt T}$, $y_{\wt T}$, $z_{\wt T} \in \bdry$ in such a way that $x_{\wt T}y_{\wt T}$ is the side of $\wt T$ that faces $\wt T_0$. Then, for each component $\wt T$ of $\wt S-\wt\lambda$ (including $\wt T_0$), we consider the flags $E_{\wt T} = \F_\rho(x_{\wt T})$, $F_{\wt T} = \F_\rho(y_{\wt T})$, $G_{\wt T} = \F_\rho(z_{\wt T})\in \Flag$.

\begin{prop}
\label{prop:InfinitesimalGapFormula}
 Let $\wt T$ be a component of $\wt S - \wt \lambda$ that is distinct from the base component $\wt T_0$. Then, for every vertex $w\in \{ x_{\wt T}, y_{\wt T}, z_{\wt T} \}$ of $\wt T$, 
\begin{align*}
 \frac{d}{dt} \F_{\rho_t}(w)_{|t=0}  &=  \dot B(\wt T_0)  \F_\rho(w)
 + \sum_{\wt T' \text{ between } \wt T_0 \text{ and } \wt T} \dot A(\wt T')  \F_\rho(w)
 + \dot B(\wt T)  \F_\rho(w)
\end{align*}
where the sum is over all components $\wt T'$ of $\wt S-\wt\lambda$ separating $\wt T_0$ from $\wt T$, where 
$$
\dot A(\wt T') =
\begin{cases}
\displaystyle
 \sum_{a+b+c=n} \dot\tau^{abc}_V(\wt T', x_{\wt T'}) \left( \dot L_{E_{\wt T'}F_{\wt T'}G_{\wt T'}}^{abc} + \dot S^{a(b+c)}_{E_{\wt T'}G_{\wt T'}}  \right)
 \\
 \displaystyle
 \qquad\qquad + \sum_{a+b=n} \dot \sigma_V^{ab}(\wt T_0, \wt T') \left(  \dot S^{ab}_{E_{\wt T'}F_{\wt T'}} -  \dot S^{ab}_{E_{\wt T'}G_{\wt T'}}  \right)
 \\
 \qquad\qquad  \qquad\qquad  \qquad\qquad  \qquad\qquad  \text{if } \wt T \text{ faces the side } x_{\wt T'}z_{\wt T'}\text{ of }\wt T'
 \\
 \displaystyle
  \sum_{a+b+c=n} \dot\tau_V^{abc}(\wt T', x_{\wt T'}) \left( \dot R_{E_{\wt T'}F_{\wt T'}G_{\wt T'}}^{abc} + \dot S^{(a+c)b}_{G_{\wt T'}F_{\wt T'}}  \right)
 \\
 \displaystyle
 \qquad\qquad + \sum_{a+b=n} \dot \sigma_V^{ab}(\wt T_0, \wt T') \left(  \dot S^{ab}_{E_{\wt T'}F_{\wt T'}} -  \dot S^{ab}_{G_{\wt T'}F_{\wt T'}}  \right)
 \\
 \qquad\qquad  \qquad\qquad  \qquad\qquad  \qquad\qquad  \text{if } \wt T \text{ faces the side } y_{\wt T'}z_{\wt T'}\text{ of } \wt T',
\end{cases}
$$
where
$$
\dot B(\wt T_0) =
\begin{cases}
0
& \text{if } \wt T \text{ faces the side } x_{\wt T_0}y_{\wt T_0}\text{ of }\wt T_0
\\
\displaystyle
 \sum_{a+b+c=n} \dot\tau_V^{abc}(\wt T_0, x_{\wt T_0})  \dot L_{E_{\wt T_0}F_{\wt T_0}G_{\wt T_0}}^{abc} 
& \text{if } \wt T \text{ faces the side } x_{\wt T_0}z_{\wt T_0}\text{ of } \wt T_0
 \\
 \displaystyle
 \sum_{a+b+c=n} \dot\tau_V^{abc}(\wt T_0, x_{\wt T_0})  \dot R_{E_{\wt T_0}F_{\wt T_0}G_{\wt T_0}}^{abc} 
& \text{if } \wt T \text{ faces the side } y_{\wt T_0}z_{\wt T_0}\text{ of }\wt T_0,
\end{cases}
$$
and where
$$
\dot B(\wt T) =
\begin{cases}
 0
 & \text{if } w= x_{\wt T}
 \\
 0
 & \text{if } w= y_{\wt T}
 \\
 \displaystyle
 \sum_{a+b=n} \dot \sigma_V^{ab}(\wt T_0, \wt T)  \dot S^{ab}_{E_{\wt T}F_{\wt T}} 
+  \sum_{a+b+c=n} \dot\tau_V^{abc}(\wt T, x_{\wt T}) \dot L_{E_{\wt T}F_{\wt T}G_{\wt T}}^{abc} 
 \\
 \displaystyle
\quad=   \sum_{a+b=n} \dot \sigma_V^{ab}(\wt T_0, \wt T)  \dot S^{ab}_{E_{\wt T}F_{\wt T}} 
+\sum_{a+b+c=n} \dot\tau_V^{abc}(\wt T, x_{\wt T}) \dot R_{E_{\wt T}F_{\wt T}G_{\wt T}}^{abc} 
&\text{if } w=z_{\wt T}.
\end{cases}
$$
\end{prop}

\begin{proof}
As in  Proposition~\ref{prop:VariationSlitheringMap}, index the vertices of the base triangle $\wt T_0$ as $u_0$, $v_0$, $w_0$, {clockwise} in this order, in such a way that $\wt T$ faces the side $u_0v_0$ of $\wt T_0$.

 We clearly need to distinguish cases. Let us first focus on the case where $\wt T$ faces the side $x_{\wt T_0}y_{\wt T_0}$ of $\wt T_0$. In particular, $u_0=y_{\wt T_0}$  and $v_0 = x_{\wt T_0}$. 
 
 Let us apply Proposition~\ref{prop:VariationSlitheringMap}  to compute the slithering map $\Sigma_{x_{\wt T}y_{\wt T}, u_0v_0}^t$ associated to  $\wh\rho = \rho_t$. Since  $u_0=y_{\wt T_0}$  and $v_0 = x_{\wt T_0}$,  the map $A_0 \in \PGL$ occurring in the formula is equal to the identity, by our normalization that the flag map $\F_{\rho_t}$ respects the flags $E_{\wt T_0}$, $F_{\wt T_0}$ and the line $G_{\wt T_0}^{(1)}$. 
 
 Let $\wt T_1$, $\wt T_2$, \dots, $\wt T_{i_0}$ be a family of components of $\wt S-\wt\lambda$ separating $\wt T$ from $\wt T_0$, indexed in this order from $\wt T_0$ to $\wt T$. Then, if $\Delta_t \tau^{abc}=\tau^{abc}_{\rho_t} - \tau^{abc}_\rho$ and $\Delta_t \sigma^{ab}=\sigma^{ab}_{\rho_t} - \sigma^{ab}_\rho$ denote the variations of generalized Fock-Goncharov invariants, Proposition~\ref{prop:VariationSlitheringMap} asserts that the slithering map $\Sigma_{x_{\wt T}y_{\wt T}, u_0v_0}^t$ is the limit of
 $$
  A_t(\wt T_1) \circ A_t(\wt T_2) \circ \dots \circ A_t(\wt T_{i_0}) \circ \Sigma_{x_{\wt T}y_{\wt T}, u_0v_0}  \in \GL
 $$
 as the  family $\{ \wt T_1, \wt T_2, \dots, \wt T_{i_0}\}$ tends to the set of all components of $\wt S - \wt \lambda$ separating $\wt T$ from $\wt T_0$, where
 $$
A_t(\wt T_i)=
\begin{cases}
\displaystyle
\Bigcirc_{a+b=n}  S^{ab}_{E_{\wt T_i}F_{\wt T_i}} (\Delta_t\sigma^{ab}(\wt T_0, \wt T_{i}))
\circ
 \Bigcirc_{a+b+c=n} L^{abc}_{E_{\wt T_i}F_{\wt T_i}G_{\wt T_i}} \left( \Delta_t \tau^{abc} (\wt T_i, x_i) \right)
 \\ \qquad\qquad\displaystyle
 \circ \Bigcirc_{a+b+c=n} S^{a(b+c)}_{E_{\wt T_i} G_{\wt T_i}} ( \Delta_t\tau^{abc}(\wt T_i, x_i))
 \circ 
 \Bigcirc_{a+b=n} S^{ab}_{E_{\wt T_i} G_{\wt T_i}} (-\Delta_t\sigma^{ab}(\wt T_0, \wt T_i))
\\ \qquad \qquad\qquad\qquad\qquad\qquad\qquad\qquad\qquad\qquad
\text{if } \wt T \text{ faces the side } x_iz_i \text{ of }\wt T_i 
\\
\displaystyle
\Bigcirc_{a+b=n}  S^{ab}_{E_{\wt T_i}F_{\wt T_i}} (\Delta_t\sigma^{ab}(\wt T_0, \wt T_{i}))
\circ
\Bigcirc_{a+b+c=n} R^{abc}_{E_{\wt T_i}F_{\wt T_i}G_{\wt T_i}} \left( \Delta_t \tau^{abc} (\wt T_i, x_i) \right)
 \\ \qquad\qquad\displaystyle
\circ \Bigcirc_{a+b+c=n} S^{(a+c)b}_{G_{\wt T_i}F_{\wt T_i}} ( \Delta_t\tau^{abc}(\wt T_i, x_i))
\circ 
\Bigcirc_{a+b=n} S^{ab}_{G_{\wt T_i}F_{\wt T_i}} (-\Delta_t\sigma^{ab}(\wt T_0, \wt T_i))
\\ \qquad \qquad\qquad\qquad\qquad\qquad\qquad\qquad\qquad\qquad
\text{if } \wt T \text{ faces the side } y_iz_i \text{ of }\wt T_i .
\end{cases}
$$ 

Now, the constructions of \cite{BonDre2} and the proof of Proposition~\ref{prop:VariationSlitheringMap} admit automatic holomorphic extensions if we allow the variations $\Delta \tau^{abc}$ and $\Delta\sigma^{ab}$ of generalized Fock-Goncharov invariants to be complex with a small enough real part, as in \cite{Bon96} for the case $n=2$ and \cite{MalMarMazZha} for any $n$. The existence of this holomorphic extension enables us to take the derivative term-by-term in the limit and to prove that, for every flag $F\in \Flag$, the tangent vector
$$
\frac d{dt} \Sigma_{x_{\wt T}y_{\wt T}, u_0v_0}^t(F)_{|t=0} \in T_{\Sigma_{x_{\wt T}y_{\wt T}, u_0v_0}(F)} \Flag
$$
is equal to the limit of 
$$
\left( \sum_{i=1}^{i_0} \dot A(\wt T_i) \right) \Sigma_{x_{\wt T}y_{\wt T}, u_0v_0}(F)
$$
as the family $\{ \wt T_1, \wt T_2, \dots, \wt T_{i_0} \}$ tends to the set of all components of $\wt S-\wt\lambda$ separating $\wt T$ from the base component $\wt T_0$, where 
$$
\dot A(\wt T_i) = \frac d{dt} A_t(\wt T_i)_{|t=0}.
$$
An application of the chain rule  shows that $\dot A(\wt T_i)$ is as indicated in the statement of Proposition~\ref{prop:InfinitesimalGapFormula}. 

We can rephrase this property by saying that 
$$
\frac d{dt} \Sigma_{x_{\wt T}y_{\wt T}, u_0v_0}^t(F)_{|t=0} = \left( \sum_{\wt T' \text{ between } \wt T_0 \text{ and } \wt T} \dot A(\wt T') \right) \Sigma_{x_{\wt T}y_{\wt T}, u_0v_0}(F)
$$
for every flag $F\in \Flag$. 

We now consider the subcase where the vertex $w$ is equal to $x_{\wt T}$. We can then apply the above conclusion to $F= \F_\rho(u_0)$. Since we are in the case where $u_0=y_{\wt T_0}$, the flag $F$ is also equal to $\F_{\rho_t}(u_0)$ for every $t$ by our normalization convention for the Hitchin representations $\rho_t$.  Then,  by definition of the slithering map, 
$$
 \Sigma_{x_{\wt T}y_{\wt T}, u_0v_0}^t(F)
 =   \Sigma_{x_{\wt T}y_{\wt T}, u_0v_0}^t \big(  \F_{\rho_t}(u_0) \big)
 =  \F_{\rho_t}(x_{\wt T}) .
$$
 The special case $t=0$ yields
 $$
  \Sigma_{x_{\wt T}y_{\wt T}, u_0v_0}(F)  =  \F_{\rho}(x_{\wt T}).
 $$
 
 Reporting these values into our earlier computation, we obtain that
 $$
\frac d{dt} \F_{\rho_t}(x_{\wt T})_{|t=0} =  \sum_{\wt T' \text{ between } \wt T_0 \text{ and } \wt T} \dot A(\wt T')  \F_{\rho}(x_{\wt T}).
$$

This proves Proposition~\ref{prop:InfinitesimalGapFormula} in the case where $\wt T$ faces the side $x_{\wt T_0} y_{\wt T_0}$ of $\wt T_0$ and $w=x_{\wt T}$. 

Considering  the flag $F= \F_\rho(v_0)= \F_\rho(x_{\wt T_0})$ instead similarly provides the proof of Proposition~\ref{prop:InfinitesimalGapFormula} when $w=y_{\wt T}$, still assuming that $\wt T$ faces the side $x_{\wt T_0} y_{\wt T_0}$ of $\wt T_0$. 

When $w=z_{\wt T}$, we apply Proposition~\ref{prop:VariationSlitheringMap} to the leaf $x_{\wt T} z_{\wt T}$ and take term-by-term derivatives in the limit as above. After  observing that $\Sigma_{x_{\wt T}z_{\wt T}, u_0v_0} \big( \F_\rho(v_0) \big) = \F_{\rho_t}(z_{\wt T})$, we conclude that
 $$
\frac d{dt} \F_{\rho_t}(z_{\wt T})_{|t=0} =  \sum_{\wt T' \text{ between } \wt T_0 \text{ and } \wt T} \dot A(\wt T')  \F_{\rho}(z_{\wt T}) + \dot A(\wt T) \F_{\rho}(z_{\wt T}) 
$$
with
\begin{align*}
\dot A(\wt T) 
&=  \sum_{a+b+c=n} \dot\tau_V^{abc}(\wt T, x_{\wt T}) \left( \dot L_{E_{\wt T}F_{\wt T}G_{\wt T}}^{abc} + \dot S^{a(b+c)}_{E_{\wt T}G_{\wt T}}  \right)
 \\
& \qquad\qquad 
+ \sum_{a+b=n} \dot \sigma_V^{ab}(\wt T_0, \wt T) \left(  \dot S^{ab}_{E_{\wt T}F_{\wt T}} -  \dot S^{ab}_{E_{\wt T}G_{\wt T}}  \right) .
\end{align*}

By construction, the shearing map $S^{ab}_{E_{\wt T}G_{\wt T}}(t) \in \PGL$ respects $G_{\wt T} =  \F_{\rho}(z_{\wt T}) $. Therefore, $\dot S^{ab}_{E_{\wt T}G_{\wt T}} \F_{\rho}(z_{\wt T}) =0$ and the term $ \dot A(\wt T) \F_{\rho}(z_{\wt T}) $ is equal to $ \dot B(\wt T) \F_{\rho}(z_{\wt T}) $ with 
$$
 \dot B(\wt T) =  \sum_{a+b+c=n} \dot\tau_V^{abc}(\wt T, x_{\wt T}) \dot L_{E_{\wt T}F_{\wt T}G_{\wt T}}^{abc} 
 +\sum_{a+b=n} \dot \sigma_V^{ab}(\wt T_0, \wt T)  \dot S^{ab}_{E_{\wt T}F_{\wt T}} 
$$
as in the statement of Proposition~\ref{prop:InfinitesimalGapFormula}. (Note that, although $ \dot A(\wt T) \F_{\rho}(z_{\wt T}) = \dot B(\wt T) \F_{\rho}(z_{\wt T}) $, the terms  $ \dot A(\wt T)$, $ \dot B(\wt T) \in \sln$ are usually different.) This shows that
$$
\frac d{dt} \F_{\rho_t}(z_{\wt T})_{|t=0} =  \sum_{\wt T' \text{ between } \wt T_0 \text{ and } \wt T} \dot A(\wt T')  \F_{\rho}(z_{\wt T}) + \dot B(\wt T) \F_{\rho}(z_{\wt T}) 
$$
as in  Proposition~\ref{prop:InfinitesimalGapFormula}.   

Now, we could have used the side $y_{\wt T}z_{\wt T}$ of $\wt T$ instead of $x_{\wt T}z_{\wt T}$. This would lead us to the same statement as above, but with
$$
 \dot B(\wt T) =  \sum_{a+b+c=n} \dot\tau_V^{abc}(\wt T, x_{\wt T}) \dot R_{E_{\wt T}F_{\wt T}G_{\wt T}}^{abc} 
 +\sum_{a+b=n} \dot \sigma_V^{ab}(\wt T_0, \wt T)  \dot S^{ab}_{E_{\wt T}F_{\wt T}} . 
$$
This proves the equality of the two expressions for $ \dot B(\wt T) $ given in the statement of Proposition~\ref{prop:InfinitesimalGapFormula} (which can also be tracked down to the identity $L^{abc}_{EFG}(G) = R^{abc}_{EFG}(G)$ of Lemma~\ref{lem:EruptionsMainProperties}). 

This concludes the proof of Proposition~\ref{prop:InfinitesimalGapFormula} in the case where $\wt T$ faces the side $x_{\wt T_0}y_{\wt T_0}$ of $\wt T_0$. 

Let us now consider the case  where $\wt T$ faces the side $x_{\wt T_0}z_{\wt T_0}$. In this case, $u_0=x_{\wt T_0}$ and $v_0 = z_{\wt T_0}$. The formula for $\Sigma_{x_{\wt T}y_{\wt T}, u_0v_0}$ provided by Proposition~\ref{prop:VariationSlitheringMap} then involves the unique projective map $A_0^t\in \PGL$ sending the flag $E_{\wt T_0}=\F_\rho(x_{\wt T_0})$ to $E_{\wt T_0}^t=\F_{\rho_t}(x_{\wt T_0})$, the flag $G_{\wt T_0}=\F_\rho(z_{\wt T_0})$ to $G_{\wt T_0}^t=\F_{\rho_t}(z_{\wt T_0})$, and the line  $F_{\wt T_0}^{(1)}=\F_\rho(y_{\wt T_0})^{(1)}$ to $F_{\wt T_0}^{t\kern .2em (1)}= \F_{\rho_t}(y_{\wt T_0})^{(1)}$. 

By definition of Fock-Goncharov invariants and since $E_{\wt T_0}^t=E_{\wt T_0}$ and $F_{\wt T_0}^t=F_{\wt T_0}$, the flag $G_{\wt T_0}^t$ is such that
$$
X_{abc}(E_{\wt T_0}, F_{\wt T_0}, G^t_{\wt T_0}) = \exp \tau_{\rho_t}^{abc} (\wt T_0, x_0)
$$
for every $a$, $b$, $c\geq 1$ with $a+b+c=n$. In addition, the line  $ G_{\wt T_0}^{t\kern .2em (1)}$ coincides with $ G_{\wt T_0}^{(1)}$ by our normalization of the flag maps $\F_{\rho_t}$. 

On the other hand,  Lemma~\ref{lem:EruptionsMainProperties}(4) shows that the composition 
$$
\Bigcirc_{a+b+c=n} L_{E_{\wt T_0}F_{\wt T_0}G_{\wt T_0}} \big(\Delta_t \tau^{abc}(\wt T_0, x_0) \big)
$$ 
also sends $G_{\wt T_0}$ to a flag $\wh G_{\wt T_0}^t$ such that 
$$
X_{abc}(E_{\wt T_0}, F_{\wt T_0}, \wh G^t_{\wt T_0}) = \exp \tau_{\rho_t}^{abc} (\wt T_0, x_0)
$$
for every $a$, $b$, $c\geq 1$ with $a+b+c=n$. In addition, the line $\wh G_{\wt T_0}^{t\kern .2em (1)}$ coincides with $ G_{\wt T_0}^{(1)}$ by  Lemma~\ref{lem:EruptionsMainProperties}(3). It then immediately follows from Proposition~\ref{prop:TripleRatioClassifyFlagTriples} and Lemma~\ref{lem:ProjectiveMapBetweenFlagTriples} that $G_{\wt T_0}^t = \wh G_{\wt T_0}^t$. 

Since, by construction, the above composition of left eruptions also leaves $E_{\wt T_0}$ invariant, as well as the line $F_{\wt T_0}^{(1)}$, this proves that
$$
A_0^t= \Bigcirc_{a+b+c=n} L_{E_{\wt T_0}F_{\wt T_0}G_{\wt T_0}} \big(\Delta_t \tau^{abc}(\wt T_0, x_0) \big). 
$$ 

In the case where $w=x_{\wt T}$, 
$$
 \F_{\rho_t}(x_{\wt T}) = \Sigma_{x_{\wt T}y_{\wt T}, u_0v_0}^t \big( \F_{\rho_t}(u_0) \big) 
= \Sigma_{x_{\wt T}y_{\wt T}, u_0v_0}^t \circ A_0^t \big( \F_{\rho}(u_0) \big)  .
$$
 Proposition~\ref{prop:VariationSlitheringMap} then shows that $ \F_{\rho_t}(x_{\wt T}) $ is the limit of
$$
A_0^t \circ  A_t(\wt T_1) \circ A_t(\wt T_2) \circ \dots \circ A_t(\wt T_{i_0}) \big( \F_{\rho}(x_{\wt T}) \big) .
$$

Taking the derivative term-by-term in the limit as before, we conclude  that 
$$
\frac d{dt}  \F_{\rho_t}(x_{\wt T}) _{|t=0} 
= \dot B(\wt T_0) \big( \F_{\rho}(x_{\wt T}) \big)+ \sum_{\wt T' \text{ between } \wt T_0 \text{ and } \wt T} \dot A(\wt T')  \F_{\rho}(x_{\wt T})
$$
with
$$
\dot B(\wt T_0) = \frac d{dt} A_0^t{}_{|t=0} =   \sum_{a+b+c=n} \dot\tau^{abc}(\wt T_0, x_{\wt T_0})  \dot L_{E_{\wt T_0}F_{\wt T_0}G_{\wt T_0}}^{abc} . 
$$
This proves Proposition~\ref{prop:InfinitesimalGapFormula} in this subcase, where $\wt T$ faces the side $x_{\wt T_0}z_{\wt T_0}$ of $\wt T_0$ and $w=x_{\wt T}$. 

The argument is essentially identical in the subcase when $w=y_{\wt T}$, using now the fact that $\F_{\rho_t}(y_{\wt T}) =  \Sigma_{x_{\wt T}y_{\wt T}, u_0v_0}^t \circ A_0^t \big( \F_{\rho}(v_0) \big) $. 

In the subcase where $w=z_{\wt T}$, the same analysis as in the case where $\wt T$ faces  $x_{\wt T_0}y_{\wt T_0}$ provides an additional term
$$
\frac d{dt} A_t(\wt T) \F_{\rho}(z_{\wt T})_{|t=0} = \dot A(\wt T) \F_{\rho}(z_{\wt T}) = \dot B(\wt T) \F_{\rho}(z_{\wt T} )
$$
as in the statement of Proposition~\ref{prop:InfinitesimalGapFormula} .

This concludes the proof of Proposition~\ref{prop:InfinitesimalGapFormula}  in the case where $\wt T$ faces the side $x_{\wt T_0}z_{\wt T_0}$ of $\wt T_0$.

The proof when $\wt T$ faces the side $y_{\wt T_0}z_{\wt T_0}$ of $\wt T_0$ is essentially identical. 
\end{proof}

\section{The Weil cohomology class associated to a tangent vector $V \in T_{[\rho]} \Hit(S)$}
\label{bigsect:CohomologyClassTgtVector}

We now have all the analytic tools needed to implement the framework of \S \ref{subsect:SimplicialWeil}, and express the Weil cohomology class $[c_V] \in H^1(S; \sln_{\Ad\rho})$ associated to a tangent vector $V\in T_{[\rho]} \Hit(S)$ as a simplicial cohomology class for a suitable triangulation of the surface $S$. As in \S \ref{subsect:InfinitesimalVariationFlagMap}, we represent $V$ as the vector tangent to a curve $t\mapsto [\rho_t] \in \Hit(S)$ at  $[\rho] = [\rho_0]$, and we consider the directional derivatives 
\begin{align*}
 \dot \tau_V^{abc}(T,x) &= \frac{d}{dt} \tau^{abc}_{[\rho_t]}(T,x) _{|t=0}
 \\
 \dot \sigma_V^{ab}(e) &= \frac{d}{dt} \sigma^{ab}_{[\rho_t]}(e) _{|t=0}
\end{align*}
of the corresponding generalized Fock-Goncharov invariants. We also normalize the Hitchin representations $\rho_t$ so that, at the vertices of the base component $\wt T_0$, the flags $\F_{\rho_t}(x_{\wt T_0})$, $\F_{\rho_t}(y_{\wt T_0}) \in \Flag$ and the line $\F_{\rho_t}(z_{\wt T_0})^{(1)} \in \mathbb{RP}^{n-1}$ are independent of $t$. 

\subsection{Barriers for the geodesic lamination $\lambda$} 
\label{subsect:Barriers}

Consider a triangulation $\Sigma$ of the surface $S$, which we lift to a triangulation $\wt \Sigma$ of the universal cover $\wt S$. As in \S \ref{subsect:SimplicialWeil}, we want to define for each vertex $\wt v$ of $\wt \Sigma$ a projective basis $ \B_t(\wt v)$ such that
\begin{enumerate}
 \item the choice of $\mathcal B_t(\wt v)$ is $\rho_t$--equivariant, in the sense that $ \B_t(\gamma \wt v)= \rho_t(\gamma) \left(  \B_t(\wt v) \right)$ for every $\gamma \in \pi_1(S)$ and every vertex $\wt v$ of $\wt \Sigma$;
 \item for every vertex $\wt v$, the projective basis $\mathcal B_t(\wt v)$ depends differentiably on $t$.
\end{enumerate}

For this, following an idea of Sun-Zhang \cite{SunZha}, we  introduce additional data for the geodesic lamination $\lambda$ and impose a very mild transversality condition for the triangulation~$\Sigma$. 

The additional data is the following. In each connected component $T$ of $S-\lambda$,  we choose a Y-shaped \emph{barrier} that is the union $\beta_T$ of a point $\pi_T$ in the ideal triangle $T$ and of three infinite curves that respectively go from $\pi_T$ to each of the vertices at infinity of $T$; for instance, these three arms of $\beta_T$ can be the geodesics joining $\pi_T$ to the  vertices of $T$. Lift all these barriers $\beta_T$  to barriers $\beta_{\widetilde T}$ for the components $\widetilde T$ of $\widetilde S- \widetilde \lambda$. See Figure~\ref{fig:Barrier}. 

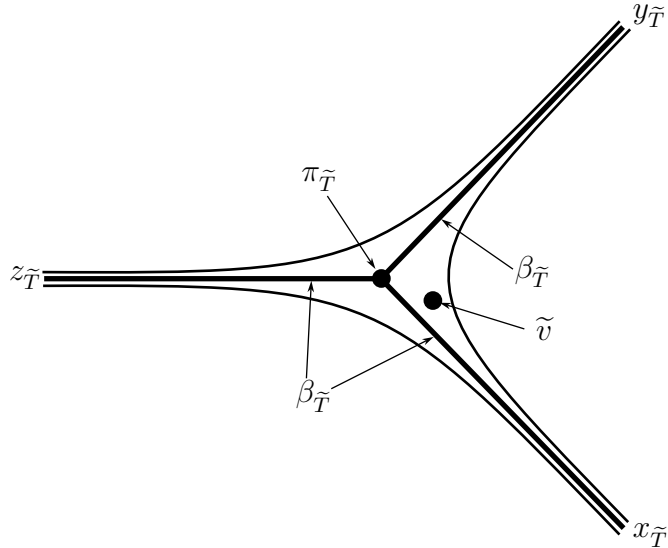
\begin{figure}[htbp]

\SetLabels
\E( .47 *.69  )  $\pi_{\wt T}$ \\
( .46 * .26 ) $\beta_{\wt T}$  \\
( .83 *  .5)  $\beta_{\wt T}$ \\
\L( 1 * 0 ) $x_{\wt T}$  \\
\L( 1 * 1 ) $y_{\wt T}$  \\
\E\R( .0 * .5 ) $z_{\wt T}$  \\
( .85 * .38 ) $\wt v$  \\
\endSetLabels
\centerline{\AffixLabels{\includegraphics{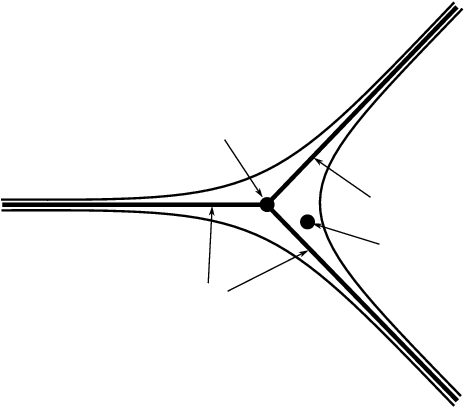}}}

\caption{A barrier in a component $\wt T$ of $\wt S- \wt\lambda$}
\label{fig:Barrier}
\end{figure}

We then require that each vertex of the triangulation $\Sigma$ is disjoint from the geodesic lamination $\lambda$ and its barriers $\beta_T$. This condition is easily attained by a small perturbation of $\Sigma$.

We can now use this data to define the projective basis $\B_t(\widetilde v)$ associated to each vertex $\wt v$ of the triangulation $\wt\Sigma$. By construction, $\wt v$ is contained in a component $\wt T$ of $\wt S - \wt\lambda$ and is disjoint from its barrier $\beta_{\wt T}$. This barrier splits $\wt T$ into three smaller triangles, each of which is delimited by one side of $\wt T$ and by two arms of the barrier $\beta_{\wt T}$. This naturally associates to $\wt v$ a side of $\wt T$. Let $x_{\wt v}$, $y_{\wt v}\in \bdry $ be the two endpoints of this side, and let $z_{\wt v}\in\bdry$ be the third vertex of $\wt T$, choosing the indexing so that $x_{\wt v}$, $y_{\wt v}$ and $z_{\wt v}$ occur in this order counterclockwise around $\wt T$. See Figure~\ref{fig:Barrier}. 

The projective basis $\B_t(\widetilde v)= \{ e_1, e_2, \dots, e_n\}$ is then defined as the unique projective basis such that: 
\begin{enumerate}
 \item the flag $\F_t(x_{\wt v})\in \Flag$ is the ascending flag of this basis, in the sense that the linear subspace $\F_t(x_{\wt v})^{(a)}$ is spanned by $e_1$, $e_2$, \dots, $e_a$, for every $a=1$, $2$, \dots, $n$;
 \item the flag $\F_t(y_{\wt v})$ is its descending flag, in the sense that $\F_t(y_{\wt v})^{(b)}$ is spanned by $e_{n-b+1}$, $e_{n-b+2}$, \dots, $e_n$, for every $b=1$, $2$, \dots, $n$;
 \item the line $\F_t(z_{\wt v})^{(1)}$ is spanned by the vector $e_1 + e_2 + \dots + e_n$. 
\end{enumerate}

Proposition~\ref{prop:InfinitesimalGapFormula} now provides us with a computation of the derivative
$$
 \frac d{dt} \F_{\rho_t}(w)_{|t=0} \in T_{\F_\rho(w)} \Flag
$$
for each vertex $w \in \{x_{\wt v},y_{\wt v}, z_{\wt v} \}$ of $\wt T$. This will enable us to determine
$$
c^0(\wt v) = \frac d{dt} \B_t(\widetilde v)\B_0(\widetilde v)^{-1}_{|t=0} \in \sln.
$$

Interpreting $c^0$ as a 0--cocycle $c^0 \in C^0(\wt S; \sln)$ for the simplicial cohomology associated to the triangulation $\Sigma$, Lemma~\ref{lem:ProjectiveBasisDefinesWeilCocyle} shows that the Weil class $[c_V] \in H^1(S; \sln_{\Ad\rho})$ associated to the tangent vector  $V\in T_{[\rho]} \Hit(S)$ is represented by the 1--cocycle $c_V = dc^0 \in C^1(\wt S; \sln)$.  Namely, for every oriented edge $\wt k$ of the triangulation $\wt \Sigma$ going from the vertex $\wt v_-$ to the vertex $\wt v_+$, 
$$
c_V(\wt k) =   \frac d{dt} \B_t(\widetilde v_+)\B_0(\widetilde v_+)^{-1}_{|t=0} -  \frac d{dt} \B_t(\widetilde v_-)\B_0(\widetilde v_-)^{-1}_{|t=0} \in \sln. 
$$

\subsection{A triangulation adapted to a  train track neighborhood}
\label{subsect:Triangulation}

 Let $\Phi$ be a train track neighborhood of $\lambda$ and let $\Psi$ be its associated train track, as in \S \ref{subsect:GeodLamTrainTracks}. In order to connect the computational scheme outlined in the previous section to the generalized Fock-Goncharov invariants associated to $\Psi$, we will use a triangulation $\Sigma$ of the surface $S$ that is well adapted to $\Phi$. This will be particularly convenient to compute the cup-product occurring  in \S \ref{subsect:Atiyah-Bott-Goldman} for the definition of the Atiyah-Bott-Goldman symplectic form.

Given the train track neighborhood $\Phi$, we can choose the barrier $\beta_T$ so that each point where it enters $\Phi$ is located very close to  a corner of the complement $S-\Phi$, corresponding to a switch of $\Phi$; see Figure~\ref{fig:TriangleAndTrainTrack} in \S \ref{subsect:GeodLamTrainTracks}. To simplify subsequent exposition, we can even arrange that this entry point is systematically located in the  branch incoming on the left side at that switch. See Figure~\ref{fig:Triangulation}. 
 
 We then decompose each branch of $\Phi$ into rectangles by splitting it along finitely many ties, including the switch ties at the end of the branch. Finally, we obtain a triangulation of $\Phi$, with all vertices on its boundary, by subdividing each rectangle into two triangles meeting along an arbitrary diagonal of the rectangle. 
 
 We need to be a little careful in the description of this triangulation of $\Phi$ near a switch. First of all, the first rectangle of the outgoing branch at that point is actually a pentagon, because of the additional vertex coming from the corner where the two incoming branches meet; this pentagon therefore needs two additional edges decomposing it into three triangles. To ease the exposition, we decide to choose these two edges so that they are disjoint from this additional vertex. See Figure~\ref{fig:Triangulation}.  
 
 Also, by our convention, the last rectangle of the branch incoming on the left at a switch meets a barrier $\beta_T$. Again to simplify further exposition, we decide to split this rectangle along a diagonal that is disjoint from the part of $\beta_T$ that is near the switch. See Figure~\ref{fig:Triangulation}. 

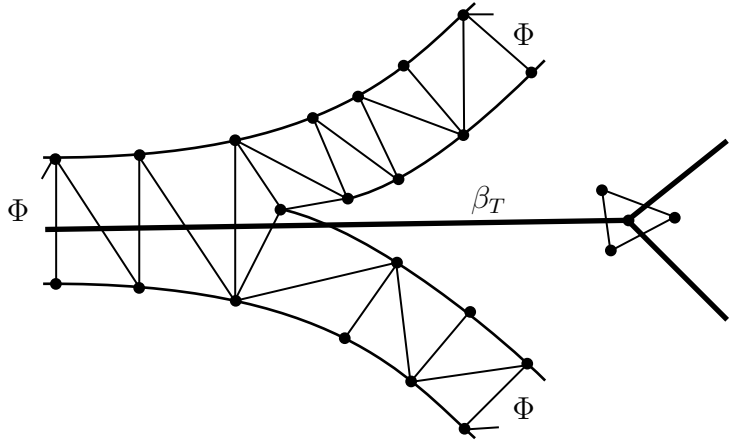
\begin{figure}[htbp]

\SetLabels
( .65 * 0.53 )  $\beta_T$ \\
( -.03 * .5)  $\Phi$ \\
( .7 * .9 )  $\Phi$ \\
( .7 * .05 )  $\Phi$ \\
\endSetLabels
\centerline{\AffixLabels{\includegraphics{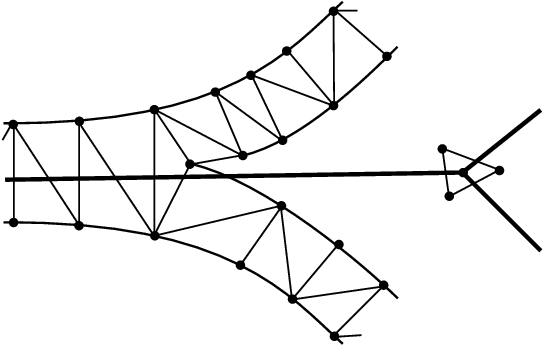}}}

\caption{A triangulation adapted to the train track neighborhood $\Phi$}
\label{fig:Triangulation}
\end{figure}

Finally, we extend this triangulation of the train track neighborhood $\Phi$ to a triangulation $\Sigma$ of the whole surface $S$. We choose the triangulation so that its vertices are disjoint from the barriers, and so that the triangulation is fine enough that every edge that is not contained in $\Phi$ meets the barriers $\beta_T$ in at most one point. 

A key feature of this triangulation is that almost every face has at least one side that is disjoint from the geodesic lamination $\lambda$ and from the barriers $\beta_T$. The only exceptions are the faces that contain the centers $\pi_T$ of the barriers $\beta_T$, plus two faces near each switch. This property will come in handy when we compute cup-products in \S \ref{bigsect:ComputeEstimateCupProduct}.

\subsection{Evaluation of the Weil cocycle on the edges of the triangulation}
\label{subsect:EvaluateCocyle}

Lift the triangulation $\Sigma$ that we just constructed to a triangulation $\wt \Sigma$ of the universal cover $\wt S$. It was specially designed so that, for the cocycle $c_V\in C^1(\wt S; \sln)$  representing the Weil class $[c_V] \in H^1(S; \sln_{\Ad\rho})$ that we constructed in \S \ref{subsect:Barriers},  the evaluation of $c_V$ over the oriented edges of $\wt \Sigma$ is relatively simple.

Many edges $\wt k$ of this triangulation are disjoint from the geodesic lamination $\wt \lambda$ and from the barriers $\beta_{\wt T}$. If,  for an arbitrary orientation, $\wt v_+$ and $\wt v_-$ are the positive and negative endpoints of such an edge, the projective bases $\mathcal B_t(\wt v_\pm)$ constructed in \S \ref{subsect:Barriers} coincide and it follows that
$$
c_V(\wt k) =   \frac d{dt} \B_t(\widetilde v_+)\B_0(\widetilde v_+)^{-1}_{|t=0} -  \frac d{dt} \B_t(\widetilde v_-)\B_0(\widetilde v_-)^{-1}_{|t=0} =0. 
$$
For future reference, we repeat this as follows.
\begin{lem}
\label{lem:EvaluateWeilEdgeDisjoinLaminationBarrier}
 If $\wt k$ is an oriented edge of the triangulation $\wt\Sigma$ that is disjoint from the geodesic lamination $\wt \lambda$ and from the barriers $\beta_{\wt T}$, then $ c_V(\wt k) =0$. \qed
\end{lem}

The next simplest case is that of an oriented edge $\wt k$ of $\wt\Sigma$ that is disjoint from the geodesic lamination $\wt\lambda$ but meets a barrier $\beta_{\wt T}$. By construction of the triangulation, $\wt k$ is contained in the closure of the complement of the train track neighborhood $\Phi$, and it meets $\beta_{\wt T}$ in exactly one point. 

This case splits into two subcases according to whether $\wt k$ turns to the left or to the right with respect to $\beta_{\wt T}$. More precisely, let $\wt v_+$ and $\wt v_-$ be the positive and negative endpoints of $\wt k$.  As in \S \ref{subsect:Barriers}, the component of $\wt T - \beta_{\wt T}$ that contains the negative endpoint $\widetilde v_{-}$ of $\wt k$ is adjacent to a side $x_{\wt T}y_{\wt T}$ of $\wt T$, whose orientation is the boundary orientation coming from $\wt T$. Let $z_{\wt T}$ denote the third vertex of $\wt T$. Then, either the positive endpoint $\wt v_+$ is in the component of $\wt T - \beta_{\wt T}$ that is adjacent to $x_{\wt T}z_{\wt T}$, in which case we say that \emph{$\wt k$ turns to the left}, or $\wt v_+$ is in the component of $\wt T - \beta_{\wt T}$ that is adjacent to $y_{\wt T}z_{\wt T}$ and  \emph{$\wt k$ turns to the right}. See Figure~\ref{fig:TriangleTurn}.

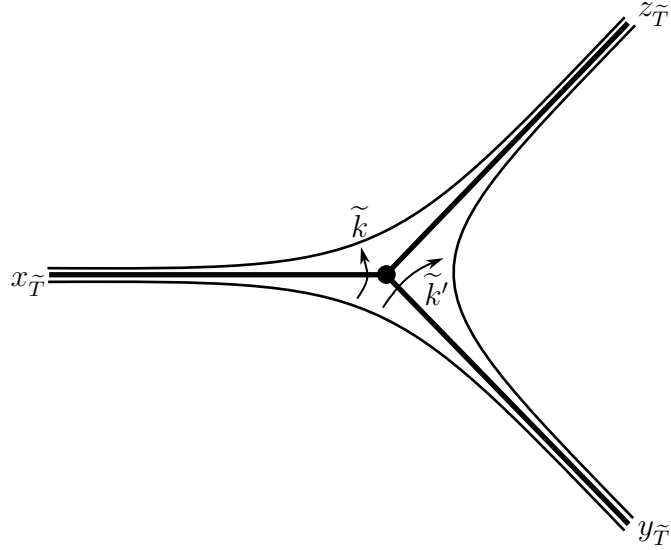
\begin{figure}[htbp]

\SetLabels
( .53 *.57  )  $\wt k$ \\
( .66 *  .44)  $\wt k'$ \\
\E\R(  0* .48 )  $x_{\wt T}$ \\
\E\L(  1* 1 )  $z_{\wt T}$ \\
\E\L(  1 * 0 )  $y_{\wt T}$ \\
\endSetLabels
\centerline{\AffixLabels{\includegraphics{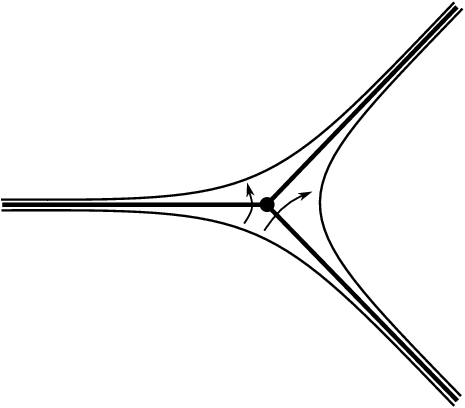}}}

\caption{The oriented arcs $\wt k$ and $\wt k'$ respectively turn left and right.}
\label{fig:TriangleTurn}
\end{figure}

\begin{lem}
\label{lem:EvaluateWeilEdgeMeetingBarrier}
If the oriented edge $\wt k$ of $\wt\Sigma$  is disjoint from the geodesic lamination $\wt\lambda$ and meets the barrier $\beta_{\wt T}$ in exactly one point,
$$
c_V(\wt k) =
\begin{cases}
\displaystyle
\sum_{a+b+c=n} \dot\tau_V^{abc}(\wt T,  x_{\wt T}) \dot L_{E_{\wt T}F_{\wt T}G_{\wt T}}^{abc}  &\text{ if } \wt k \text{ turns to the left}
 \\
 \displaystyle
 \sum_{a+b+c=n} \dot\tau_V^{abc}(\wt T,  x_{\wt T}) \dot R_{E_{\wt T}F_{\wt T}G_{\wt T}}^{abc}  
 &\text{ if } \wt k \text{ turns to the right}
\end{cases}
$$
with $E_{\wt T} = \F_\rho(x_{\wt T})$, $F_{\wt T} = \F_\rho(y_{\wt T})$, $G_{\wt T} = \F_\rho(z_{\wt T}) \in \Flag$. 
\end{lem}

\begin{proof}
[Proof of Lemma~{\upshape\ref{lem:EvaluateWeilEdgeMeetingBarrier}} in a special case]
 First consider the very special case where $\wt T$ is the base component $\wt T_0$ used to define the bases $\mathcal B_t(\wt v)$, and the vertices $x_{\wt T}$, $y_{\wt T}$, $z_{\wt T}$ are respectively equal to the vertices $x_{\wt T_0}$, $y_{\wt T_0}$, $z_{\wt T_0}$. We will then use the fact that we have normalized the Hitchin representations $\rho_t$ so that the flag $\F_{\rho_t}(x_{\wt T_0})$, the flag $\F_{\rho_t}(y_{\wt T_0})$ and the line $\F_{\rho_t}(z_{\wt T_0})^{(1)}$ are independent of $t$, and therefore respectively equal to $E_{\wt T_0}$, $F_{\wt T_0}$ and $G_{\wt T_0}^{(1)}$. 
 
 In particular, the flag $\F_{\rho_t}(z_{\wt T_0})\in \Flag$ is such that
 $$
 X_{abc} \big( E_{\wt T_0}, F_{\wt T_0} , \F_{\rho_t}(z_{\wt T_0}) \big) = \exp \tau^{abc}_{\rho_t} (\wt T_0, x_{\wt T_0})
 $$
 for every $a$, $b$, $c \geq 1$ with $a+b+c=n$. In addition, the line $\F_{\rho_t}(z_{\wt T_0})^{(1)}$ is equal to $G_{\wt T_0}^{(1)}$. From the combination of Proposition~\ref{prop:TripleRatioClassifyFlagTriples} and Lemma~\ref{lem:ProjectiveMapBetweenFlagTriples},  $\F_{\rho_t}(z_{\wt T_0})$ is completely determined by these two properties. Lemma~\ref{lem:EruptionsMainProperties}(3--4)  shows that the flag $\Bigcirc_{a+b+c=n} L_{E_{\wt T_0}F_{\wt T_0}G_{\wt T_0}}^{abc} \big( \Delta_t \tau^{abc} (\wt T_0, x_{\wt T_0}) \big)(G_{\wt T_0})$ satisfies the same properties. It follows that, for the variation $\Delta_t \tau^{abc} (\wt T_0, x_{\wt T_0}) $ of the triangle invariant, 
\begin{align*}
 \F_{\rho_t}(z_{\wt T_0}) 
 &=
 \Bigcirc_{a+b+c=n} L_{E_{\wt T_0}F_{\wt T_0}G_{\wt T_0}}^{abc} \big( \Delta_t \tau^{abc} (\wt T_0, x_{\wt T_0}) \big)(G_{\wt T_0})
 \\
 &=\displaystyle
 \Bigcirc_{a+b+c=n} R_{E_{\wt T_0}F_{\wt T_0}G_{\wt T_0}}^{abc} \big( \Delta_t \tau^{abc} (\wt T_0, x_{\wt T_0}) \big)(G_{\wt T_0}),
\end{align*}
where the second equality follows from Lemma~\ref{lem:EruptionsMainProperties}(2). 

We are now ready to compute
$$
c_V(\wt k) =   \frac d{dt} \B_t(\widetilde v_+)\B_0(\widetilde v_+)^{-1}_{|t=0} -  \frac d{dt} \B_t(\widetilde v_-)\B_0(\widetilde v_-)^{-1}_{|t=0} .
$$

By definition, the projective basis  $\B_t(\widetilde v_-)$ is the one for which $ \F_{\rho_t}(x_{\wt T_0}) $ is the ascending flag, $ \F_{\rho_t}(y_{\wt T_0}) $ is the descending flag, and the line $ \F_{\rho_t}(z_{\wt T_0}) ^{(1)}$ is spanned by the sum of the basis vectors. Since these objects are all independent of $t$, so is the basis  $\B_t(\widetilde v_-)$, and
$$
\frac d{dt} \B_t(\widetilde v_-)\B_0(\widetilde v_-)^{-1}_{|t=0} =0.
$$

For the projective basis $ \B_t(\widetilde v_+)$, we need to distinguish whether $\wt k $ turns left or right, although these two subcases are not fundamentally different. 

If $\wt k$ turns left,  $\B_t(\widetilde v_+)$ is the projective basis for which $ \F_{\rho_t}(z_{\wt T_0}) $ is the ascending flag, $ \F_{\rho_t}(x_{\wt T_0}) $ is the descending flag, and the line $ \F_{\rho_t}(y_{\wt T_0}) ^{(1)}$ is spanned by the sum of the basis vectors. Set
$$
B_t =  \Bigcirc_{a+b+c=n} L_{E_{\wt T_0}F_{\wt T_0}G_{\wt T_0}}^{abc} \big( \Delta_t \tau^{abc} (\wt T_0, x_{\wt T_0}) \big) .
$$
Then
\begin{align*}
 \F_{\rho_t}(z_{\wt T_0})  &= B_t  \F_\rho(z_{\wt T_0}) 
 \\
 \F_{\rho_t}(x_{\wt T_0})  &=  \F_{\rho}(x_{\wt T_0})  = B_t  \F_\rho(x_{\wt T_0}) 
 \\
 \F_{\rho_t}(y_{\wt T_0})^{(1)}  &=  \F_{\rho}(y_{\wt T_0})^{(1)}  = B_t  \F_\rho(y_{\wt T_0}) ^{(1)}
\end{align*}
where the first equality is just a rephrasing of our earlier computation, the second and fourth equalities are consequences of our normalization, and the remaining two come from the properties of left eruptions stated in Lemma~\ref{lem:EruptionsMainProperties}(1--2). It follows that $\B_t(\widetilde v_+) = B_t \B_0(\widetilde v_+)$, and 
$$
\frac d{dt} \B_t(\widetilde v_+)\B_0(\widetilde v_+)^{-1}_{|t=0} 
= \frac d{dt} B_t{}_{|t=0} =  \sum_{a+b+c=n} \dot\tau_V^{abc}(\wt T_0,  x_{\wt T_0}) \dot L_{E_{\wt T_0}F_{\wt T_0}G_{\wt T_0}}^{abc}  .
$$
Therefore,
\begin{align*}
 c_V(\wt k) &=   \frac d{dt} \B_t(\widetilde v_+)\B_0(\widetilde v_+)^{-1}_{|t=0} -  \frac d{dt} \B_t(\widetilde v_-)\B_0(\widetilde v_-)^{-1}_{|t=0}
 \\
&=\sum_{a+b+c=n} \dot\tau_V^{abc}(\wt T,  x_{\wt T}) \dot L_{E_{\wt T}F_{\wt T}G_{\wt T}}^{abc}   
\end{align*}
in this subcase where $\wt k$ turns to the left, and $\wt T= \wt T_0$, $x_{\wt T}=x_{\wt T_0}$, $y_{\wt T}=y_{\wt T_0}$, $z_{\wt T}=z_{\wt T_0}$. This proves Lemma~\ref{lem:EvaluateWeilEdgeMeetingBarrier} in this subcase. 

The argument is very similar for the subcase when $\wt k$ turns to the right, in which case $\B_t(\widetilde v_+) = C_t \B_0(\widetilde v_+)$ with
$$
C_t =  \Bigcirc_{a+b+c=n} R_{E_{\wt T_0}F_{\wt T_0}G_{\wt T_0}}^{abc} \big( \Delta_t \tau^{abc} (\wt T_0, x_{\wt T_0}) \big) .
$$

 This concludes the proof of Lemma~\ref{lem:EvaluateWeilEdgeMeetingBarrier} in the special case where $\wt T= \wt T_0$, $x_{\wt T}=x_{\wt T_0}$, $y_{\wt T}=y_{\wt T_0}$, $z_{\wt T}=z_{\wt T_0}$. 
 \end{proof}
 
 \begin{proof}
[Proof of Lemma~{\upshape\ref{lem:EvaluateWeilEdgeMeetingBarrier}} in the general case]
 
 In the general case, consider the unique projective map $A_t\in \PGL$ sending the flag  $E_{\wt T_0}= \F_\rho(x_{\wt T_0})$ to $E_{\wt T}^t= \F_{\rho_t}(x_{\wt T})$, the flag $F_{\wt T_0}= \F_\rho(y_{\wt T_0})$ to $F_{\wt T}^t= \F_{\rho_t}(y_{\wt T})$, and the line $G_{\wt T_0}^{(1)}= \F_\rho(z_{\wt T_0})^{(1)}$ to $G_{\wt T}^{t\kern .1em (1)}= \F_{\rho_t}(z_{\wt T})^{(1)}$, as provided by Lemma~\ref{lem:ProjectiveMapBetweenFlagTriples}. We then use $A_t$ to conjugate $\rho_t$ to another representative $\rho_t' \colon \pi_1(S) \to \PGL$  of the Hitchin character $[\rho_t]=[\rho_t'] \in \Hit(S)$, defined by the property that
 $$
 \rho_t'(\gamma) = A_t^{-1} \rho_t(\gamma) A_t \in \PGL
 $$
for every $\gamma \in \pi_1(S)$.
 
 The flag map of this new Hitchin representation $\rho_t'$ is then $\F_{\rho_t'}= A_t^{-1} \circ \F_{\rho_t}$, if we use the same symbols to denote $A_t \in \PGL$ and its action on $\Flag$. As a consequence, the projective basis $\mathcal B_t'(\wt v)$ associated to each vertex $\wt v$ by the flag map $\F_{\rho_t'}$ is equal to the image $A_t^{-1} \mathcal B_t(\wt v)$ of the earlier basis $\mathcal B_t(\wt v)$ under the action of $A_t^{-1}$. We can then use this new projective basis to compute
 \begin{align*}
c_V(\wt k)
&=\frac d{dt} \mathcal B_t(\wt v_+)\mathcal B_0(\wt v_+)^{-1}{}_{|t=0} - \frac d{dt} \mathcal B_t(\wt v_-)\mathcal B_0(\wt v_-)^{-1}{}_{|t=0}
\\
&= \frac d{dt} A_t \mathcal B_t'(\wt v_+)\mathcal B_0'(\wt v_+)^{-1}A_0^{-1}{}_{|t=0} - \frac d{dt} A_t \mathcal B_t'(\wt v_-)\mathcal B_0'(\wt v_-)^{-1}A_0^{-1}{}_{|t=0}
\\
&= \frac d{dt} A_t A_0^{-1} {}_{|t=0} + \frac d{dt} A_0 \mathcal B_t'(\wt v_+)\mathcal B_0'(\wt v_+)^{-1}A_0^{-1}{}_{|t=0} 
\\
&\qquad\qquad
-\frac d{dt} A_t A_0^{-1} {}_{|t=0} - \frac d{dt} A_0 \mathcal B_t'(\wt v_-)\mathcal B_0'(\wt v_-)^{-1}A_0^{-1}{}_{|t=0} 
\\
&= A_0 \left( \frac d{dt}  \mathcal B_t'(\wt v_+)\mathcal B_0'(\wt v_+)^{-1}{}_{|t=0} 
- \frac d{dt}  \mathcal B_t'(\wt v_-)\mathcal B_0'(\wt v_-)^{-1}{}_{|t=0} \right) A_0^{-1}.
\end{align*}

The flag $\F_{\rho_t'}(x_{\wt T}) = E_{\wt T_0}$, the flag $\F_{\rho_t'}(y_{\wt T}) = F_{\wt T_0}$ and the line $\F_{\rho_t'}(z_{\wt T})^{(1)} = G_{\wt T_0}^{(1)}$ are independent of $t$. We can therefore apply the special case to this new normalization $\rho_t'$ of the Hitchin representations representing the characters $[\rho_t]=[\rho_t'] \in \Hit(S)$. For instance, when $\wt k$ turns to the left, this gives
\begin{align*}
 c_V(\wt k)
&= A_0 \left( \frac d{dt}  \mathcal B_t'(\wt v_+)\mathcal B_0'(\wt v_+)^{-1}{}_{|t=0} 
- \frac d{dt}  \mathcal B_t'(\wt v_-)\mathcal B_0'(\wt v_-)^{-1}{}_{|t=0} \right) A_0^{-1} 
\\
&= A_0 \left( \sum_{a+b+c=n} \dot\tau_V^{abc}(\wt T,  x_{\wt v_-}) \dot L_{E_{\wt T_0}F_{\wt T_0}G_{\wt T_0}}^{abc} \right) A_0^{-1} 
\\
&= \sum_{a+b+c=n} \dot\tau_V^{abc}(\wt T,  x_{\wt v_-})  \dot L_{(A_0E_{\wt T_0}) (A_0 F_{\wt T_0})(A_0G_{\wt T_0})}^{abc} 
\\
&= \sum_{a+b+c=n} \dot\tau_V^{abc}(\wt T,  x_{\wt v_-})  \dot L_{E_{\wt T} F_{\wt T}G_{\wt T}}^{abc} .
\end{align*}
This proves the lemma when $\wt k$ turns to the left. 

The argument is  identical when $\wt k$ turns to the right, by replacing left eruptions with right eruptions, and concludes the proof of Lemma~\ref{lem:EvaluateWeilEdgeMeetingBarrier}. 
\end{proof}

Lemmas \ref{lem:EvaluateWeilEdgeDisjoinLaminationBarrier} and \ref{lem:EvaluateWeilEdgeMeetingBarrier} take care of all the edges of the triangulation $\wt\Sigma$ that are not contained in the train track neighborhood $\wt\Phi$. Among the edges of the triangulations of $\wt\Phi$, many of them were constructed as ties of the train track neighborhood. We will later see that analyzing the images  of these tie edges under the 1--cocycle $c_V$ will enable us to determine the images  under $c_V$ of all the edges of the triangulation $\wt \Sigma$ that are contained in $\wt\Phi$. 

\begin{lem}
\label{lem:EvaluateWeilTie}
 Let $\wt k$ be an oriented tie of the train track neighborhood $\wt\Phi$ that is generic, in the sense that it is not contained in a switch tie, and let $\wt T_-$ and $\wt T_+$ be the components of $\wt S-\wt\lambda$ that respectively contain the negative and positive endpoints of $\wt k$. For every component $\wt T'$ of $\wt S - \wt\lambda$ that separates $\wt T_-$ from $\wt T_+$, counterclockwise label its vertices as $x_{\wt T'}$, $y_{\wt T'}$, $z_{\wt T'}\in \bdry$ in such a way that  the side $x_{\wt T'}y_{\wt T'}$ faces $\wt T_-$. Consider the flags $E_{\wt T'}=\F_\rho(x_{\wt T'})$, $F_{\wt T'}=\F_\rho(y_{\wt T'})$, $G_{\wt T'}=\F_\rho(z_{\wt T'})\in \Flag$ for each such component $\wt T'$, and use the same labelling conventions for $\wt T_+$. Then, 
 \begin{align*}
c_V(\wt k) &=  \sum_{\wt T' \text{ between } \wt T_- \text{ and } \wt T_+} \dot A(\wt T') 
 +  \sum_{a+b=n} \dot \sigma_V^{ab}(\wt T_-, \wt T_+)  \dot S^{ab}_{E_{\wt T_+}F_{\wt T_+}}  \in \sln
\end{align*}
where the sum is over all components $\wt T'$ of $\wt S-\wt\lambda$ separating $\wt T_-$ from $\wt T_+$ and where, as in Proposition~{\upshape\ref{prop:InfinitesimalGapFormula}}, 
$$
\dot A(\wt T') =
\begin{cases}
\displaystyle
 \sum_{a+b+c=n} \dot\tau_V^{abc}(\wt T', x_{\wt T'}) \left( \dot L_{E_{\wt T'}F_{\wt T'}G_{\wt T'}}^{abc} + \dot S^{a(b+c)}_{E_{\wt T'}G_{\wt T'}}  \right)
 \\
 \displaystyle
 \qquad\qquad + \sum_{a+b=n} \dot \sigma_V^{ab}(\wt T_-, \wt T') \left(  \dot S^{ab}_{E_{\wt T'}F_{\wt T'}} -  \dot S^{ab}_{E_{\wt T'}G_{\wt T'}}  \right)
 \\
 \qquad\qquad  \qquad\qquad  \qquad\qquad  \qquad\qquad  \text{if } \wt T_+ \text{ faces the side } x_{\wt T'}z_{\wt T'}\text{ of }\wt T'
 \\
 \displaystyle
  \sum_{a+b+c=n} \dot\tau_V^{abc}(\wt T', x_{\wt T'}) \left( \dot R_{E_{\wt T'}F_{\wt T'}G_{\wt T'}}^{abc} + \dot S^{(a+c)b}_{G_{\wt T'}F_{\wt T'}}  \right)
 \\
 \displaystyle
 \qquad\qquad + \sum_{a+b=n} \dot \sigma_V^{ab}(\wt T_-, \wt T') \left(  \dot S^{ab}_{E_{\wt T'}F_{\wt T'}} -  \dot S^{ab}_{G_{\wt T'}F_{\wt T'}}  \right)
 \\
 \qquad\qquad  \qquad\qquad  \qquad\qquad  \qquad\qquad  \text{if } \wt T_+ \text{ faces the side } y_{\wt T'}z_{\wt T'}\text{ of } \wt T'.
\end{cases}
$$
\end{lem}

\begin{proof}
[Proof of Lemma~{\upshape\ref{lem:EvaluateWeilTie}} in a special case]

We again consider a very special case, where $\wt T_-$ is the base component $\wt T_0$, and  where $\wt T_+$ faces the side $x_{\wt T_0}y_{\wt T_0}$ of $\wt T_0$. 
 
 A key property of generic ties  is that the components of $\wt k-\wt\lambda$ containing each of the endpoints of $\wt k$ are disjoint from the barriers. (Note that this fails at each switch for the boundary tie of the left incoming branch.)

 In particular, for the negative endpoint $\wt u_-$ of $\wt k$, the projective basis $\mathcal B_t(\wt u_-)$ is determined by the properties that the flag $\F_{\rho_t}(x_{\wt T_0})$ is its ascending flag, the flag $\F_{\rho_t}(y_{\wt T_0})$ is its descending flag, and the line $\F_{\rho_t}(z_{\wt T_0})^{(1)}$ is spanned by the sum of the basis elements. By our normalization of the Hitchin representations $\rho_t$, these two flags and this line are independent of  $t$. It follows that 
$$
\frac d{dt} \B_t(\widetilde u_-)\B_0(\widetilde u_-)^{-1}_{|t=0} =0.
$$

For the positive endpoint $\wt u_+$ of $\wt k$, the projective basis $\mathcal B_t(\wt u_+)$ is determined by the properties that the flag $\F_{\rho_t}(x_{\wt T_+})$ is its ascending flag, the flag $\F_{\rho_t}(y_{\wt T_+})$ is its descending flag, and the line $\F_{\rho_t}(z_{\wt T_+})^{(1)}$ is spanned by the sum of the basis elements.

To compute the variation of $\mathcal B_t (\wt u_+)$, we use the formula
\begin{align*}
 \frac{d}{dt} \F_{\rho_t}(w)_{|t=0}  &= 
  \sum_{\wt T' \text{ between } \wt T_0 \text{ and } \wt T_+} \dot A(\wt T')  \F_\rho(w)
 + \dot B(\wt T_+)  \F_\rho(w)
\end{align*} 
 of Proposition~\ref{prop:InfinitesimalGapFormula} for each $w \in \{x_{\wt T_+}, y_{\wt T_+}, z_{\wt T_+}\}$, where $ \dot A(\wt T') $ is as repeated in the statement of Lemma~\ref{lem:EvaluateWeilTie}, and where
 $$
\dot B(\wt T_+) =
\begin{cases}
 0
 & \text{if } w= x_{\wt T_+}
 \\
 0
 & \text{if } w= y_{\wt T_+}
 \\ \displaystyle
 \sum_{a+b=n} \dot \sigma_V^{ab}(\wt T_0, \wt T_+)  \dot S^{ab}_{E_{\wt T_+}F_{\wt T_+}} 
+  \sum_{a+b+c=n} \dot\tau_V^{abc}(\wt T_+, x_{\wt T_+}) \dot L_{E_{\wt T_+}F_{\wt T_+}G_{\wt T_+}}^{abc} 
&\text{if } w=z_{\wt T_+}.
\end{cases}
$$

Note that  $ \dot S^{ab}_{E_{\wt T_+}F_{\wt T_+}} \F_\rho(x_{\wt T_+}) =0$ and $ \dot S^{ab}_{E_{\wt T_+}F_{\wt T_+}} \F_\rho(y_{\wt T_+}) =0$  since the shearing map $S^{ab}_{E_{\wt T_+}F_{\wt T_+}}(t) \in \PGL$ fixes the flags $E_{\wt T_+}=\F_\rho(x_{\wt T_+}) $ and $F_{\wt T_+}=\F_\rho(y_{\wt T_+}) $. Also, the left eruption $L_{E_{\wt T_+}F_{\wt T_+}G_{\wt T_+}}^{abc} (t)$ respects the line $G_{\wt T_+}^{(1)} $, and therefore $\dot L_{E_{\wt T_+}F_{\wt T_+}G_{\wt T_+}}^{abc} G_{\wt T_+}^{(1)} =0$. It follows that
\begin{align*}
 \frac{d}{dt} \F_{\rho_t}(x_{\wt T_+})_{|t=0}  &=   \dot C \F_\rho(x_{\wt T_+})  \in T_{E_{\wt T_+}} \Flag
 \\
  \frac{d}{dt} \F_{\rho_t}(y_{\wt T_+})_{|t=0}  &=   \dot C \F_\rho(y_{\wt T_+})  \in T_{F_{\wt T_+}}\Flag
  \\
  \frac{d}{dt} \F_{\rho_t}(z_{\wt T_+})^{(1)}{}_{|t=0}  &=   \dot C \F_\rho(z_{\wt T_+}) ^{(1)} \in T_{G_{\wt T_+}^{(1)}} \mathbb{RP}^{n-1}
\end{align*}
for the same element 
$$
\dot C = 
 \sum_{\wt T' \text{ between } \wt T_0 \text{ and } \wt T_+} \dot A(\wt T') +  \sum_{a+b=n} \dot \sigma_V^{ab}(\wt T_0, \wt T_+)  \dot S^{ab}_{E_{\wt T_+}F_{\wt T_+}}  \in \sln. 
$$

Since the projective basis $ \mathcal B_t(\wt v_+)$ is determined by the flags $ \F_{\rho_t}(x_{\wt T_+})$,  $\F_{\rho_t}(y_{\wt T_+})$ and the line $ \F_{\rho_t}(z_{\wt T_+})^{(1)}$, it follows that 
$$
\frac d{dt} \mathcal B_t(\wt v_+)_{|t=0} = \dot C  \mathcal B_0(\wt v_+).
$$

Therefore,
$$
c_V(\wt k) = \frac d{dt} \mathcal B_t(\wt v_+) \mathcal B_0(\wt v_+)^{-1}_{|t=0} 
- \frac d{dt} \mathcal B_t(\wt v_-) \mathcal B_0(\wt v_-)^{-1}_{|t=0} = \dot C,
$$
which is exactly what we were supposed to prove in this special case  where $\wt T_-$ is the base component $\wt T_0$ and  $\wt T_+$ faces the side $x_{\wt T_0}y_{\wt T_0}$ of $\wt T_0$. 
\end{proof}

\begin{proof}
 [Proof of Lemma~{\upshape\ref{lem:EvaluateWeilTie}} in the general case]
 
 The general case can be deduced from the special case by the same argument that we used for Lemma \ref{lem:EvaluateWeilEdgeMeetingBarrier}. 
 
 Namely  consider the projective map $A_t\in \PGL$ sending the flag  $\F_\rho(x_{\wt T_0})$ to $\F_{\rho_t}(x_{\wt T_-})$, the flag $\F_\rho(y_{\wt T_0})$ to $ \F_{\rho_t}(y_{\wt T_-})$, and the line $ \F_\rho(z_{\wt T_0})^{(1)}$ to $\F_{\rho_t}(z_{\wt T_-})^{(1)}$. If the Hitchin representation $\rho_t'$ is obtained by conjugating $\rho_t$ with $A_t$, the flags $\F_{\rho_t}(x_{\wt T_-})$, $\F_{\rho_t}(y_{\wt T_-})$ and the line $\F_{\rho_t}(z_{\wt T_-})^{(1)}$ are now independent of $t$, and we can apply the special case to $\rho_t'$. As in the proof of Lemma~\ref{lem:EvaluateWeilEdgeMeetingBarrier}, the formula for $\rho_t$ is then obtained from the formula for $\rho_t'$ by a simple conjugation by $A_0$. 
\end{proof}

\subsection{Opening zippers}
\label{subsect:OpeningZippers}

An obvious problem with the formula of Lemma~\ref{lem:EvaluateWeilTie} is that it involves an infinite sum, which would be cumbersome to directly use in cup-product computations. So, as in \cite{SozBon, Zey},  we will use an approximation based on the idea of opening zippers in a train track neighborhood to make it arbitrarily close to the geodesic lamination~$\lambda$. 

The operation of zipper opening replaces a train track neighborhood $\Phi$ of the geodesic lamination $\lambda$ by one that more closely approximates $\lambda$. More precisely, let $a$ be an arc in $\Phi$ that starts at a switch point of $\Phi$, is disjoint from $\lambda$ and transverse to the ties of $\Phi$, and ends on a tie that is not a switch tie. Splitting $\Phi$ along $a$ gives a new train track neighborhood $\Phi'$ of $\lambda$. We say that $\Phi'$ is obtained from $\Phi$ by \emph{zipper opening along $a$}. See Figure~\ref{fig:ZipperOpening}, and \cite[\S 1.7, \S 2.4]{PenHar} for details.

Instead of a single arc $a$, we can also split $\Phi$ along  a family of (necessarily disjoint) such arcs issued from distinct switch points of $\Phi$. In particular, for a given integer $\delta\geq 1$ we can take, for each switch point $s$, an arc $a_s$ as above that is issued from $s$ and crosses exactly $\delta$ switch ties. If the train track neighborhood $\Phi_\delta$ is obtained from $\Phi$ by zipper opening along the union of these arcs $a_s$, we say that $\Phi_\delta$ is obtained from $\Phi$ by \emph{zipper opening up to depth $\delta$}. One easily sees that, up to isotopy,  $\Phi_\delta$ depends only on $\Phi$ and $\delta$. 

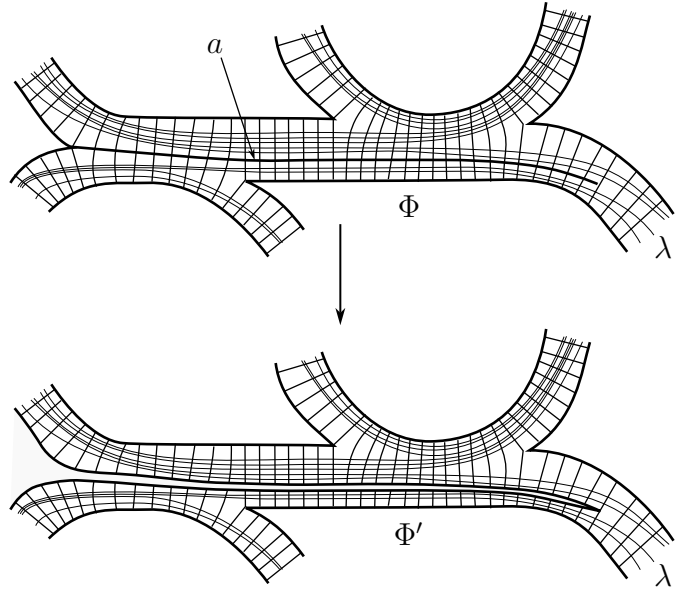
\begin{figure}[htbp]

\SetLabels
( .31 * .91 )  $a$ \\
( .6 * .63 )  $\Phi$ \\
( .6 * .07 )  $\Phi'$ \\
( .98 * .56 )  $\lambda$ \\
( .98 * .0 )  $\lambda$ \\
\endSetLabels
\centerline{\AffixLabels{\includegraphics{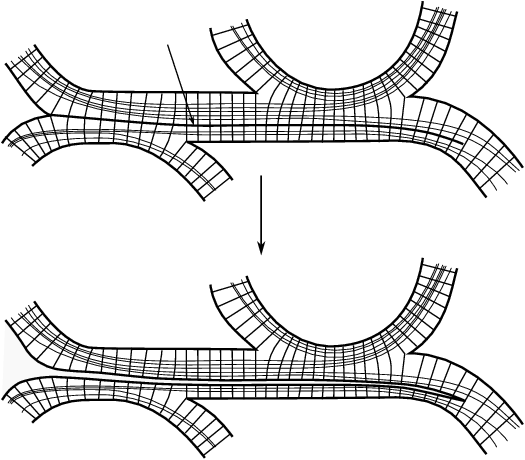}}}

\caption{Opening a zipper}
\label{fig:ZipperOpening}
\end{figure}

\begin{lem}
\label{lem:EstimateWeilZipperOpening}
 Let $\Phi_0$ be a train track neighborhood of the geodesic lamination $\lambda$,  let $\Phi_\delta$ be obtained  from $\Phi_0$ by a zipper opening up to depth $\delta$, and let $\wt\Phi_\delta$ be the preimage of $\Phi_\delta$ in the universal cover $\wt S$. Then, for every tie $\wt k$ of $\wt\Phi_\delta$ and with the conventions of Lemma~{\upshape\ref{lem:EvaluateWeilTie}}, there exists a constant $C>0$ such that
 $$
 c_V(\wt k) =  \sum_{a+b=n} \dot \sigma_V^{ab}(\wt T_-, \wt T_+)  \dot S^{ab}_{E_{\wt T_+}F_{\wt T_+}}  + O ( \E^{-C\delta} ) \in \sln
 $$
 where $C$ and the other constant hidden in the Landau symbol $O(\ )$ depend only on the Hitchin representation $\rho$,  the tangent vector $V\in T_{[\rho]} \Hit(S)$, the original train track neighborhood $\Phi_0$, and a compact subset of $\wt S$ containing the tie $\wt k$. 
\end{lem}

\begin{proof} By construction, the tie $\wt k$ of $\wt\Phi_\delta$ is contained in a tie $\wt k_0$ of the original train track $\Phi_0$. 

In addition to our earlier estimates, the argument is based on a classical measure of the combinatorial complexity of a component of $\wt k_0 - \wt\lambda$ with respect to $\wt\Phi_0$. Such a component is of the form $\wt k_0 \cap \wt T$ for some component $\wt T$ of $\wt S-\wt\lambda$. If none of the endpoints $\wt k_0 \cap \wt T$ is an endpoint of $\wt k_0$, the two sides of $\wt T$ passing through these endpoints are asymptotic to each other on one side of $\wt k_0$, and on the other side follow a common path of $r(\wt T)\geq 1$ branches of $\wt \Phi_0$ before diverging at some switch of $\wt \Phi_0$. This number $r(\wt T)$ is the \emph{combinatorial depth} of  $\wt k_0 \cap \wt T$ in the train track neighborhood $\wt\Phi_0$. By convention, $r(\wt T)=0$ if $\wt k_0 \cap \wt T$ contains one of the endpoints of $\wt k_0$. 

This combinatorial depth actually played a crucial role in the proof of the  estimates of Lemmas~\ref{lem:ElementarySlitheringBoundedByGapLength} and \ref{lem:HolderSumOfGapLengths} that we borrowed from \cite{BonDre2}. One key ingredient was the following:

\begin{sublem}
\label{sublem:EstimateWeilZipperOpening0}
 There exist constants $C'$, $C''$, $D'$, $D''>0$, depending only on the negatively curved metric $m_0$ that we used to define geodesic laminations and on the train track neighborhood $\Phi_0$,   such that
 $$
 D' \E^{-C' r(\wt T)} \leq \ell( \wt k_0 \cap \wt T ) \leq D'' \E^{-C'' r(\wt T)}
 $$
 for every component $ \wt k_0 \cap \wt T$ of $ \wt k_0 \cap \wt\lambda$. 
\end{sublem}
\begin{proof}
 Before they diverge at some switch of $\wt\Phi_0$, the two sides of $\wt T$ travel together over a length that is bounded above and below by a constant times $r(\wt T)$. The estimate then follows from an estimate using the negative curvature of $m_0$. The constants $C'$, $C''$, $D'$, $D''>0$ depend on curvature bounds for $m_0$, on the ``lengths'' of the branches of $\Phi_0$, and on a lower bound on the angles between the leaves of $\lambda$ and the ties of $\Phi_0$. 
\end{proof}

We are now ready to begin the proof of Lemma~\ref{lem:EstimateWeilZipperOpening}. From the formula 
 \begin{align*}
c_V(\wt k) &= \sum_{a+b=n} \dot \sigma_V^{ab}(\wt T_-, \wt T_+)  \dot S^{ab}_{E_{\wt T_+}F_{\wt T_+}}  
+ \sum_{\wt T' \text{ between } \wt T_- \text{ and } \wt T_+} \dot A(\wt T') 
\end{align*}
 of Lemma~\ref{lem:EvaluateWeilTie}, we need to show that
$$
 \sum_{\wt T' \text{ between } \wt T_- \text{ and } \wt T_+} \dot A(\wt T')   =  O ( \E^{-C\delta} )  
$$
for some constant $C>0$. 

For this, we will use the estimates that already appeared in our proofs of Propositions~\ref{prop:VariationSlitheringMap} and \ref{prop:InfinitesimalGapFormula}. However, the quantity 
$$
\dot A(\wt T') =
\begin{cases}
\displaystyle
 \sum_{a+b+c=n} \dot\tau_V^{abc}(\wt T', x_{\wt T'}) \left( \dot L_{E_{\wt T'}F_{\wt T'}G_{\wt T'}}^{abc} + \dot S^{a(b+c)}_{E_{\wt T'}G_{\wt T'}}  \right)
 \\
 \displaystyle
 \qquad\qquad + \sum_{a+b=n} \dot \sigma_V^{ab}(\wt T_-, \wt T') \left(  \dot S^{ab}_{E_{\wt T'}F_{\wt T'}} -  \dot S^{ab}_{E_{\wt T'}G_{\wt T'}}  \right)
 \\
 \qquad\qquad  \qquad\qquad  \qquad\qquad  \qquad\qquad  \text{if } \wt T_+ \text{ faces the side } x_{\wt T'}z_{\wt T'}\text{ of }\wt T'
 \\
 \displaystyle
  \sum_{a+b+c=n} \dot\tau_V^{abc}(\wt T', x_{\wt T'}) \left( \dot R_{E_{\wt T'}F_{\wt T'}G_{\wt T'}}^{abc} + \dot S^{(a+c)b}_{G_{\wt T'}F_{\wt T'}}  \right)
 \\
 \displaystyle
 \qquad\qquad + \sum_{a+b=n} \dot \sigma_V^{ab}(\wt T_-, \wt T') \left(  \dot S^{ab}_{E_{\wt T'}F_{\wt T'}} -  \dot S^{ab}_{G_{\wt T'}F_{\wt T'}}  \right)
 \\
 \qquad\qquad  \qquad\qquad  \qquad\qquad  \qquad\qquad  \text{if } \wt T_+ \text{ faces the side } y_{\wt T'}z_{\wt T'}\text{ of } \wt T',
\end{cases}
$$
appearing in Lemma~\ref{lem:EvaluateWeilTie} is not exactly identical to the similar quantity occurring in Proposition~\ref{prop:InfinitesimalGapFormula}, which we rewrite here as
$$
\dot A'(\wt T') =
\begin{cases}
\displaystyle
 \sum_{a+b+c=n} \dot\tau_V^{abc}(\wt T', x_{\wt T'}) \left( \dot L_{E_{\wt T'}F_{\wt T'}G_{\wt T'}}^{abc} + \dot S^{a(b+c)}_{E_{\wt T'}G_{\wt T'}}  \right)
 \\
 \displaystyle
 \qquad\qquad + \sum_{a+b=n} \dot \sigma_V^{ab}(\wt T_0, \wt T') \left(  \dot S^{ab}_{E_{\wt T'}F_{\wt T'}} -  \dot S^{ab}_{E_{\wt T'}G_{\wt T'}}  \right)
 \\
 \qquad\qquad  \qquad\qquad  \qquad\qquad  \qquad\qquad  \text{if } \wt T_+ \text{ faces the side } x_{\wt T'}z_{\wt T'}\text{ of }\wt T'
 \\
 \displaystyle
  \sum_{a+b+c=n} \dot\tau_V^{abc}(\wt T', x_{\wt T'}) \left( \dot R_{E_{\wt T'}F_{\wt T'}G_{\wt T'}}^{abc} + \dot S^{(a+c)b}_{G_{\wt T'}F_{\wt T'}}  \right)
 \\
 \displaystyle
 \qquad\qquad + \sum_{a+b=n} \dot \sigma_V^{ab}(\wt T_0, \wt T') \left(  \dot S^{ab}_{E_{\wt T'}F_{\wt T'}} -  \dot S^{ab}_{G_{\wt T'}F_{\wt T'}}  \right)
 \\
 \qquad\qquad  \qquad\qquad  \qquad\qquad  \qquad\qquad  \text{if } \wt T_+ \text{ faces the side } y_{\wt T'}z_{\wt T'}\text{ of } \wt T'.
\end{cases}
$$
The difference is that the base triangle $\wt T_0$ is replaced by the component $\wt T_-$, which depends on the tie $\wt k$ of the train track neighborhood $\wt \Phi_\delta$. Some additional care is therefore needed to guarantee  uniformity for the estimates. 

For this, we first consider the case where $\wt T_-$ separates $\wt T_+$ from $\wt T_0$. We then use the quasi-additivity property of Proposition~\ref{lem:LiftShearsToUniversalCover} which, after taking derivatives, tells us that
$$
\dot \sigma_V^{ab}(\wt T_0, \wt T') =
\begin{cases}
\displaystyle
 \dot \sigma_V^{ab}(\wt T_0, \wt T_-) + \dot \sigma_V^{ab}(\wt T_-, \wt T') - \sum_{b' + c'=n-a} \dot \tau_V^{ab'c'}(\wt T_-, x_{\wt T_-} )
 \\
  \qquad\qquad  \qquad\qquad  \qquad\qquad  \qquad\qquad  \text{if } \wt T' \text{ faces the side } x_{\wt T_-}z_{\wt T_-}\text{ of }\wt T_-
 \\
 \displaystyle
  \dot \sigma_V^{ab}(\wt T_0, \wt T_-) + \dot \sigma_V^{ab}(\wt T_-, \wt T') - \sum_{a' + c'=n-b} \dot \tau_V^{bc'a'}(\wt T_-, y_{\wt T_-} )
  \\
    \qquad\qquad  \qquad\qquad  \qquad\qquad  \qquad\qquad  \text{if } \wt T' \text{ faces the side } y_{\wt T_-}z_{\wt T_-}\text{ of }\wt T_-.
\end{cases}
$$

This enables us to write the term $\dot A(\wt T')$ of  Lemma~\ref{lem:EvaluateWeilTie} as
$$
\dot A(\wt T') = \dot A'(\wt T') + \dot A''(\wt T')
$$
 where $ \dot A'(\wt T') $ is as above and
 $$
  \dot A''(\wt T') =
  \begin{cases}
 \displaystyle
- \sum_{a+b=n}  \dot \sigma_V^{ab}(\wt T_0, \wt T_-) \left(  \dot S^{ab}_{E_{\wt T'}F_{\wt T'}} -  \dot S^{ab}_{E_{\wt T'}G_{\wt T'}}  \right) 
\\
\displaystyle
\qquad\qquad\qquad+ \sum_{a+b+c=n} \dot \tau_V^{abc}(\wt T_-, x_{\wt T_-} )  \left(  \dot S^{a(b+c)}_{E_{\wt T'}F_{\wt T'}} -  \dot S^{a(b+c)}_{E_{\wt T'}G_{\wt T'}}  \right)
 \\
 \qquad\qquad  \qquad\qquad  \qquad\qquad  \qquad\qquad  \text{if } \wt T_+ \text{ faces the side } x_{\wt T'}z_{\wt T'}\text{ of }\wt T'
 \\
 \displaystyle
 -  \sum_{a+b=n} \dot \sigma_V^{ab}(\wt T_0, \wt T_-) \left(  \dot S^{ab}_{E_{\wt T'}F_{\wt T'}} -  \dot S^{ab}_{G_{\wt T'}F_{\wt T'}}  \right)
 \\
\displaystyle
\qquad\qquad\qquad+ \sum_{a+b+c=n} \dot \tau_V^{abc}(\wt T_-, x_{\wt T_-} ) \left(  \dot S^{(a+c)b}_{E_{\wt T'}F_{\wt T'}} -  \dot S^{(a+c)b}_{G_{\wt T'}F_{\wt T'}}  \right)
 \\
 \qquad\qquad  \qquad\qquad  \qquad\qquad  \qquad\qquad  \text{if } \wt T_+ \text{ faces the side } y_{\wt T'}z_{\wt T'}\text{ of } \wt T'.
\end{cases}
 $$
 
 This will enable us to split the analysis into two steps.
 
\begin{sublem} 
\label{sublem:EstimateWeilZipperOpening1}
For $\dot A'(\wt T') $ as above, there exists $C>0$ such that
$$
 \sum_{\wt T' \text{ between } \wt T_- \text{ and } \wt T_+} \dot A'(\wt T')   =  O ( \E^{-C\delta} ) ,
$$
where the constant $C>0$ and the constant hidden in the Landau symbol $O(\ )$ depend only on a compact subset of $\wt S$ containing $\wt k$, on the Hitchin representation $\rho$, and on the tangent vector $V\in T_{[\rho]}\Hit(S)$. 
\end{sublem}
 
\begin{proof}
This is a simple consequence of the estimates that we  used in the proofs of Propositions~\ref{prop:VariationSlitheringMap} and \ref{prop:InfinitesimalGapFormula}.  

Indeed, in the proof of Proposition~\ref{prop:InfinitesimalGapFormula}, $\dot A'(\wt T') $ (called $\dot A(\wt T') $ in that statement) occurred as the derivative at $t=0$ of a quantity $ A_t(\wt T') $ defined in that proof, which makes sense as a holomorphic function of  $t\in \C$ with $|t|<1$. Lemma~\ref{lem:ElementarySlitheringBoundedByGapLength2} asserts that $ A_t(\wt T') = \Id_{\R^n}+ O\left( \ell(\wt k \cap \wt T')^\nu \right)$ for some $\nu$. An application of Pick's lemma then shows that
$$
 \sum_{\wt T' \text{ between } \wt T_- \text{ and } \wt T_+} \dot A'(\wt T')   =  \sum_{\wt T' \text{ between } \wt T_- \text{ and } \wt T_+}  O\left( \ell(\wt k \cap \wt T')^\nu \right) .
$$

Sublemma~\ref{sublem:EstimateWeilZipperOpening0} shows  that $\ell(\wt k \cap \wt T') = O(\E^{-C'' r(\wt T')})$ for some constant $C''>0$ and, by definition of zipper openings, $r(\wt T')\geq \delta$ for each component $ \wt T'$ occurring in the above sum.  It then follows from the second part of Lemma~\ref{lem:HolderSumOfGapLengths} that 
$$
\sum_{\wt T' \text{ between } \wt T_- \text{ and } \wt T_+}  O\left( \ell(\wt k \cap \wt T')^\nu \right) 
= O(\E^{-\nu' C'' \delta}) = O(\E^{-C \delta})
$$
for any $\nu'<\nu$ and the constant $C=\nu'C''$. This concludes the proof of Sublemma~\ref{sublem:EstimateWeilZipperOpening1}.
\end{proof}
 
 \begin{sublem} 
\label{sublem:EstimateWeilZipperOpening2}
 For $\dot A''(\wt T') $ as above, there exists $C>0$ such that
$$
 \sum_{\wt T' \text{ between } \wt T_- \text{ and } \wt T_+} \dot A''(\wt T')   = O(\E^{-C \delta}) ,
 $$
where the exponent $\nu$ and the constant hidden in the Landau symbol $O(\ )$ depend only on a compact subset of $\wt S$ containing $\wt k$, on the Hitchin representation $\rho$, and on the tangent vector $V\in T_{[\rho]}\Hit(S)$. 
\end{sublem}

\begin{proof}
In order to bound $\dot A''(\wt T') $, let us first focus on the case where $\wt T_+$ faces the side $x_{\wt T'}z_{\wt T'}$ of $\wt T'$. The other case will be identical.  

In this case, 
  \begin{align*}
 \dot A''(\wt T') & =- \sum_{a+b=n}  \dot \sigma_V^{ab}(\wt T_0, \wt T_-) \left(  \dot S^{ab}_{E_{\wt T'}F_{\wt T'}} -  \dot S^{ab}_{E_{\wt T'}G_{\wt T'}}  \right) 
\\
&
\qquad\qquad\qquad+ \sum_{a+b+c=n} \dot \tau_V^{abc}(\wt T_-, x_{\wt T_-} )  \left(  \dot S^{a(b+c)}_{E_{\wt T'}F_{\wt T'}} -  \dot S^{a(b+c)}_{E_{\wt T'}G_{\wt T'}}  \right).
\end{align*}

Because the flag map $\F_\rho \colon \bdry \to \Flag$ is H\"older continuous \cite[Thm. 4.1]{Lab}, 
\begin{align*}
\dot S^{ab}_{E_{\wt T'}F_{\wt T'}} -  \dot S^{ab}_{E_{\wt T'}G_{\wt T'}} 
&= O\Big( d \big(  (E_{\wt T'}, F_{\wt T'}), (E_{\wt T'}, G_{\wt T'}) \big) \Big)
\\
&= O\big( d (  x_{\wt T'}y_{\wt T'}, x_{\wt T'}z_{\wt T'})^\nu \big)
\\
&= O\big( \ell( \wt k \cap \wt T' )^\nu \big)
\end{align*}
for some $\nu>0$, where we use the same symbol $d(\ , \ )$ to denote the distance in the space of flag pairs for the first line, and the distance in the space of geodesics of $\wt S$ for the second line.

The coefficients  $\dot \tau_V^{abc}(\wt T_-, x_{\wt T_-} ) $ are obviously uniformly  bounded, since they take only finitely many values. 

To bound the coefficients $ \dot \sigma_V^{ab}(\wt T_0, \wt T') $, a combinatorial argument using the Switch Relation of \S \ref{subsect:GenFocGonInvariants} (compare \cite[Lem.~6]{Bon97Top}) shows that 
$$
 \dot \sigma_V^{ab}(\wt T_0, \wt T_-) = O\big( r(\wt T_-) \big) $$
with the constant depending only on the tangent vector $V\in T_{[\rho]} \Hit(S)$. Also, $r(T_-)\leq \delta$ by definition of zipper openings. Therefore, 
$$
 \dot \sigma_V^{ab}(\wt T_0, \wt T_-) = O (\delta).  $$

Combining all these estimates gives us that
$$
 \dot A''(\wt T')  = O \left( \delta \ell( \wt k \cap \wt T' )^\nu  \right) ,
$$
 in the case where $\wt T_+$ faces the side $x_{\wt T'}z_{\wt T'}$ of $\wt T'$. 

The same argument gives us an identical  estimate
in the other case, where $\wt T_+$ faces the side $y_{\wt T'}z_{\wt T'}$ of $\wt T'$. 

Now, by another application of Sublemma~\ref{sublem:EstimateWeilZipperOpening0} and Lemma~\ref{lem:HolderSumOfGapLengths},
$$
 \sum_{\wt T' \text{ between } \wt T_- \text{ and } \wt T_+} \dot A''(\wt T')   
 =  \sum_{\wt T' \text{ between } \wt T_- \text{ and } \wt T_+}  O \big(  \delta \ell( \wt k \cap \wt T' )^\nu \big) 
  = O(\delta \E^{-\nu' C'' \delta}) = O(\E^{-C \delta})
$$
for any $\nu'<\nu$ and $C<\nu' C''$. This concludes the proof of Sublemma~\ref{sublem:EstimateWeilZipperOpening2}. 
\end{proof}
 
By taking the smallest of the two, we can arrange that the constants $C$ of  Sublemmas \ref{sublem:EstimateWeilZipperOpening1} and \ref{sublem:EstimateWeilZipperOpening2} coincide. These statements  then prove that
 \begin{align*}
c_V(\wt k) &= \sum_{a+b=n} \dot \sigma^{ab}(\wt T_-, \wt T_+)  \dot S^{ab}_{E_{\wt T_+}F_{\wt T_+}}  
+ \sum_{\wt T' \text{ between } \wt T_- \text{ and } \wt T_+} \dot A'(\wt T') + \sum_{\wt T' \text{ between } \wt T_- \text{ and } \wt T_+} \dot A''(\wt T') 
\\
&=  \sum_{a+b=n} \dot \sigma^{ab}(\wt T_-, \wt T_+)  \dot S^{ab}_{E_{\wt T_+}F_{\wt T_+}}  + O(\E^{-C \delta})
\end{align*}
and completes the proof of Lemma~\ref{lem:EstimateWeilZipperOpening} in the special case considered, where $\wt T_-$ separates $\wt T_0$ from $\wt T_+$.

We will deduce the general case from this special case, but we need some care to guarantee the uniformity of the estimates. The statement of Lemma~\ref{lem:EstimateWeilZipperOpening} involves a compact subset of $\wt S$ containing the tie $\wt k$. This compact subset $K$ meets only finitely branches of the original train track neighborhood $\wt \Phi_0$. 

For each oriented branch $e$ of $\wt\Phi_0$ (namely a branch of $\wt\Phi_0$ with an orientation of the corresponding branch of the train track $\wt \Psi_0$) meeting $K$, the orientation of $e$ determines an orientation for the leaves of $\wt\lambda$ that pass through $e$, and these orientations are all parallel. Among these leaves, let $g_e^\Left$ be the one that is most to the left. Since the fundamental group $\pi_1(S)$ acts minimally on the space of geodesics of $\wt S$, these exists $\gamma_e \in \pi_1(S)$ such that $\gamma_e g_e^\Left$ is oriented to the left as seen from the base triangle $\wt T_0$. This guarantees that the images under $\gamma_e$ of all the leaves of $\wt\lambda$ passing through $e$ are in the same component of $\wt S - \wt T_0$ and are oriented to the left as seen from $\wt T_0$. 

If $\wt k$ is an oriented tie of $\wt \Phi_\delta$ contained in $K$, let $e$ be the branch of $\wt\Phi_0$ that contains it. Orient $e$ so that the orientation that it induces on the leaves of $\wt\lambda$ passing through $\wt k$ points to the left of the orientation of $\wt k$. Then, $\gamma_e \wt T_-$ separates $\wt T_0$ from $\gamma_e \wt T_+$, and we can therefore apply the special case to the tie $\gamma_e \wt k$.  This gives
$$
c_V(\gamma_e \wt k) =  \sum_{a+b=n} \dot \sigma_V^{ab}(\gamma_e \wt T_-, \gamma_e \wt T_+)  \dot S^{ab}_{E_{\gamma_e \wt T_+} F_{\gamma_e \wt T_+}}  + O(\E^{-C \delta}),
$$
where the exponent $\nu$ and the constant in $O(\ )$ depend only on the union on the finitely many compact subsets $\gamma_{e'}K$ as $e'$ ranges over all branches of $\wt\Phi_0$ meeting $K$ (as well as on the Hitchin representation $\rho$ and the tangent vector $V \in T_{[\rho]}\Hit(S)$). Then
\begin{align*}
 c_V( \wt k) &= \Ad\rho(\gamma_e^{-1}) \big( c_V(\gamma_e \wt k) \big)
 \\
 &=  \sum_{a+b=n} \dot \sigma_V^{ab}(\gamma_e \wt T_-, \gamma_e \wt T_+)  \Ad\rho(\gamma_e^{-1}) (\dot S^{ab}_{E_{\gamma_e \wt T_+} F_{\gamma_e \wt T_+}} ) + O(\E^{-C \delta})
 \\
 &=  \sum_{a+b=n} \dot \sigma_V^{ab}( \wt T_-,  \wt T_+)  \dot S^{ab}_{E_{ \wt T_+} F_{ \wt T_+}}  + O(\E^{-C \delta}),
 \end{align*}
where the   constant in $O(\ )$ now includes, in addition, the distorsion in $\sln$ coming from the finitely many isomorphisms $ \Ad\rho(\gamma_{e'}^{-1}) $ that can occur. For the last equality, use the $\pi_1(S)$--equivariance property of Lemma~\ref{lem:LiftShearsToUniversalCover}.   

This completes the proof of Lemma~\ref{lem:EstimateWeilZipperOpening}.
\end{proof}

\section{Computing and estimating the cup-product}
\label{bigsect:ComputeEstimateCupProduct}

In \S \ref{subsect:Atiyah-Bott-Goldman}, we saw that the Atiyah-Bott-Goldman symplectic form is defined in terms of the evaluation on the fundamental class $[S] \in H_2(S; \Z)$ of the cup-product $[c_{V_1}] \cupp [c_{V_2}]\in H^2(S;\R)$ of the Weil classes $[c_{V_1}]$, $[c_{V_2}] \in H^1(S; \sln_{\Ad\rho})$ associated to two tangent vectors $V_1$, $V_2 \in T_{[\rho]} \Hit(S)$. Now that we have an explicit description of the cocycles $c_{V_1}$, $c_{V_2} \in C^1(S; \sln)$, we will use the triangulation of \S \ref{subsect:Triangulation} to compute, or at least estimate in a first step, this cup-product. 

\subsection{A remark on the computation of cup-products}
\label{subsect:RemarkCupProduct}

We begin with a general observation on cup-products in dimension 2. This property was used without  discussion in \cite{SozBon, Zey}, but does not seem to be well-known and therefore may deserve some explanation. 

Let $\Delta_1 = \{(t_0, t_1)\in \R^2; t_0 + t_1 =1\}$ and $\Delta_2 = \{(t_0, t_1, t_2) \in \R^3; t_0 + t_1 +t_2=1\}$ denote the standard 1-- and 2--simplex. A singular 2--simplex $ f \colon \Delta_2 \to S$ has three sides: the \emph{front side} is the 1--simplex $F^{\mathrm{front}}  f \colon \Delta_1 \to S$ defined by $F^{\mathrm{front}} f (t_0, t_1) = (t_0, t_1, 0)$, the \emph{middle side} $F^{\mathrm{mid}}  f$ is defined by $F^{\mathrm{mid}} f (t_0, t_1) = (t_0, 0, t_1)$, and the \emph{back side} $F^{\mathrm{back}}  f$ is defined by $F^{\mathrm{back}} f (t_0, t_1) = ( 0, t_0, t_1)$. Then the boundary $\partial f$ is the 1--chain $F^{\mathrm{back}} f - F^{\mathrm{mid}} f + F^{\mathrm{front}} f$. 

The currently popular definition (see for instance \cite[\S 3.2]{Hat}) of the cup-product $[c_1] \cupp [c_2]\in H^2(S;\R)$ of two homology classes $[c_1]$, $[c_2] \in H^1(S; \sln_{\Ad\rho})$ is the class of the 2--cochain $c_1 \cupp c_2$ defined by
$$
c_1 \cupp c_2 ( f ) = K \big( c_1(F^{\mathrm{front}} \wt f ), c_2(F^{\mathrm{back}} \wt f ) \big) \in \R
$$
where $K \colon \sln \times \sln \to \R$ is the Killing form and where $\wt f$ is an arbitrary lift of $ f$ to the universal cover $\wt S$. The invariance of the Killing form under the adjoint representation guarantees that this does not depend on the choice of the lift $\wt f$. This definition has the advantage of simplicity, but the antisymmetry of the cup-product then becomes a nontrivial result (see for instance \cite[Th.~3.14]{Hat}). 

We now express the fundamental class $[S] \in H_2(S; \Z)$ by a 2--chain based on a triangulation $\Sigma$ of the oriented  surface $S$. For each face $ f$ of this triangulation, choose an arbitrary parametrization of this face as a singular 2--simplex $ f \colon \Delta_2 \to S$, in such a way that the orientation of the front and back sides  $F^{\mathrm{front}} f$, $F^{\mathrm{back}} f$  coincide with the boundary orientation of $ f$ (while the orientation of the middle side $F^{\mathrm{mid}} f$ is the opposite of this boundary orientation); we are here orienting the 1--simplex $\Delta_1$ from $(1,0)$ to $(0,1)$. Let the 2--chain  
$$
\gamma = \sum_ f  f \in C_2(S)
$$ 
be the sum of the simplices $ f$ thus associated to the faces of the triangulation $\Sigma$.  However, this chain is in general not closed. 

An edge $k$ of the triangulation $\Sigma$ separates two faces $ f_1$ and $ f_2$ of $\Sigma$, and therefore is the image of a side $F^{i_1(k)} f_1$ and a side $F^{i_2(k)} f_2$ with $i_1(k)$, $i_2(k) \in \{ \mathrm{front}, \mathrm{back}, \mathrm{mid}\}$. If  $F^{i_1(k)} f_1$ and $F^{i_2(k)} f_2$ induce the same orientation on $k$,  we can arrange that $F^{i_1(k)} f_1=F^{i_2(k)} f_2$ as maps $\Delta_1 \to S$. Note that this happens precisely when $i_1(k)$, $i_2(k)$ are not in the same subset $\{ \mathrm{front}, \mathrm{back}\}$ or $\{ \mathrm{mid}\}$. Consequently,  the two sides  $F^{i_1(k)} f_1$ and $F^{i_2(k)}f_2$ occur with opposite signs in the expressions of $\partial  f_1$ and $\partial f_2$, and they  cancel out in the expression of $\partial \gamma$. In other words, we can arrange that the contribution of the edge $k$ to $\partial \gamma$ is equal to 0, in this case where $F^{i_1(k)} f_1$ and $F^{i_2(k)} f_2$ induce the same orientation on $k$. 

When  $F^{i_1(k)} f_1$ and $F^{i_2(k)} f_2$ induce opposite orientations on $k$, there exists a singular 2-simplex $ f_k \colon \Delta_2 \to S$, with image contained in $k$, whose front side $F^{\mathrm{front}}  f_k$ coincides with $F^{i_1(k)} f_1$, whose back  side $F^{\mathrm{back}} f_k$ coincides with  $F^{i_2(k)} f_2$, and whose middle side $F^{\mathrm{mid}} f_k$ is constant, valued at the initial vertex of the front side $F^{\mathrm{front}} f_k= F^{i_1(k)} f_1$ (which is also the terminal vertex of the back side $F^{\mathrm{back}}  f_k=F^{i_2(k)} f_2$).  Also, let $ f_k' \colon \Delta_2 \to S$ be constant, valued at the same initial vertex of $F^{\mathrm{front}} f_k$. Then, 
\begin{align*}
\partial( f_k +  f_k') &= F^{\mathrm{back}}  f_k - F^{\mathrm{mid}}  f_k + F^{\mathrm{front}}  f_k + F^{\mathrm{back}}  f_k' - F^{\mathrm{mid}}  f_k' + F^{\mathrm{front}}  f_k'
\\
&= F^{i_1(k)} f_1 +  F^{i_2(k)} f_2
\end{align*}
since  $F^{\mathrm{back}} f_k= F^{i_2(k)} f_2$, $F^{\mathrm{front}}  f_k=F^{i_1(k)} f_1$, and $F^{\mathrm{mid}}  f_k$, $ F^{\mathrm{back}}  f_k'$, $ F^{\mathrm{mid}}  f_k'$, $ F^{\mathrm{front}}  f_k'$ are all equal to the same constant map. 

To unify the two cases, let $\gamma_k \in C_2(S)$ be defined by
$$
\gamma_k=
\begin{cases}
-f_k - f_k' &\text{ if } i_1(k), i_2(k) \in \{ \mathrm{front}, \mathrm{back}\}
 \\
f_k + f_k' &\text{ if }  i_1(k) = i_2(k) \in \{ \mathrm{mid} \}
\\
0&\text{ otherwise}.
\end{cases}
$$

We have done everything so that the 2--chain
$$
\gamma_S = \sum_ f  f + \sum_k \gamma_k \in C_2(S; \Z)
$$
is now closed, where the sums are over all faces and over all edges of the triangulation $\Sigma$, respectively. A degree argument applied to a point in the interior of a face shows that the homology class $[\gamma_S] \in H_2 (S;\Z)$ is equal to the fundamental class $[S]$.  

\begin{lem}
\label{lem:GeneralFormulaCupProduct}
 For any two cocycles $c_1$, $c_2 \in C^1(S; \sln_{\Ad\rho})$, the evaluation of the cup-product $[c_1] \cupp [c_2]\in H^2(S;\R)$ over the fundamental class $[S] \in H_2 (S;\Z)$ is equal to 
 $$
 \langle [c_1] \cupp [c_2], [S] \rangle = \tfrac12 \sum_ f  \Big( K \big( c_1(F^{\mathrm{front}} \wt f ), c_2(F^{\mathrm{back}} \wt f ) \big) - K \big( c_2(F^{\mathrm{front}} \wt f ), c_1(F^{\mathrm{back}} \wt f ) \big)  \Big),
 $$
 where the sum ranges over all faces $f$ of the triangulation $\Sigma$ considered as $2$--simplices whose parametrization is compatible with the orientation of $S$. 
\end{lem}

\begin{proof}
 Let us represent $[S]$ by the cycle $\gamma_S$ as above. Then
 $$
 c_1 \cupp c_2(\gamma_S) = \sum_ f  K \big( c_1(F^{\mathrm{front}} \wt f ), c_2(F^{\mathrm{back}} \wt f ) \big)
+ \sum_{k} c_1\cupp c_2(\wt \gamma_k),
 $$
 where $\wt \gamma_k \in C_2(\wt S)$ is the lift of $\gamma_k$ contained in an arbitrary lift $\wt k$ of the edge $k$. 
 
To make the formula more symmetric, we now take advantage of the  nontrivial property (see for instance \cite[Th.~3.14]{Hat}) that the cup-product is antisymmetric at the homological level (although not at the chain level). Namely, $[c_2] \cupp [c_1] =- [c_1] \cupp [c_2] $ in $H^2(S;\R)$. This enables us to rewrite the formula as
\begin{align*}
  \langle [c_1] \cupp [c_2], [S] \rangle
  &= \tfrac12 \big(   \langle [c_1] \cupp [c_2], [S] \rangle -   \langle [c_2] \cupp [c_1], [S] \rangle \big)
  \\
  &=   \tfrac12 \sum_ f  \Big( K \big( c_1(F^{\mathrm{front}} \wt f ), c_2(F^{\mathrm{back}} \wt f ) \big) - K \big( c_2(F^{\mathrm{front}} \wt f ), c_1(F^{\mathrm{back}} \wt f ) \big)  \Big) 
  \\
 & \qquad \qquad\qquad\qquad\qquad
 + \tfrac12 \sum_k \left(c_1 \cupp c_2 - c_2 \cupp c_1 \right)(\wt \gamma_k). 
\end{align*}

We claim that $\left(c_1 \cupp c_2 - c_2 \cupp c_1 \right)(\wt \gamma_k)=0$ for every edge $k$. This is certainly clear when $\gamma_k=0$. Otherwise, $\gamma_k = \pm(f_k + f_k')$. 

Since $c(\kappa)=0$ for any constant 1--simplex $\kappa$ and any 1--cocycle $c$ (as a chain, $\kappa$ is the boundary of a constant 2--simplex), the evaluations of $c_1 \cupp c_2 $ and $ c_2 \cupp c_1$ on the constant simplex $\wt f_k'$ are both equal to $0$. 

Since $\partial \wt f_k = F^{\mathrm{back}} \wt f_k -F^{\mathrm{mid}} \wt f_k +F^{\mathrm{front}} \wt f_k $, we have that
$$
c_i (F^{\mathrm{back}} \wt f_k) = -c_i (F^{\mathrm{front}} \wt f_k) + c_i (F^{\mathrm{mid}} \wt f_k) = -c_i (F^{\mathrm{front}} \wt f_k)
$$
for each $i=1$, $2$, since $c_i$ is closed and the middle side $F^{\mathrm{mid}} \wt f_k$ is constant. It follows that
\begin{align*}
\left(c_1 \cupp c_2 - c_2 \cupp c_1 \right)(\wt f_k) 
&= K \big( c_1(F^{\mathrm{front}} \wt f_k ), c_2(F^{\mathrm{back}} \wt f_k ) \big) - K \big( c_2(F^{\mathrm{front}} \wt f_k ), c_1(F^{\mathrm{back}} \wt f_k ) \big)
\\
&=- K \big( c_1(F^{\mathrm{front}} \wt f_k ), c_2(F^{\mathrm{front}} \wt f_k ) \big) + K \big( c_2(F^{\mathrm{front}} \wt f_k ), c_1(F^{\mathrm{front}} \wt f_k ) \big) =0 .
\end{align*}

This proves that $\left(c_1 \cupp c_2 - c_2 \cupp c_1 \right)(\wt \gamma_k)=0$ for every edge $k$, and completes the computation. 
\end{proof}

The formula of Lemma~\ref{lem:GeneralFormulaCupProduct} will enable us to take advantage of the following symmetry. 

\begin{lem}
 \label{lem:CupProductIndependentParam}
 The formula of Lemma~{\upshape\ref{lem:GeneralFormulaCupProduct}} is independent of the parametrization of each face $ f$ of the triangulation $\Sigma$ as a singular $2$--simplex $ f \colon \Delta_2 \to S$ (as long as this parametrization is compatible with the orientation of $S$, in the sense that the orientation of the front and back sides coincides with the boundary orientation). 
\end{lem}

\begin{proof}
 We need to show invariance if we ``rotate'' such a parametrization $ f \colon \Delta_2 \to S$ to a new parametrization $ f' \colon \Delta_2 \to S$ such that the front side $F^{\mathrm{front}} f'$ is equal to the back side $F^{\mathrm{back}} f$, the middle side $F^{\mathrm{mid}} f'$ is the front side $F^{\mathrm{front}} f$ with the orientation reversed, and the back side $F^{\mathrm{back}} f'$ is the middle side $F^{\mathrm{mid}} f$ with the orientation reversed.  Then
\begin{align*}
 c_i (F^{\mathrm{front}} f') &= c_i (F^{\mathrm{back}} f)
 \\
 c_i (F^{\mathrm{back}} f') &= -c_i (F^{\mathrm{mid}} f) = -c_i (F^{\mathrm{front}} f)-c_i (F^{\mathrm{back}} f) 
\end{align*}
for each $i=1$, $2$. Expanding the contribution
$$
K \big( c_1(F^{\mathrm{front}} \wt f' ), c_2(F^{\mathrm{back}} \wt f' ) \big) - K \big( c_2(F^{\mathrm{front}} \wt f' ), c_1(F^{\mathrm{back}} \wt f' ) \big) 
$$
of this new parametrization $ f'$ and using the bilinearity and symmetry  of the Killing form $K$, we readily see that it is equal to the contribution
$$
K \big( c_1(F^{\mathrm{front}} \wt f ), c_2(F^{\mathrm{back}} \wt f ) \big) - K \big( c_2(F^{\mathrm{front}} \wt f ), c_1(F^{\mathrm{back}} \wt f ) \big) 
$$
of the old parametrization $ f$. 
\end{proof}

\subsection{General setup}
\label{subsect:SetUpComputation}

We now return to our computation of the evaluation 
$
\left\langle [c_{V_1}] \cupp [c_{V_2}], [S]  \right\rangle \in \R
$
of the cup-product $[c_{V_1}] \cupp [c_{V_2}]\in H^2(S;\R)$ on the fundamental class $[S] \in H_2(S; \Z)$, for the Weil classes $[c_{V_1}]$, $[c_{V_2}] \in H^1(S; \sln_{\Ad\rho})$  corresponding to tangent vectors  $V_1$, $V_2 \in T_{[\rho]} \Hit(S)$. For this we will use the description of the classes $[c_{V_1}]$, $[c_{V_2}]$ given in \S \ref{subsect:Barriers} and \S \ref{subsect:EvaluateCocyle}, as well as  the triangulation $\Sigma$ introduced in \S \ref{subsect:Triangulation}. 

A consequence of Lemma~\ref{lem:CupProductIndependentParam} is that, in the formula for $\left\langle [c_{V_1}] \cupp [c_{V_2}], [S]  \right\rangle$ given by Lemma~\ref{lem:GeneralFormulaCupProduct},  the contribution of most faces of the triangulation $\Sigma$  is equal to 0. 

\begin{lem}
\label{lem:CupProductNoContribution}
 Suppose that at least one side of the face $f$ of the triangulation $\Sigma$  is disjoint from the geodesic lamination $\lambda$ and from the barriers $\beta_T$. Then the contribution of $f$ to the formula for $\left\langle [c_{V_1}] \cupp [c_{V_2}], [S]  \right\rangle$ given by Lemma~{\upshape\ref{lem:GeneralFormulaCupProduct}} is equal to $0$.
\end{lem}
\begin{proof}
Using the flexibility provided by  Lemma~\ref{lem:CupProductIndependentParam}, we can choose the parametrization of $f$ so that its front face $F^{\mathrm{front}}f$ is disjoint from $\lambda$ and from the barriers $\beta_T$.  Then, Lemma~\ref{lem:EvaluateWeilEdgeDisjoinLaminationBarrier} shows that 
 $c_{V_1}(\wt F^{\mathrm{front}}f)=c_{V_2}(\wt F^{\mathrm{front}}f)=0$, and it  follows that the contribution of $f$ to the formula of  Lemma~\ref{lem:GeneralFormulaCupProduct} is equal to 0.
\end{proof}

In the triangulation $\Sigma$ of \S \ref{subsect:Triangulation}, there are only a few exceptional faces whose  three sides all meet the union of the geodesic lamination $\lambda$ and the barriers $\beta_T$, and may therefore have a nontrivial contribution to $\left\langle [c_{V_1}] \cupp [c_{V_2}], [S]  \right\rangle$.

The first type is that of the face $ f$ that contains the central point $\pi_T$ of a barrier $\beta_T$. By our assumption that each edge of $\Sigma$ that is outside of the  train track neighborhood meets the barriers in at most one point, each side of $ f$ meets the barriers in exactly one point.

The second type occurs near the switches of $\Phi$. More precisely, near each switch, there are two faces of $\Sigma$ whose sides all meet the union  of $\lambda$ and of the barriers $\beta_T$. See Figure~\ref{fig:TriangulationNearSwitch} below. 

By inspection of the construction of the triangulation $\Sigma$, these triangles, one for each component of $\Phi-S$ and two for each switch, are the only ones whose contribution to $\left\langle [c_{V_1}] \cupp[c_{V_2}], [S]  \right\rangle $ may be nonzero. We will now compute the contribution of these faces.

\subsection{The contribution of the complement of the train track}
\label{subsect:ContributionTriangle}

We consider the first type of faces of $\Sigma$ that may have a nontrivial contribution to $\left\langle [c_{V_1}] \cupp[c_{V_2}], [S]  \right\rangle $, namely a face $ f$ that contains the center $\pi_T$ of a barrier $\beta_T$. By construction of the triangulation $\Sigma$ in \S \ref{subsect:Triangulation}, there is exactly one such face in each component of the complement $S-\Phi$ of the train track neighborhood $\Phi$. 

Consider the component $T$ of $S-\lambda$ containing this face $ f$. For the train track $\Psi$ associated to the train track neighborhood $\Phi$,  let $U$ be the component of $S-\Psi$ associated to $T$ by the correspondence of Proposition~\ref{prop:ComplementsTrainTrackGeodLamination}. Arbitrarily pick a boundary corner $s_U$ of $U$, corresponding to a switch of $\Psi$.

Lift $T$ to  a component $\wt T$ of $\wt S - \wt\lambda$, and $ f$ to a triangle $\wt  f$ contained in $\wt T$. Also, let $x_{\wt T} \in \bdry$ be the vertex of $\wt T$ associated to the corner $s_U$ by the correspondence of Proposition~\ref{prop:ComplementsTrainTrackGeodLamination}, and index the other two vertices as $y_{\wt T}$, $z_{\wt T}$ in such a way that  $x_{\wt T}$, $y_{\wt T}$, $z_{\wt T} \in \bdry$ occur counterclockwise in this order around $\wt T$. As usual, let $E_{\wt T}=\F_\rho(x_{\wt T})$, $F_{\wt T}=\F_\rho(y_{\wt T})$, $G_{\wt T}=\F_\rho(z_{\wt T})\in \Flag$ be the flags associated to these vertices by the flag map $\F_\rho$. 

Finally, let $\dot \tau_V^{abc}(U, s_U) \in \R$ denote the directional derivative of the triangle invariant  $\tau_{\rho}^{abc}(U, s_U)$ in the direction of the tangent vector  $V \in T_{[\rho]}\Hit(S)$. Namely,
$$
\dot \tau_V^{abc}(U, s_U) = \frac d{dt}  \tau^{abc}_{\rho_t}(U, s_U)_{|t=0} 
$$
for every curve $t\mapsto [\rho_t] $ with $\rho_0=\rho$ and $\frac d{dt} [\rho_t]_{|t=0}=V$. 

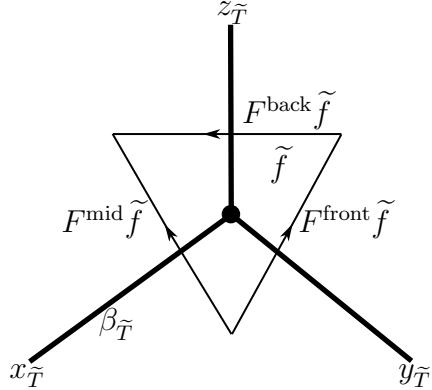
\begin{figure}[htbp]

\SetLabels
( .65 * .55 )  $\wt f$ \\
\L(  .7* .39 )  $F^{\mathrm{front}}\wt f$ \\
(  .67* .71 )   $F^{\mathrm{back}}\wt f$\\
\R( .3 * .39 )   $F^{\mathrm{mid}}\wt f$\\
\T(0  * 0 )   $x_{\wt T}$\\
\T( 1 * 0 )   $y_{\wt T}$\\
\B( .53 *  1)   $z_{\wt T}$\\
\B( .23 *  .07)   $\beta_{\wt T}$\\
\endSetLabels
\centerline{\AffixLabels{\includegraphics{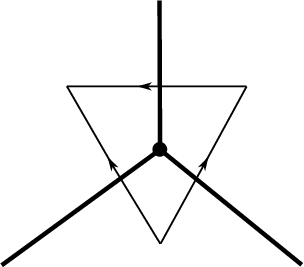}}}

\caption{A face $\wt f$ containing the center $\pi_{\wt T}$ of a barrier $\beta_{\wt T}$}
\label{fig:FaceContainsBarrierCenter}
\end{figure}

\begin{lem}
\label{lem:ContribTriangleToCupProd1}
 The contribution of the face $f$ to the formula for $\left\langle [c_{V_1}] \cupp [c_{V_2}], [S]  \right\rangle$ given by Lemma~{\upshape\ref{lem:GeneralFormulaCupProduct}} is equal to
$$
\tfrac12
\sum_{\substack{a+b+c=n\\a'+b'+c'=n}}
\left( \dot\tau_{V_1}^{abc}(U, s_U) \,\dot\tau_{V_2}^{a'b'c'}(U, s_U) - \dot\tau_{V_2}^{abc}(U, s_U) \, \dot\tau_{V_1}^{a'b'c'}(U, s_U) \right) K\bigl(\dot R_{E_{\wt T}F_{\wt T}G_{\wt T}}^{abc}, \dot L_{E_{\wt T}F_{\wt T}G_{\wt T}}^{a'b'c'} \bigr),
$$
where $K \colon \sln \times \sln \to \R$ is the Killing form, and where $\dot L_{EFG}^{abc}$ and $\dot R_{EFG}^{abc} $ are the infinitesimal left and right eruptions of \S {\upshape \ref{subsect:InfinitesimalVariationFlagMap}}. 
\end{lem}

\begin{proof}  The barrier $\beta_{\wt T}$ divides $\wt T$ into three triangles, each adjacent to one side of $\wt T$. With the flexibility provided by Lemma~\ref{lem:CupProductIndependentParam}, choose the parametrization of the face $f$ so that the front side $F^{\mathrm{front}}\wt f$ goes from the component of $\wt T-\beta_{\wt T}$ that is adjacent to $x_{\wt T}y_{\wt T}$ to the component adjacent to $y_{\wt T}z_{\wt T}$. Then the back side $F^{\mathrm{back}}\wt f$ goes from the component adjacent to $y_{\wt T}z_{\wt T}$ to the component adjacent to $x_{\wt T}z_{\wt T}$, and the middle side $F^{\mathrm{mid}}\wt f$ goes from the component adjacent to $x_{\wt T}y_{\wt T}$ to the component adjacent to $x_{\wt T}z_{\wt T}$. 

The contribution of $f$ to $\left\langle [c_{V_1}] \cupp [c_{V_2}], [S]  \right\rangle$ is then equal to 
$$
 \tfrac12  \Big( K \big( c_{V_1}(F^{\mathrm{front}} \wt f ), c_{V_2}(F^{\mathrm{back}} \wt f ) \big) - K \big( c_{V_2}(F^{\mathrm{front}} \wt f ), c_{V_1}(F^{\mathrm{back}} \wt f ) \big)  \Big),
$$
which is also equal to 
$$
 \tfrac12  \Big( K \big( c_{V_1}(F^{\mathrm{front}} \wt f ), c_{V_2}(F^{\mathrm{mid}} \wt f ) \big) - K \big( c_{V_2}(F^{\mathrm{front}} \wt f ), c_{V_1}(F^{\mathrm{mid}} \wt f ) \big)  \Big),
$$
since $c_{V_i}(F^{\mathrm{mid}} \wt f ) =c_{V_i}(F^{\mathrm{front}} \wt f )  + c_{V_i}(F^{\mathrm{back}} \wt f ) $ as  each $c_{V_i}$ is closed. 

Lemma~\ref{lem:EvaluateWeilEdgeMeetingBarrier} then provides
\begin{align*}
 c_{V}(F^{\mathrm{front}} \wt f ) &= \sum_{a+b+c=n} \dot \tau_V^{abc}(\wt T, x_{\wt T}) \dot R_{E_{\wt T}F_{\wt T}G_{\wt T}}^{abc}
 \\
  c_{V}(F^{\mathrm{mid}} \wt f ) &= \sum_{a'+b'+c'=n} \dot \tau_V^{a'b'c'}(\wt T, x_{\wt T}) \dot L_{E_{\wt T}F_{\wt T}G_{\wt T}}^{a'b'c'}
\end{align*}
for each $V=V_1$ or $V_2$. Since $\dot \tau_V^{abc}(\wt T, x_{\wt T}) =\dot \tau_V^{abc}(U, s_U) $ by the correspondence of Lemma~\ref{lem:LiftTriangleInvariantsToUniversalCover}, the result immediately follows.
\end{proof}

We now compute the coefficient $K\bigl(\dot R_{E_{\wt T}F_{\wt T}G_{\wt T}}^{abc}, \dot L_{E_{\wt T}F_{\wt T}G_{\wt T}}^{a'b'c'} \bigr)$. 

\begin{lem}
\label{lem:ComputeKillingFormLeftRight}
For any maximum-span flag triple $(E,F,G) \in \Flag^3$, 
 $$
 K\bigl(\dot R_{EFG}^{abc}, \dot L_{EFG}^{a'b'c'} \bigr) =
\begin{cases}
2a'b & \text{if } a\geq a' \text{ or } b \leq b'
 \\
  2ab' -2bc'+2b'c  &\text{if } a\leq a', b \geq b' \text{ and } c \geq c'
 \\
  2ab' +2ac'-2a'c  &\text{if } a\leq a', b \geq b' \text{ and } c \leq c'.
 \end{cases}
 $$
\end{lem}

Note the remarkable, and unexpected, property that $ K\bigl(\dot R_{EFG}^{abc}, \dot L_{EFG}^{a'b'c'} \bigr)$ is independent of the maximum-span flag triple $\bigl( E,F,G\bigr)$ (whereas Proposition~\ref{prop:TripleRatioClassifyFlagTriples} shows that  the moduli space of these triples has dimension $\frac12(n-1)(n-2)$). 

\begin{proof}
 This is essentially \cite[Lem.~6.7]{SunZha}, which is here restated with different notation and conventions. 
 
By definition, 
\begin{align*}
  \dot L_{EFG}^{a'b'c'} &= \tfrac{-b'-c'}n \Id_{E^{(a')}} \oplus \tfrac{a'}n \Id_{F^{(b')}} \oplus \tfrac{a'}n \Id_{G^{(c')}}
 \\ &= \tfrac {a'} n \Id_{\R^n} - \Id_{E^{(a')}} \oplus 0\, \Id_{F^{(b')}} \oplus0\, \Id_{G^{(c')}}  =   \tfrac {a'} n \Id_{\R^n} -P_{EFG}^{a'b'c'} 
 \\
 \dot R_{EFG}^{abc} &= \tfrac{-b}n \Id_{E^{(a)}} \oplus \tfrac{a+c}n \Id_{F^{(b)}} \oplus \tfrac{-b}n \Id_{G^{(c)}}
 \\ &= 0\,\Id_{E^{(a)}} \oplus  \Id_{F^{(b)}} \oplus 0\,\Id_{G^{(c)}} - \tfrac bn \Id_{\R^n} = P_{FGE}^{bca} - \tfrac bn \Id_{\R^n}
\end{align*}
for the two projection maps
\begin{align*} 
P_{EFG}^{a'b'c'} &= \Id_{E^{(a')}} \oplus 0\, \Id_{F^{(b')}} \oplus 0\,\Id_{G^{(c')}} 
\\
P_{FGE}^{bca} &= 0\,\Id_{E^{(a)}} \oplus  \Id_{F^{(b)}} \oplus 0\,\Id_{G^{(c)}} .
\end{align*}

As a consequence, 
\begin{align*}
 K\bigl(\dot R_{EFG}^{abc}, \dot L_{EFG}^{a'b'c'} \bigr) &= 2n\,\mathrm{Tr}\, \dot R_{EFG}^{abc}\dot L_{EFG}^{a'b'c'}
 \\
 &=2n\, \mathrm{Tr} \left(P_{FGE}^{bca} - \tfrac bn \Id_{\R^n} \right) \left(  \tfrac {a'} n \Id_{\R^n} -P_{EFG}^{a'b'c'} \right)
 \\
 &= - 2n\, \mathrm{Tr}\, P_{FGE}^{bca}P_{EFG}^{a'b'c'} 
+ 2a' \,\mathrm{Tr}\, P_{FGE}^{bca} +2 b\, \mathrm{Tr}\, P_{EFG}^{a'b'c'}
-2 \tfrac{a'b}{n}  \mathrm{Tr}\,  \Id_{\R^n}
\\
 &=  -2n\, \mathrm{Tr}\, P_{FGE}^{bca}P_{EFG}^{a'b'c'} 
+ 2 a' b + 2b a'
- 2\tfrac{a'b}{n} n
\\
 &= -2n\, \mathrm{Tr}\, P_{FGE}^{bca}P_{EFG}^{a'b'c'} 
 + 2 a' b .
\end{align*}

To compute $ \mathrm{Tr}\, P_{FGE}^{bca}P_{EFG}^{a'b'c'} $, we distinguish cases. Note that these cases will overlap. 

\medskip
\noindent\textsc{Case 1.} $a \geq a'$. 

The image of the projection $P_{EFG}^{a'b'c'}$ is equal to $E^{(a')} \subset E^{(a)}$, and is therefore contained in the kernel $E^{(a)} \oplus G^{(c)}$ of $P_{FGE}^{bca}$. Therefore,
$$
\mathrm{Tr}\, P_{FGE}^{bca}P_{EFG}^{a'b'c'}  = \mathrm{Tr}\, 0 =0
$$
and
$$
 K\bigl(\dot R_{EFG}^{abc}, \dot L_{EFG}^{a'b'c'} \bigr)  = -2n\, \mathrm{Tr}\, P_{FGE}^{bca}P_{EFG}^{a'b'c'} 
 + 2 a' b = 2a'b.
$$

\medskip
\noindent\textsc{Case 2.} $b \leq b'$. 

Similarly, the image $F^{(b)}$ of  $P_{FGE}^{bca} $ is contained in the kernel $F^{(b')} \oplus G^{(c')}$ of $P_{EFG}^{a'b'c'}  $, which implies that
$$
\mathrm{Tr}\, P_{FGE}^{bca}P_{EFG}^{a'b'c'}  =\mathrm{Tr}\, P_{EFG}^{a'b'c'} P_{FGE}^{bca} = \mathrm{Tr}\, 0 =0 
$$
and
$$
 K\bigl(\dot R_{EFG}^{abc}, \dot L_{EFG}^{a'b'c'} \bigr)  = -2n\, \mathrm{Tr}\, P_{FGE}^{bca}P_{EFG}^{a'b'c'} 
 + 2 a' b = 2a'b.
$$

\medskip
\noindent\textsc{Case 3.} $a\leq a'$, $b\geq b'$ and $c\geq c'$. 
Here we need to be more explicit. Let $\{e_1, e_2, \dots, e_n\}$ be a basis for $\R^n$ that is adapted to the flag $E$, in the sense that each subspace $E^{(a'')}$ is spanned by the first $a''$ vectors $e_1$, $e_2$, \dots, $e_{a''}$. Similarly, let $\{f_1, f_2, \dots, f_n \}$ be a basis adapted to $F$, and let $\{g_1, g_2, \dots, g_n \}$ be a basis adapted to $G$. We will express the composition $P_{EFG}^{a'b'c'} P_{FGE}^{bca}$ in the basis $\mathcal B= \{e_1, e_2, \dots, e_a, f_1, f_2, \dots, f_b, g_1, g_2, \dots, g_c\}$ for $\R^n=E^{(a)} \oplus F^{(b)} \oplus G^{(c)}$. 

The map $P_{FGE}^{bca}$ sends each $e_1$, $e_2$, \dots, $e_a$ and each $g_1$, $g_2$, \dots, $g_c$ to 0. Also, for $i \leq b'\leq b$, 
$ P_{EFG}^{a'b'c'} P_{FGE}^{bca}(f_i)=P_{EFG}^{a'b'c'}  (f_i)=0$. Therefore, in the matrix expressing $ P_{EFG}^{a'b'c'}P_{FGE}^{bca} $ in the basis $\mathcal B$, the only columns that may be nonzero are those corresponding to the images of the basis elements $f_{b'+1}$, $f_{b'+2}$, \dots, $f_b$ (with no such column when $b=b'$). 

Express these $f_i$, with $b'<i\leq b$, in the basis $\mathcal B'= \{e_1, e_2, \dots, e_{a'}, f_1, f_2, \dots, f_{b'}, g_1, g_2, \dots, g_{c'}\}$. Namely,
$$
f_i = \sum_{j=1}^{a'} \lambda_{ij}e_j + \sum_{j=1}^{b'} \mu_{ij} f_j + \sum_{j=1}^{c'} \nu_{ij} g_j
$$
for some coefficients $\lambda_{ij}$, $\mu_{ij}$, $\nu_{ij} \in \R$. Then,
\begin{align*}
P_{EFG}^{a'b'c'} P_{FGE}^{bca}(f_i)
 &=P_{EFG}^{a'b'c'}  (f_i)
 = P_{EFG}^{a'b'c'} \left(  \sum_{j=1}^{a'} \lambda_{ij}e_j + \sum_{j=1}^{b'} \mu_{ij} f_j + \sum_{j=1}^{c'} \nu_{ij} g_j \right)
 \\
 &=  \sum_{j=1}^{a'} \lambda_{ij}e_j
 = f_i - \sum_{j=1}^{b'} \mu_{ij} f_j - \sum_{j=1}^{c'} \nu_{ij} g_j 
 \end{align*}
 which, because $i \leq b$, $b'\leq b$ and $c'\leq c$, expresses $P_{EFG}^{a'b'c'} P_{FGE}^{bca}(f_i)$ in the basis $\mathcal B$. 
 
 Therefore, in the matrix expressing $P_{EFG}^{a'b'c'} P_{FGE}^{bca}$ in the basis $\mathcal B$, the column corresponding to $f_i$ with $b'<i\leq b$ has an entry $1$ on the diagonal and possibly a few more nonzero entries above and below this diagonal. As a consequence, the diagonal entries of this matrix consist of $a+b'$ terms 0, followed by $b-b'$ terms 1, and then another $c$ terms $0$.  
 
 This proves that 
 $$
\mathrm{Tr}\, P_{FGE}^{bca}P_{EFG}^{a'b'c'} 
=\mathrm{Tr}\,P_{EFG}^{a'b'c'} P_{FGE}^{bca}
=b-b'. 
$$

Therefore, 
\begin{align*}
  K\bigl(\dot R_{EFG}^{abc}, \dot L_{EFG}^{a'b'c'} \bigr) 
  &= -2n \mathrm{Tr}\, P_{FGE}^{bca}P_{EFG}^{a'b'c'}  + 2a'b
  \\
   &=-2n(b-b') + 2a'b= 2ab' -2bc'+2b'c.
\end{align*}
by using the properties that $a+b+c=a'+b'+c'=n$.  

We now reach the last case.

\medskip
\noindent\textsc{Case 4.} $a\leq a'$, $b\geq b'$ and $c\leq c'$. 

The argument is very similar to the previous one. The difference is that we now express $P_{FGE}^{bca}P_{EFG}^{a'b'c'}  $ in the basis $\mathcal B'= \{e_1, e_2, \dots, e_{a'}, f_1, f_2, \dots, f_{b'}, g_1, g_2, \dots, g_{c'}\}$ for $\R^n$. 

The vectors  $f_1$, $f_2$, \dots, $f_{b'}$, $g_1$, $g_2$, \dots, $g_{c'}$ are sent to 0 by the projection $P_{EFG}^{a'b'c'} $. Also, for $i \leq a \leq a'$, $P_{FGE}^{bca}P_{EFG}^{a'b'c'} (e_i) =P_{FGE}^{bca}(e_i) = 0$. Therefore, in the matrix expressing  $P_{EFG}^{abc}P_{FGE}^{b'c'a'}  $ in the basis $\mathcal B'$, the columns corresponding to $e_1$, $e_2$, \dots, $e_{a}$, $f_1$, $f_2$, \dots, $f_{b'}$, $g_1$, $g_2$, \dots, $g_c$ are all equal to 0. 

Express the remaining vectors $e_i$, with $a<i\leq a'$, in our earlier basis $\B$ as
$$
e_i = \sum_{j=1}^{a} \lambda'_{ij}e_j + \sum_{j=1}^{b} \mu'_{ij} f_j + \sum_{j=1}^{c} \nu'_{ij} g_j
$$ 
for some coefficients $\lambda_{ij}'$, $\mu_{ij}'$, $\nu_{ij}' \in \R$. Then, as in the previous case, its image $P_{FGE}^{bca}P_{EFG}^{a'b'c'} (e_i)$ can be expressed in the basis $\B'$ by
\begin{align*}
P_{FGE}^{bca}P_{EFG}^{a'b'c'}  (e_i)
 &=P_{FGE}^{bca} (e_i)
 =P_{FGE}^{bca} \left(  \sum_{j=1}^{a} \lambda'_{ij}e_j + \sum_{j=1}^{b} \mu'_{ij} f_j + \sum_{j=1}^{c} \nu'_{ij} g_j \right)
 \\
 &=\sum_{j=1}^{b} \mu'_{ij} f_j 
 = e_i - \sum_{j=1}^{a} \lambda'_{ij}e_j - \sum_{j=1}^{c} \nu'_{ij} g_j .
 \end{align*}
 Because $i>a$ this shows that, in the matrix expressing $P_{FGE}^{bca}P_{EFG}^{a'b'c'} $ in the basis $\B'$, the diagonal entry of the $i$--th column with $a<i\leq a'$ is equal to 1. 
 
 The conclusion is again that the diagonal entries of this matrix consist of $a$ terms $0$, followed by $a'-a$ terms 1 and then $b'+c'$ terms 0. 
As a consequence, 
$$
\mathrm{Tr}\,P_{FGE}^{bca}P_{EFG}^{a'b'c'}  = a'-a
$$
and
\begin{align*}
  K\bigl(\dot R_{EFG}^{abc}, \dot L_{EFG}^{a'b'c'} \bigr) 
  &= -2n \mathrm{Tr}\, P_{FGE}^{bca}P_{EFG}^{a'b'c'}  + 2a'b
  \\
   &=-2n(a'-a) + 2a'b= 2ab' +2ac'-2a'c,
\end{align*}
using again the properties that $a+b+c=a'+b'+c'=n$.  
\end{proof}

We can now return to the contribution of the face $f$ to the formula of Lemma~\ref{lem:GeneralFormulaCupProduct}. 

\begin{lem}
\label{lem:ContribTriangleToCupProd2}
Let $f$ be the face of the triangulation $\Sigma$ that contains the center $\pi_T$ of the barrier $\beta_T$ of $T$ and, in the complement $S-\Psi$ of the train track $\Psi$,  let $U$ be the component of  $S-\Psi$  corresponding to the component of $S-\Phi$ that contains $f$. Then, for an arbitrary corner $s_U$ of the triangle $U$, the contribution of $f$ to the evaluation $\left\langle [c_{V_1}] \cupp [c_{V_2}], [S] \right\rangle$ of the cup product is equal to
 $$
\tfrac12 \sum_{\substack{a+b+c=n\\a'+b'+c'=n}}
 C^{abc}_{a'b'c'}\left( \dot\tau_{V_1}^{abc}(U, s_U) \,\dot\tau_{V_2}^{a'b'c'}(U, s_U) -\dot\tau_{V_2}^{abc}(U, s_U) \, \dot\tau_{V_1}^{a'b'c'}(U, s_U)  \right)
$$
where, as in Lemma~{\upshape \ref{lem:ContribTriangleToCupProd1}}, $\dot\tau_{V}^{abc}(U, s_U) $ denotes the directional derivative of $\tau_{\rho}^{abc}(U, s_U) $ in the direction of the tangent vector $V\in T_{[\rho]}\Hit(S)$ and where
$$
C^{abc}_{a'b'c'}=
\begin{cases}
a'b-ab' &\text{if } a\leq a', b\leq b' \text{ and } c\geq c'\\
ac'-a'c &\text{if } a\leq a', b\geq b' \text{ and } c\leq c'\\
b'c-bc' &\text{if } a\leq a', b\geq b' \text{ and } c\geq c'\\
b'c-bc' &\text{if } a\geq a', b\leq b' \text{ and } c\leq c'\\
ac'-a'c &\text{if } a\geq a', b\leq b' \text{ and } c\geq c'\\
a'b -ab'&\text{if } a\geq a', b\geq b' \text{ and } c\leq c'.
\end{cases}
$$
\end{lem}

\begin{proof} Using the antisymmetry of the formula of Lemma~\ref{lem:ContribTriangleToCupProd1}, it can be rewritten as
\begin{align*}
\tfrac14
\sum_{\substack{a+b+c=n\\a'+b'+c'=n}}
&\left( \dot\tau_{V_1}^{abc}(U, s_U) \,\dot\tau_{V_2}^{a'b'c'}(U, s_U) - \dot\tau_{V_2}^{abc}(U, s_U) \, \dot\tau_{V_1}^{a'b'c'}(U, s_U) \right) 
\\
&\qquad\qquad\qquad
\left( K\bigl(\dot R_{E_{\wt T}F_{\wt T}G_{\wt T}}^{abc}, \dot L_{E_{\wt T}F_{\wt T}G_{\wt T}}^{a'b'c'}\bigr) - K\bigl(\dot R_{E_{\wt T}F_{\wt T}G_{\wt T}}^{a'b'c'}, \dot L_{E_{\wt T}F_{\wt T}G_{\wt T}}^{abc} \bigr) 
\right).
\end{align*}
The result then follows from a case-by-case application of Lemma~\ref{lem:ComputeKillingFormLeftRight}.
\end{proof}

Note the rotational symmetry $C^{bca}_{b'c'a'}=C^{abc}_{a'b'c'}$ of  the constants $C^{abc}_{a'b'c'}$ which shows that the output of Lemma~\ref{lem:ContribTriangleToCupProd2} is, as expected, independent of the choice of the corner $s_U$ for $U$. 

\subsection{The contribution of the switches of the train track}
\label{subsect:ContributionSwitch}

Near a switch $s$ of the train track $\Psi$ or, more precisely, near the corresponding switch tie of the  train track neighborhood $\Phi$, there are exactly two faces $f_1$ and $f_2$ of the triangulation $\Sigma$ which may have a nontrivial contribution to the evaluation $\left\langle [c_{V_1}] \cupp [c_{V_2}], [S] \right\rangle$, because their  sides all meet the union of the geodesic lamination $\lambda$ and of the barriers $\beta_T$ (see Lemma~\ref{lem:CupProductNoContribution}). These are illustrated in Figure~\ref{fig:TriangulationNearSwitch}.

Let $k_s^\Left$ and $k_s^\Right$ be two generic ties of the incoming branches $e_s^\Left$ and $e_s^\Right$, respectively, and orient these ties to the left as seen from the switch $s$. Also, let $k_s^{\mathrm{switch}}$ be a small arc meeting  in one point the branch of the barrier $\beta_T$ that enters $\Phi$ near $s$, and again orient this arc to the left as seen from $s$. See Figure~\ref{fig:TriangulationNearSwitch}. 

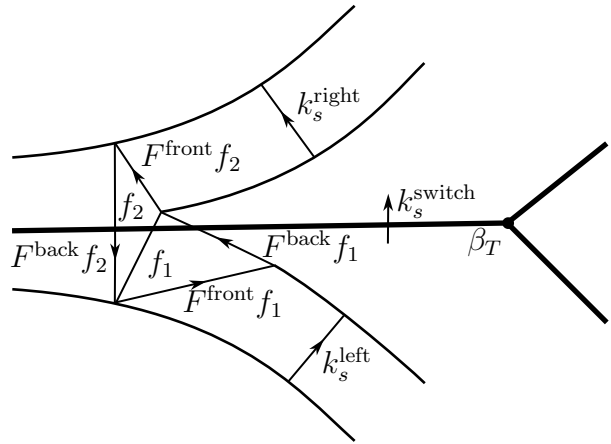
\begin{figure}[htbp]

\SetLabels
( .2 * .52 ) $f_2$  \\
( .25 * .39 ) $f_1$  \\
( .08 * .4 ) $F^{\mathrm{back}}f_2$  \\
( .37 *.3  )    $F^{\mathrm{front}}f_1$ \\
(  .3* .63 )  $F^{\mathrm{front}}f_2$ \\
( .5 * .425 ) $F^{\mathrm{back}}f_1$ \\
(.79 * .44 ) $\beta_T$ \\
( .56 * .16 ) $k_s^\Left$\\
( .53 * .75 ) $k_s^\Right$\\
( .71 * .53 ) $k_s^{\mathrm{switch}}$\\
\endSetLabels
\centerline{\AffixLabels{\includegraphics{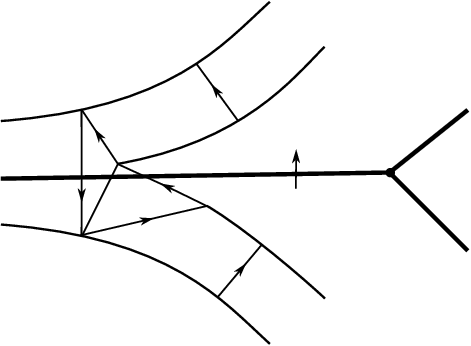}}}

\caption{The contribution of a switch}
\label{fig:TriangulationNearSwitch}
\end{figure}

Lift the switch tie $s$ to a switch tie $\wt s$ of the train track neighborhood $\wt \Phi$ in the universal cover $\wt S$. Then lift $k_s^\Left$, $k_s^\Right$, $k_s^{\mathrm{switch}}$ to oriented arcs $\wt k_s^\Left$, $\wt k_s^\Right$, $\wt k_s^{\mathrm{switch}}$ near $\wt s$. 

\begin{lem}
\label{lem:ContribSwitchToCupProd1}
The contribution of the switch $s$ to  the evaluation $\left\langle [c_{V_1}] \cupp [c_{V_2}], [S] \right\rangle$, namely the sum of the respective contributions of the faces $f_1$ and $f_2$ to the formula of Lemma~{\upshape\ref{lem:GeneralFormulaCupProduct}},  is equal to 
\begin{align*}
& \tfrac12 \Bigl( K\bigl(c_{V_1}(\wt k_s^\Left),  c_{V_2}(\wt k_s^\Right) \bigr)-K\bigl(c_{V_2}(\wt k_s^\Left),  c_{V_1}(\wt k_s^\Right) \bigr)  \Bigr) 
\\
& \qquad
+\tfrac12 \Bigl( K\bigl(c_{V_1}(\wt k_s^\Left),  c_{V_2}(\wt k_s^{\mathrm{switch}}) \bigr)-K\bigl(c_{V_2}(\wt k_s^\Left),  c_{V_1}(\wt k_s^{\mathrm{switch}}) \bigr)  \Bigr) 
\\
&\qquad
-\tfrac12 \Bigl( K\bigl(c_{V_1}(\wt k_s^\Right),  c_{V_2}(\wt k_s^{\mathrm{switch}}) \bigr)-K\bigl(c_{V_2}(\wt k_s^\Right),  c_{V_1}(\wt k_s^{\mathrm{switch}}) \bigr)  \Bigr) .
\end{align*}
\end{lem}

\begin{proof}
Choose the parametrization of the faces $f_1$, $f_2$ so that the sides  $F^{\mathrm{front}}f_1$,  $F^{\mathrm{back}}f_1$,  $F^{\mathrm{front}}f_2$,  $F^{\mathrm{back}}f_2$ are as indicated in  Figure~\ref{fig:TriangulationNearSwitch}.  Then, in the  triangulation $\wt \Sigma$ of the universal covering $\wt S$, lift $f_1$ and $f_2$ and their parametrizations to the corresponding two faces  $\wt f_1$ and $\wt f_2$ that are near $\wt s$. 

By construction, there is an arc $\wt k_+$ that is disjoint from the union of the geodesic lamination $\wt \lambda$ and the barriers $\beta_{\wt T}$ and goes from the positive endpoint of $F^{\mathrm{front}} \wt f_1$ to the positive endpoint of $\wt k_s^\Left$; for instance, the arc can be chosen in the boundary of $\wt\Phi$. Similarly, pick an arc $\wt k_-$ that is disjoint from the union of $\wt \lambda$ and the  $\beta_{\wt T}$ and goes from the negative endpoint of $F^{\mathrm{front}} \wt f_1$ to the negative endpoint of $\wt k_s^\Left$. Then, the chain $F^{\mathrm{front}} \wt f_1 + \wt k_+ - \wt k_s^\Left - \wt k_-$ is homologous to 0, and
$$
c_{V_i}(F^{\mathrm{front}} \wt f_1 ) = -c_{V_i}(\wt k_+ ) + c_{V_i}( \wt k_s^\Left) + c_{V_i}(\wt k_-) = c_{V_i}( \wt k_s^\Left) 
$$
for each $i=1$, $2$, since $c_{V_i}$ is closed and $c_{V_i}(\wt k_\pm ) =0$ by Lemma~\ref{lem:EvaluateWeilEdgeDisjoinLaminationBarrier}. 

Similarly, 
$$
c_{V_i}(F^{\mathrm{front}} \wt f_2 ) = c_{V_i}( \wt k_s^\Right)
$$
and
$$
c_{V_i}(F^{\mathrm{back}} \wt f_1 ) = c_{V_i}( \wt k_s^{\mathrm{switch}}).
$$

Also, the images of the sides $ F^{\mathrm{mid}}\wt f_1 $ and $F^{\mathrm{mid}}\wt f_2 $ are equal to the same edge of the triangulation $\Sigma$, but induce opposite orientations on this edge. It follows that $ c_{V_i} ( F^{\mathrm{mid}}\wt f_1 ) = -  c_{V_i} ( F^{\mathrm{mid}}\wt f_2 )$ for each $i=1$, $2$, since $c_{V_i}$ is closed. Using twice more the fact that $c_{V_i}$ is closed, 
\begin{align*}
 c_{V_i} ( F^{\mathrm{back}}\wt f_2 ) 
 &= c_{V_i} ( F^{\mathrm{mid}}\wt f_2 )  - c_{V_i} ( F^{\mathrm{front}}\wt f_2 ) 
 = - c_{V_i} ( F^{\mathrm{mid}}\wt f_1 )  - c_{V_i} ( F^{\mathrm{front}}\wt f_2 ) 
  \\
   &=- c_{V_i} ( F^{\mathrm{front}}\wt f_1 )  - c_{V_i} ( F^{\mathrm{back}}\wt f_1 )  - c_{V_i} ( F^{\mathrm{front}}\wt f_2 ) 
   \\
  &=- c_{V_i}( \wt k_s^\Left)    - c_{V_i}( \wt k_s^{\mathrm{switch}})  - c_{V_i}( \wt k_s^\Right) .
\end{align*}

 In the formula of Lemma~\ref{lem:GeneralFormulaCupProduct}, the sum of the contributions of $f_1$ and $f_2$ to $\left\langle [c_{V_1}] \cupp [c_{V_2}], [S] \right\rangle$ is equal to
\begin{align*}
&\tfrac12 \Bigl( K\bigl(c_{V_1}(F^{\mathrm{front}} \wt f_1),  c_{V_2}(F^{\mathrm{back}} \wt f_1) \bigr)-K\bigl(c_{V_2}(F^{\mathrm{front}} \wt f_1),  c_{V_1}(F^{\mathrm{back}} \wt f_1) \bigr)  \Bigr) 
\\
&\qquad
+ \tfrac12 \Bigl( K\bigl(c_{V_1}(F^{\mathrm{front}} \wt f_2),  c_{V_2}(F^{\mathrm{back}} \wt f_2) \bigr)-K\bigl(c_{V_2}(F^{\mathrm{front}} \wt f_2),  c_{V_1}(F^{\mathrm{back}} \wt f_2) \bigr)  \Bigr) .
\end{align*}
Substituting the above values for $c_{V_i}(F^{\mathrm{front}} \wt f_1 ) $, $c_{V_i}(F^{\mathrm{back}} \wt f_1 ) $, $c_{V_i}(F^{\mathrm{front}} \wt f_2 ) $ and $c_{V_i}(F^{\mathrm{back}} \wt f_2 ) $ and rearranging terms gives the formula stated in the lemma.
\end{proof}

For $i=1$, $2$, the values of $c_{V_i}(\wt k_s^{\mathrm{switch}})$, $c_{V_i}(\wt k_s^{\mathrm{left}})$ and $c_{V_i}(\wt k_s^{\mathrm{right}}) \in \sln$ are provided by Lemmas~\ref{lem:EvaluateWeilEdgeMeetingBarrier} and \ref{lem:EvaluateWeilTie}. However, the formulas of Lemma~\ref{lem:EvaluateWeilTie} for $c_{V_i}(\wt k_s^{\mathrm{left}})$ and $c_{V_i}(\wt k_s^{\mathrm{right}})$ involve infinite sums, which is not very convenient if we want to input these values in the formula of Lemma~\ref{lem:ContribSwitchToCupProd1}. We will consequently use the approximation of \S \ref{subsect:OpeningZippers}, and assume that $\Phi=\Phi_\delta$ is obtained from a first train track neighborhood $\Phi_0$ by opening zippers  up to depth $\delta\geq 1$. 

As before, lift the switch tie $s$ of $\Phi$ to a switch tie $\wt s$ of the preimage $\wt \Phi \subset \wt S$ of $\Phi$. The correspondence of Proposition~\ref{prop:ComplementsTrainTrackGeodLamination} associates a component $\wt T$ of $\wt S-\wt\lambda$ to the component of $\wt S - \wt\Phi$ that has a corner at $\wt s$, and also associates a vertex $x_{\wt T}\in \bdry$ of $\wt T$ to that corner.  Index the other two vertices of $\wt T$ as $y_{\wt T}$ and $z_{\wt T} \in \bdry$ so that, as usual, $x_{\wt T}$, $y_{\wt T}$ and $z_{\wt T}$ occur counterclockwise in this order around $\wt T$.

\begin{lem}
 \label{lem:ContribSwitchToCupProd2}
Let $\Phi$ be obtained from a first train track neighborhood $\Phi_0$ by opening zippers up to depth $\delta$.  Then, with the above data, there is a constant $C$ such that the contribution of the switch $s$ to $\left\langle [c_{V_1}] \cupp [c_{V_2}], [S] \right\rangle$ is of the form
\begin{align*}
 & \tfrac12 \sum_{\substack{a+b=n\\a'+b'=n}} \bigl( \dot \sigma_{V_1}^{ab}(e_s^\Left)  \dot \sigma_{V_2}^{a'b'}(e_s^\Right) -  \dot \sigma_{V_2}^{ab}(e_s^\Left)  \dot \sigma_{V_1}^{a'b'}(e_s^\Right)   \bigr) K\bigl( \dot S^{ab}_{E_{\wt T}G_{\wt T}}, \dot S^{a'b'}_{E_{\wt T}F_{\wt T}} \bigr)
 \\
 &\qquad
   -\tfrac12 \sum_{\substack{a+b=n\\a'+b'+c'=n}} \bigl( \dot \sigma_{V_1}^{ab}(e_s^{\mathrm{right}})  \dot \tau_{V_2}^{a'b'c'}(T,x) -  \dot \sigma_{V_2}^{ab}(e_s^{\mathrm{right}})  \dot  \tau_{V_1}^{a'b'c'}(T,x)   \bigr) K\bigl( \dot S^{ab}_{E_{\wt T}G_{\wt T}}, \dot L^{a'b'c'}_{E_{\wt T}F_{\wt T}G_{\wt T}} \bigr)
 \\
 &\qquad
   + \tfrac12 \sum_{\substack{a+b=n\\a'+b'+c'=n}} \bigl( \dot \sigma_{V_1}^{ab}(e_s^{\mathrm{left}})  \dot \tau_{V_2}^{a'b'c'}(T,x) -  \dot \sigma_{V_2}^{ab}(e_s^{\mathrm{left}})  \dot  \tau_{V_1}^{a'b'c'}(T,x)   \bigr) K\bigl( \dot S^{ab}_{E_{\wt T}F_{\wt T}}, \dot L^{a'b'c'}_{E_{\wt T}F_{\wt T}G_{\wt T}} \bigr)
   \\
   &\qquad
 + O(\E^{-C \delta}).
\end{align*}
The constant $C>0$ and the constant hidden in the Landau symbol $O(\ )$ depend only on the train track neighborhood $\Phi_0$,  on the Hitchin representation $\rho$ and on the tangent vectors $V_1$, $V_2 \in T_{[\rho]}\Hit(S)$. 
\end{lem}

\begin{proof}
Because the Killing form $K(\ , \ )$ is invariant under the adjoint representation, the proposed formula is independent of the lift $\wt s$ of the switch tie $s$. We can therefore assume that $\wt s$ is contained in a compact subset of $\wt S$ depending only on $\Phi_0$. Then,  Lemmas~\ref{lem:EvaluateWeilEdgeMeetingBarrier} and \ref{lem:EstimateWeilZipperOpening} provide  
\begin{align*}
c_{V_i}(\wt k_s^{\mathrm{switch}}) &= \sum_{a+b+c=n} \dot\tau^{abc}_{V_i}(T, x) \dot L^{abc}_{E_{\wt T}F_{\wt T}G_{\wt T}}
\\
 c_{V_i}(\wt k_s^{\mathrm{left}}) &= \sum_{a+b=n} \dot \sigma_{V_i}^{ab}(e_s^{\mathrm{left}}) \dot S^{ab}_{E_{\wt T}F_{\wt T}} + O(\E^{-C \delta})
 \\
 c_{V_i}(\wt k_s^{\mathrm{right}}) &= \sum_{a+b=n} \dot \sigma_{V_i}^{ab}(e_s^{\mathrm{right}}) \dot S^{ab}_{E_{\wt T}G_{\wt T}}  + O(\E^{-C \delta}),
\end{align*}
where the constant $C>0$ and the constants hidden in the Landau symbols $O(\ )$ depend only on that compact subset,  on the Hitchin representation $\rho$ and on the tangent vectors $V_1$, $V_2 \in T_{[\rho]}\Hit(S)$. 

The result then follows by entering these expressions in the formula of Lemma~\ref{lem:ContribSwitchToCupProd1}. 
\end{proof}

The Killing form terms are given by the following formulas.

\begin{lem}
\label{lem:ComputeKillingFormShearLeftRight}
For every maximum-span flag triple  $(E,F,G) \in \Flag^3$,
\begin{align*}
 K\left( \dot S^{ab}_{EG}, \dot S^{a'b'}_{EF} \right)&=
\begin{cases}
2a'b  &\text{if } a\geq a'
\\
 2ab' &\text{if } a\leq a'
\end{cases}
\\
 K\left( \dot S^{ab}_{EG}, \dot L^{a'b'c'}_{EFG} \right)&=
\begin{cases}
 -2a'b   &\text{if } a\geq a' 
 \\
 -2a(b'+c')   &\text{if } a\leq a' 
\end{cases}
\\
 K\left( \dot S^{ab}_{EF},  \dot L^{a'b'c'}_{EFG} \right)&=
\begin{cases}
- 2a'b  &\text{if } a\geq a'
 \\
-2a(b'+c')  &\text{if } a\leq a'  .
\end{cases}
\end{align*}
\end{lem}

\begin{proof}
 We use the computations of the proof of Lemma~\ref{lem:ComputeKillingFormLeftRight}. As in that proof consider, for any maximum-span flag triple $(E,F,G) \in \Flag^3$, the projection
 $$
P_{EFG}^{abc} = \Id_{E^{(a)}} \oplus 0 \, \Id_{F^{(b)}}  \oplus 0 \, \Id_{G^{(c)}}. 
 $$
Then, 
\begin{align*}
 \dot S_{EF}^{ab} &= \tfrac bn \Id_{E^{(a)}} \oplus  \tfrac {-a}n \Id_{F^{(b)}} 
 = \tfrac bn \Id_{\R^n} - P_{FGE}^{b0a}
 = \tfrac {-a}n \Id_{\R^n} + P_{EFG}^{ab0}
 \\
 \dot S_{EG}^{ab} &= \tfrac bn \Id_{E^{(a)}} \oplus  \tfrac {-a}n \Id_{G^{(b)}} 
  = \tfrac bn \Id_{\R^n} - P_{GEF}^{ba0}
 = \tfrac {-a}n \Id_{\R^n} + P_{EFG}^{a0b}
 \\
 \dot L_{EFG}^{abc} &= \tfrac {-b-c}n \Id_{E^{(a)}} \oplus  \tfrac an \Id_{F^{(b)}}  \oplus  \tfrac an \Id_{G^{(c)}} = \tfrac an \Id_{\R^n} - P_{EFG}^{abc}.
\end{align*}

In the proof of Lemma~\ref{lem:ComputeKillingFormLeftRight}, we computed that
$$
\mathrm{Tr}\, P_{FGE}^{\beta\gamma\alpha}P_{EFG}^{\alpha'\beta'\gamma'}  =
\begin{cases}
 0 &\text{if } \alpha\geq \alpha' \text{ or } \beta\leq \beta'
 \\
 \beta-\beta' &\text{if } \alpha\leq \alpha' ,\  \beta \geq \beta' \text{ and } \gamma\geq \gamma'
 \\
 \alpha'-\alpha &\text{if } \alpha\leq \alpha' ,\  \beta \geq \beta' \text{ and } \gamma\leq \gamma'.
\end{cases}
$$
Exchanging the roles of $F$ and $G$, and noting that $P_{GFE}^{\gamma\beta\alpha} =P_{GEF}^{\gamma\alpha\beta}$ and $P_{EGF}^{\alpha'\gamma'\beta'} =P_{EFG}^{\alpha'\beta'\gamma'}$, this also implies that
$$
\mathrm{Tr}\, P_{GEF}^{\gamma\alpha\beta}P_{EFG}^{\alpha'\beta'\gamma'}  =
\begin{cases}
 0 &\text{if } \alpha\geq \alpha' \text{ or } \gamma\leq \gamma'
 \\
 \gamma-\gamma' &\text{if } \alpha\leq \alpha' ,\  \gamma \geq \gamma' \text{ and } \beta\geq \beta'
 \\
 \alpha'-\alpha &\text{if } \alpha\leq \alpha' ,\  \gamma \geq \gamma' \text{ and } \beta\leq \beta'.
\end{cases}
$$

Therefore,
\begin{align*}
 K\left( \dot S^{ab}_{EG}, \dot S^{a'b'}_{EF} \right)&=2n\, 
 \mathrm{Tr}\, \left(  \tfrac bn \Id_{\R^n} - P_{GEF}^{ba0} \right) \left(    \tfrac {-a'}n \Id_{\R^n} + P_{EFG}^{a'b'0} \right)
 \\
 &=- \tfrac{2a'b}{n} \mathrm{Tr}\, \Id_{\R^n}  + 2b  \mathrm{Tr}\, P_{EFG}^{a'b'0} + 2a'  \mathrm{Tr}\,P_{GEF}^{ba0} - 2n\, \mathrm{Tr}\,P_{GEF}^{ba0} P_{EFG}^{a'b'0} 
 \\
 &=- 2a'b  + 2a'b + 2a'b - 2n\, \mathrm{Tr}\,P_{GEF}^{ba0} P_{EFG}^{a'b'0} 
 \\
 &=
\begin{cases}
2a'b  &\text{if } a\geq a'
\\
 2ab' &\text{if } a\leq a',
\end{cases}
\end{align*}
using the property that $a+b=a'+b'=n$. 

Similarly,
\begin{align*}
 K\left( \dot S^{ab}_{EG}, \dot L^{a'b'c'}_{EFG} \right)&=2n\, 
 \mathrm{Tr}\, \left( \tfrac bn \Id_{\R^n} - P_{GEF}^{ba0} \right) \left(   \tfrac {a'}n \Id_{\R^n} - P_{EFG}^{a'b'c'} \right)
 \\
 &= 2a'b - 2a'b -2a'b + 2n \,  \mathrm{Tr}\, P_{GEF}^{ba0} P_{EFG}^{a'b'c'} 
 \\
 &=
\begin{cases}
 -2a'b   &\text{if } a\geq a' \text{ or } b \leq c'
 \\
 -2a(b'+c')   &\text{if } a\leq a' \text{ and } b \geq c'.
\end{cases}
\end{align*}
This can be simplified by noticing that, if $a\leq a'$, then necessarily
$$
b = n-a \geq n-a' >n-a'-b'=c', 
$$
so that the second condition is irrelevant. Therefore,
$$
K\left( \dot S^{ab}_{EG}, \dot L^{a'b'c'}_{EFG} \right)=
\begin{cases}
 -2a'(n-a)   &\text{if } a\geq a' 
 \\
 -2a(n-a')   &\text{if } a\leq a' .
\end{cases}
$$

The same argument gives
\begin{align*}
 K\left( \dot S^{ab}_{EF}, \dot L^{a'b'c'}_{EFG} \right)&=2n\, 
 \mathrm{Tr}\, \left(  \tfrac bn \Id_{\R^n} - P_{FGE}^{b0a} \right) \left(   \tfrac {a'}n \Id_{\R^n} - P_{EFG}^{a'b'c'} \right)
 \\
 &= 2a'b -2 a'b -2a'b + 2n\, \mathrm{Tr}\, P_{FGE}^{b0a} P_{EFG}^{a'b'c'}
 \\
 &=
\begin{cases}
- 2a'b  &\text{if } a\geq a' \text{ or } b \leq b'
 \\
-2a(b'+c')  &\text{if } a\leq a'  \text{ and } b \geq b'
\end{cases}
 \\
 &=
\begin{cases}
- 2a'b  &\text{if } a\geq a' 
 \\
-2a(b'+c')  &\text{if } a\leq a'  .
\end{cases}
\end{align*}

This concludes the proof.  
\end{proof}

\subsection{An approximate formula}
\label{subsect:ApproximateFormula}

We can now put everything together. 

For a switch $s$ of a train track $\Psi$, let $U_s$ denote the component of $S-\Psi$ that has a corner at $s$. Also, arbitrarily choose a corner $s_U$ (corresponding to a switch of $\Psi$) for each component $U$ of $S-\Psi$. 

Every oriented branch $e$ of $\Psi$ and every $a$, $b\geq 1$ with $a+b=n$, the corresponding generalized Fock-Goncharov invariant defines a function  $\sigma^{ab} (e): \Hit(S) \to \R$, and therefore a 1--form $d\sigma^{ab} (e) \in \Omega^1 \big (\Hit(S) \big)$. Similarly, every component $U$ of $S-\Psi$ and every triple of integers $a$, $b$, $c\geq 1$ with $a+b+c=n$ provides another 1--form $d\tau^{abc} (U, s_U) \in \Omega^1 \big (\Hit(S) \big)$. 

\begin{lem}
\label{lem:ApproximateFormula}
Let $\Phi_\delta$ be obtained from a first train track neighborhood $\Phi_0$  of the geodesic lamination $\lambda$ by opening zippers up to depth $\delta\geq1$, and let $\Psi_\delta$ be the train track corresponding to $\Phi_\delta$. At each switch $s$ of $\Psi_\delta$, orient the incoming switches $e_s^\Left$ and $e_s^\Right$ toward that switch. Then,  the Atiyah-Bott-Goldman symplectic form $\omega \in \Omega^2 \big(\Hit(S) \big)$ at $[\rho] \in \Hit(S)$ is of the form  
\begin{align*}
 \omega_{[\rho]} &= 
 \sum_{\substack{U\text{component}\\ \text{of }S-\Psi_\delta}} \sum_{\substack{a+b+c=n\\a'+b'+c'=n}} \tfrac12 C^{abc}_{a'b'c'} \, d\tau^{abc}(U, s_U) \wedge d\tau^{a'b'c'}(U, s_U) 
\\
&\qquad
+\sum_{\substack{s\text{ switch}\\\text{of } \Psi_\delta}}  
\sum_{\substack{a+b=n\\a'+b'=n}} C^{ab}_{a'b'}  \, d\sigma^{ab}(e_s^{\mathrm{left}})  \wedge d\sigma^{a'b'}(e_s^{\mathrm{right}}) 
\\
&\qquad
-\sum_{\substack{s\text{ switch}\\\text{of } \Psi_\delta}}  \sum_{\substack{a+b=n\\a'+b'+c'=n}} C^{ab}_{a'(b'+c')} \, d\sigma^{ab}(e_s^{\mathrm{left}})  \wedge d\tau^{a'b'c'}(U_s, s) 
\\
&\qquad
+\sum_{\substack{s\text{ switch}\\\text{of } \Psi_\delta}}  \sum_{\substack{a+b=n \\a'+b'+c'=n}} C^{ab}_{a'(b'+c')} \, d\sigma^{ab}(e_s^{\mathrm{right}})  \wedge d\tau^{a'b'c'}(U_s, s) 
\\
&\qquad
+ O(\E^{-C\delta}),
\end{align*}
 where the constants $C^{abc}_{a'b'c'} $ and $C^{ab}_{a'b'}$ are given by
\begin{align*}
C^{abc}_{a'b'c'}&=
\begin{cases}
a'b-ab' &\text{if } a\leq a', b\leq b' \text{ and } c\geq c'\\
ac'-a'c &\text{if } a\leq a', b\geq b' \text{ and } c\leq c'\\
b'c-bc' &\text{if } a\leq a', b\geq b' \text{ and } c\geq c'\\
b'c-bc' &\text{if } a\geq a', b\leq b' \text{ and } c\leq c'\\
ac'-a'c &\text{if } a\geq a', b\leq b' \text{ and } c\geq c'\\
a'b -ab'&\text{if } a\geq a', b\geq b' \text{ and } c\leq c'
\end{cases}
\\
 C^{ab}_{a'b'} &=
\begin{cases}
a'b  &\text{if } a\geq a'
\\
 ab' &\text{if } a\leq a',
\end{cases}
\end{align*}
and where the constant $C>0$ and the constant hidden in the Landau symbol $O(\ )$ depend only on a Hitchin representation $\rho \colon \pi_1(S) \to \PGL$ representing the  character $[\rho] \in \Hit(S)$ and on the original train track neighborhood $\Phi_0$. 
\end{lem}

\begin{proof}

Let $V_1$, $V_2 \in T_{[\rho]} \Hit(S)$ be two tangent vectors of the Hitchin component at $[\rho] \in \Hit(S)$, respectively associated to variations $\dot \tau_{V_1}$, $\dot \sigma_{V_1}$ and $\dot\tau_{V_2}$, $\dot\sigma_{V_2}$ of the generalized Fock-Goncharov invariants of $[\rho]$. Lemmas~\ref{lem:ContribTriangleToCupProd2}, \ref{lem:ContribSwitchToCupProd2} and \ref{lem:ComputeKillingFormShearLeftRight}  then give us that the pairing of $\omega(V_1, V_2) = \left\langle [c_{V_1}] \cupp [c_{V_2}], [S]  \right\rangle $ of $V_1$ and $V_2$ under the symplectic form $\omega$ is of the form
\begin{align*}
&\sum_{\substack{U\text{component}\\ \text{of }S-\Psi_\delta}} \sum_{\substack{a+b+c=n\\a'+b'+c'=n}} \tfrac12 C^{abc}_{a'b'c'} \left( \dot\tau_{V_1}^{abc}(U, s_U) \,\dot\tau_{V_2}^{a'b'c'}(U, s_U) 
 -\dot\tau_{V_2}^{abc}(U, s_U) \, \dot\tau_{V_1}^{a'b'c'}(U, s_U)  \right) 
\\
&\qquad
+\sum_{\substack{s\text{ switch}\\\text{of } \Psi_\delta}}  \sum_{\substack{a+b=n\\a'+b'=n}} C^{ab}_{a'b'}  \left( \dot \sigma_{V_1}^{ab}(e_s^{\mathrm{left}})  \dot \sigma_{V_2}^{a'b'}(e_s^{\mathrm{right}}) -  \dot \sigma_{V_2}^{ab}(e_s^{\mathrm{left}})  \dot \sigma_{V_1}^{a'b'}(e_s^{\mathrm{right}})   \right) 
\\
&\qquad
-\sum_{\substack{s\text{ switch}\\\text{of } \Psi_\delta}}  \sum_{\substack{a+b=n\\a'+b'+c'=n}} C^{ab}_{a'(b'+c')} \left( \dot \sigma_{V_1}^{ab}(e_s^{\mathrm{left}})  \dot \tau_{V_2}^{a'b'c'}(U_s, s) -  \dot \sigma_{V_2}^{ab}(e^{\mathrm{left}})  \dot  \tau_{V_1}^{a'b'c'}(U_s, s_U)   \right)
\\
&\qquad
+\sum_{\substack{s\text{ switch}\\\text{of } \Psi_\delta}}  \sum_{\substack{a+b=n\\a'+b'+c'=n}} C^{ab}_{a'(b'+c')} \left( \dot \sigma_{V_1}^{ab}(e_s^{\mathrm{right}})  \dot \tau_{V_2}^{a'b'c'}(U_s, s) -  \dot \sigma_{V_2}^{ab}(e^{\mathrm{right}})  \dot  \tau_{V_1}^{a'b'c'}(U_s, s)   \right)
\\
&\qquad
+ O(\E^{-C\delta}),
\end{align*}
which is exactly the property that we needed. 
\end{proof}

We will  show that the above estimate is actually an equality, with the error term $ O(\E^{-C\delta})$ equal to 0.

\subsection{Invariance under zipper opening} For a train track neighborhood $\Phi_\delta$ of $\lambda$ obtained from a first train track neighborhood $\Phi_0$ by opening zippers up to depth $\delta$, Lemma~\ref{lem:ApproximateFormula} estimates the symplectic form $\omega$ at $[\rho] \in \Hit(S)$ as 
$$
\omega_{[\rho]} = A (\Phi_\delta) +  O(\E^{-C\delta})
$$
where, for any train track neighborhood $\Phi$ of $\lambda$ with associated train track $\Psi$, 
\begin{align*}
A(\Phi)&=\sum_{\substack{U\text{component}\\ \text{of }S-\Psi}} \sum_{\substack{a+b+c=n\\a'+b'+c'=n}} \tfrac12 C^{abc}_{a'b'c'} \, d\tau^{abc}(U,s_U) \wedge d\tau^{a'b'c'}(U,s_U) 
\\
&\qquad
+\sum_{\substack{s\text{ switch}\\\text{of } \Psi}}  \sum_{\substack{a+b=n\\a'+b'=n}} C^{ab}_{a'b'}  \, d\sigma^{ab}(e_s^{\mathrm{left}})  \wedge d \sigma^{a'b'}(e_s^{\mathrm{right}}) 
\\
&\qquad
-\sum_{\substack{s\text{ switch}\\\text{of } \Psi}}  \sum_{\substack{a+b=n\\a'+b'+c'=n}} C^{ab}_{a'(b'+c')} \, d\sigma^{ab}(e_s^{\mathrm{left}})  \wedge d \tau^{a'b'c'}(U_s, s) 
\\
&\qquad
+\sum_{\substack{s\text{ switch}\\\text{of } \Psi}}  \sum_{\substack{a+b=n\\a'+b'+c'=n}} C^{ab}_{a'(b'+c')} \, d \sigma^{ab}(e_s^{\mathrm{right}})  \wedge d \tau^{a'b'c'}(U_s, s).
\end{align*}

We show that this approximation is invariant under zipper opening. 

\begin{lem}
\label{lem:InvarianceZipperOpening}
If the train track neighborhood $\Phi'$ is obtained from $\Phi$ by a zipper opening along an arc $a$, then
$$A(\Phi')=A(\Phi).$$
\end{lem}

\begin{proof} By definition of zipper openings, the  arc $a$   starts at a switch point, is disjoint from the geodesic lamination, and is transverse to the ties of $\Phi$. If $a$ is disjoint from the switch ties of $\Phi$ (except at its initial point), then the combinatorial structure of the associated train track is unchanged, and $A(\Phi')=A(\Phi)$.

We can therefore assume, without loss of generality, that $a$ crosses the switch ties of $\Phi$ in exactly one point. Then, there are several cases to consider, according to the way $a$ crosses that switch tie. The upshot is that  the train tracks $\Psi$ and $\Psi'$ respectively associated to $\Phi$ and $\Phi'$ differ by one of the  moves of Figures~\ref{fig:ZipperOpening1} and \ref{fig:ZipperOpening2} or their mirror images. 

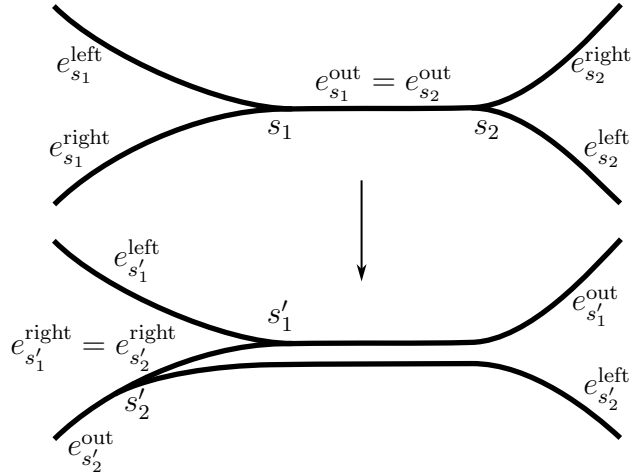
\begin{figure}[htbp]

\SetLabels
(  .4* .7 )  $s_1$ \\
(  .76* .7 )  $s_2$ \\
(.05  * .84 )  $e_{s_1}^\Left$ \\
( .05 * .65 )  $e_{s_1}^\Right$ \\
(  .58* .8 )  $e_{s_1}^\Out=e_{s_2}^\Out$ \\
( .96 * .84 )  $e_{s_2}^\Right$ \\
( .97 * .65 )  $e_{s_2}^\Left$ \\
(  .4* .27 )  $s_1'$ \\
(  .15* .07 )  $s_2'$ \\
(.15  * .4 )  $e_{s_1'}^\Left$ \\
( .07 *- .03 )  $e_{s_2'}^\Out$ \\
\R(  .22* .2 )  $e_{s_1'}^\Right=e_{s_2'}^\Right$ \\
( .95 * .3 )  $e_{s_1'}^\Out$ \\
( .97 * .11 )  $e_{s_2'}^\Left$ \\
\endSetLabels
\centerline{\AffixLabels{\includegraphics{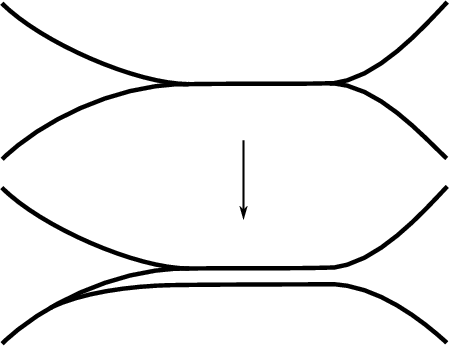}}}

\caption{Opening zippers I}
\label{fig:ZipperOpening1}
\end{figure}

\begin{figure}[htbp]

\SetLabels
(  .42* .78 )  $s_1$ \\
(  .76* .8 )  $s_2$ \\
(.65  * .89 )  $e_{s_1}^\Right$ \\
( .05 * .65 )  $e_{s_1}^\Out$ \\
(  .61* .68 )  $e_{s_1}^\Left=e_{s_2}^\Out$ \\
( .96 * .84 )  $e_{s_2}^\Right$ \\
( .97 * .65 )  $e_{s_2}^\Left$ \\
(  .42* .27 )  $s_1'$ \\
(  .15* .07 )  $s_2'$ \\
(.69  * .4 )  $e_{s_1'}^\Right$ \\
( .07 *- .03 )  $e_{s_2'}^\Out$ \\
\R(  .3* .24 )  $e_{s_1'}^\Out=e_{s_2'}^\Right$ \\
( .95 * .3 )  $e_{s_1'}^\Left$ \\
( .97 * .11 )  $e_{s_2'}^\Left$ \\
\endSetLabels
\centerline{\AffixLabels{\includegraphics{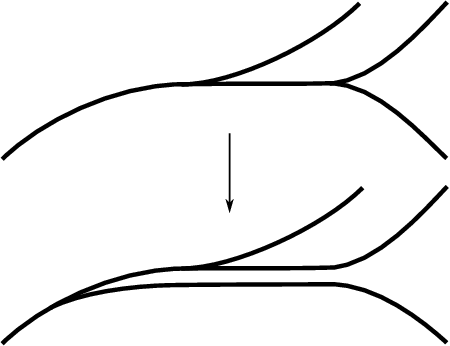}}}

\caption{Opening zippers II}
\label{fig:ZipperOpening2}
\end{figure}

We first look at the case of Figure~\ref{fig:ZipperOpening1}. Then, $A(\Phi)$ and $A(\Phi')$ differ only in the contribution of the two switches $s_1$ and $s_2$ appearing in the picture. The contribution to $A(\Phi)$ of the switches $s_1$, $s_2$ is equal to
\begin{align*}
 &
 \sum_{\substack{a+b=n\\a'+b'=n}} C^{ab}_{a'b'} \, d \sigma^{ab}(e_{s_1}^{\mathrm{left}}) \wedge d \sigma^{a'b'}(e_{s_1}^{\mathrm{right}}) 
 +
 \sum_{\substack{a+b=n\\a'+b'=n}} C^{ab}_{a'b'} \, d \sigma^{ab}(e_{s_2}^{\mathrm{left}}) \wedge d \sigma^{a'b'}(e_{s_2}^{\mathrm{right}}) 
\\
&\qquad
- \sum_{\substack{a+b=n\\a'+b'+c'=n}} C^{ab}_{a'(b'+c')} \, d \sigma^{ab}(e_{s_1}^{\mathrm{left}})  \wedge d \tau^{a'b'c'}(U_{s_1}, s_1) 
\\
&\qquad\qquad \qquad\qquad\qquad\qquad\qquad
- \sum_{\substack{a+b=n\\a'+b'+c'=n}} C^{ab}_{a'(b'+c')} \, d \sigma^{ab}(e_{s_2}^{\mathrm{left}})  \wedge d \tau^{a'b'c'}(U_{s_2}, s_2) 
\\
&\qquad
+ \sum_{\substack{a+b=n\\a'+b'+c'=n}} C^{ab}_{a'(b'+c')}  \, d \sigma^{ab}(e_{s_1}^{\mathrm{right}})  \wedge d \tau^{a'b'c'}(U_{s_1}, s_1)
\\
&\qquad\qquad\qquad\qquad\qquad\qquad\qquad
+ \sum_{\substack{a+b=n\\a'+b'+c'=n}} C^{ab}_{a'(b'+c')} \, d \sigma^{ab}(e_{s_2}^{\mathrm{right}})  \wedge d \tau^{a'b'c'}(U_{s_2}, s_2),
\end{align*}
and the contribution of $s_1'$ and $s_2'$ to $A(\Phi')$ is obtained from the same formula by replacing  $s_1$ by $s_1'$ and $s_2$ by $s_2'$ everywhere. 

Note that, for every $a$, $b$, $a'$, $b'$, $c'$,  
\begin{align*}
\tau^{a'b'c'}(U_{s_1}, s_1) &= \tau^{a'b'c'}(U_{s_1'}, s_1')
&
\tau^{a'b'c'}(U_{s_2}, s_2) &= \tau^{a'b'c'}(U_{s_2'}, s_2')
\\
 \sigma^{ab}(e_{s_1}^{\mathrm{left}}) &=  \sigma^{ab}(e_{s_1'}^{\mathrm{left}})
 &
  \sigma^{ab}(e_{s_2}^{\mathrm{left}}) &=  \sigma^{ab}(e_{s_2'}^{\mathrm{left}}),
\end{align*}
and 
$$
 \sigma^{ab}(e_{s_2'}^{\mathrm{right}}) =  \sigma^{ba}(e_{s_1'}^{\mathrm{right}})
 $$
 since we orient the left/right branches toward the associated switch; in particular, the branches $e_{s_1'}^{\mathrm{right}}$ and $e_{s_2'}^{\mathrm{right}}$ coincide but come with opposite orientations. Also, by the Switch Condition of \S \ref{subsect:GenFocGonInvariants}, 
 \begin{align*}
 \sigma^{ab}(e_{s_1}^{\mathrm{right}})
 &=   \sigma^{ba}(e_{s_2'}^{\mathrm{out}})
 = \sigma^{ba}(e_{s_2'}^{\mathrm{left}}) +  \sigma^{ba}(e_{s_2'}^{\mathrm{right}}) - \sum_{\substack{b'', c''\geq 1\\b'' + c'' = n-b}}  \tau^{bb''c''}(U_{s_2'}, s_2')
 \\
&=
  \sigma^{ba}(e_{s_2'}^{\mathrm{left}}) +  \sigma^{ab}(e_{s_1'}^{\mathrm{right}}) - \sum_{\substack{b'', c''\geq 1\\b'' + c'' = n-b}}  \tau^{bb''c''}(U_{s_2'}, s_2')
\end{align*}
and, similarly,
$$
 \sigma^{ab}(e_{s_2}^{\mathrm{right}})
=
  \sigma^{ba}(e_{s_1'}^{\mathrm{left}}) +  \sigma^{ab}(e_{s_2'}^{\mathrm{right}}) - \sum_{\substack{b'', c''\geq 1\\b'' + c'' = n-b}}  \tau^{bb''c''}(U_{s_1'}, s_1').
  $$

Combining all these observations and comparing contributions term by term, we can now express $A(\Phi) - A(\Phi')$ as
\begin{align*}
&
  \sum_{\substack{a+b=n\\a'+b'=n}} C^{ab}_{a'b'} \, d \sigma^{ab}(e_{s_1'}^{\mathrm{left}}) \wedge 
  \bigg(
  d \sigma^{b'a'}(e_{s_2'}^{\mathrm{left}}) - \sum_{\substack{b'', c''\geq 1\\b'' + c'' = n-b'}}  \, d\tau^{b'b''c''}(U_{s_2'}, s_2')
  \bigg)
  \\
  &\qquad
  + 
  \sum_{\substack{a+b=n\\a'+b'=n}} C^{ab}_{a'b'} \, d \sigma^{ab}(e_{s_2'}^{\mathrm{left}}) \wedge 
  \bigg(
  d \sigma^{b'a'}(e_{s_1'}^{\mathrm{left}}) - \sum_{\substack{b'', c''\geq 1\\b'' + c'' = n-b'}}  d\tau^{b'b''c''}(U_{s_1'}, s_1')
  \bigg)
  \\
& \qquad
+ \sum_{\substack{a+b=n\\a'+b'+c'=n}} C^{ab}_{a'(b'+c')} 
 \bigg(
  d \sigma^{ba}(e_{s_2'}^{\mathrm{left}}) - \sum_{b'' + c'' = n-b}  d\tau^{bb''c''}(U_{s_2'}, s_2')
  \bigg)
\wedge d \tau^{a'b'c'}(U_{s_1'}, s_1')
  \\
& \qquad
+ \sum_{\substack{a+b=n\\a'+b'+c'=n}} C^{ab}_{a'(b'+c')}  
\bigg(
  d \sigma^{ba}(e_{s_1'}^{\mathrm{left}}) - \sum_{b'' + c'' = n-b}  d\tau^{bb''c''}(U_{s_1'}, s_1')
  \bigg)
\wedge d \tau^{a'b'c'}(U_{s_2'}, s_2')
\end{align*}

Expanding the above expression and renaming indices, this is equal to
\begin{align*}
&
  \sum_{\substack{a+b=n\\a'+b'=n}} C^{ab}_{b'a'}  d \sigma^{ab}(e_{s_1'}^{\mathrm{left}}) \wedge 
  d \sigma^{a'b'}(e_{s_2'}^{\mathrm{left}}) 
  \\
   &\qquad\qquad\qquad
  -  \sum_{\substack{a+b=n\\a'+ b' + c' = n}}  C^{ab}_{(b'+c')a'} \, d \sigma^{ab}(e_{s_1'}^{\mathrm{left}}) \wedge   d\tau^{a'b'c'}(U_{s_2'}, s_2')
  \\
  +&
   \sum_{\substack{a+b=n\\a'+b'=n}} C^{a'b'}_{ba} \, d \sigma^{a'b'}(e_{s_2'}^{\mathrm{left}}) \wedge 
  d \sigma^{ab}(e_{s_1'}^{\mathrm{left}}) 
  \\
   &\qquad\qquad\qquad
   -   \sum_{\substack{a+b=n\\a'+ b' + c' = n}}   C^{ab}_{(b'+c')a'} \, d \sigma^{ab}(e_{s_2'}^{\mathrm{left}}) \wedge   d\tau^{a'b'c'}(U_{s_1'}, s_1')
  \\
+& 
 \sum_{\substack{a+b=n\\a'+ b' + c' = n}}  C^{ba}_{a'(b'+c')} \,
  d \sigma^{ab}(e_{s_2'}^{\mathrm{left}}) \wedge d \tau^{a'b'c'}(U_{s_1'}, s_1') 
  \\
   &\qquad\qquad\qquad
    - \sum_{\substack{a+b+c=n\\a'+ b' + c' = n}}   C^{(b'+c')a'}_{a(b+c)} \, d\tau^{a'b'c'}(U_{s_2'}, s_2')
\wedge d \tau^{abc}(U_{s_1'}, s_1')
  \\
+& 
  \sum_{\substack{a+b=n\\a'+ b' + c' = n}}  C^{ba}_{a'(b'+c')} \,
  d \sigma^{ab}(e_{s_1'}^{\mathrm{left}}) \wedge d \tau^{a'b'c'}(U_{s_2'}, s_2')
  \\
   &\qquad\qquad\qquad
   -  \sum_{\substack{a+b+c=n\\a'+ b' + c' = n}}    C^{(b+c)a}_{a'(b'+c')}  d\tau^{abc}(U_{s_1'}, s_1')
\wedge d \tau^{a'b'c'}(U_{s_2'}, s_2').
\end{align*}

From the definition of the constant $C^{ab}_{a'b'}$ in Lemma~\ref{lem:ApproximateFormula}, we have that $C^{ab}_{a'b'} = C^{a'b'}_{ab}=C^{ba}_{b'a'}$ for every $a$, $b$, $a'$, $b'$. In particular, $C^{ab}_{b'a'} = C^{a'b'}_{ba}$ and the first and third line therefore cancel out. 

Similarly,  $C^{ab}_{(b'+c')a'} = C^{ba}_{a'(b'+c')} $ and the second line cancels out with the seventh, while the fourth line cancels out with the fifth. 

Finally, $C^{(b'+c')a'}_{a(b+c)} = C^{(b+c)a}_{a'(b'+c')} $ and the sixth line cancels out with the eighth. 

Therefore, $A(\Phi) = A (\Phi')$ when $\Psi'$ is obtained from $\Psi$ by the move of Figure~\ref{fig:ZipperOpening1}.

A very similar argument shows that  $A(\Phi) = A (\Phi')$ when $\Psi'$ is obtained from $\Psi$ by the move of Figure~\ref{fig:ZipperOpening2}. Exchanging ``left'' and ``right'' and reversing signs everywhere also proves invariance under the mirror images of these two moves. 
\end{proof}

\section{Proof of the main theorem}
\label{bigsect:ProofMainThm}

We are now ready to prove our main theorem, namely Theorem~\ref{thm:MainThm} of the Introduction. With the notation of \S \ref{subsect:GenFocGonInvariants} and subsequent sections, this is the following statement. 

\begin{thm}
 \label{thm:MainThm2}
 
 Let $\lambda$ be a maximal geodesic lamination in the closed oriented surface $S$, and let $\Psi$ be a train track carrying $\lambda$. At each switch $s$ of $\Psi$, orient the incoming switches $e_s^\Left$ and $e_s^\Right$ toward that switch. Then, the Atiyah-Bott-Goldman symplectic form $\omega$ of the Hitchin component $\Hit(S)$ can be expressed in terms of the generalized Fock-Goncharov coordinate functions $\tau^{abc}(U, s)$ and $\sigma^{ab}(b)$ as
 \begin{align*}
\omega&= \sum_{\substack{U\text{component}\\ \text{of }S-\Psi}} \sum_{\substack{a,b,c\geq 1\\a+b+c=n}}  \sum_{\substack{a',b', c'\geq 1\\a'+b'+c'=n}}  \tfrac12 C^{abc}_{a'b'c'} \, d\tau^{abc}(U,s_U) \wedge d\tau^{a'b'c'}(U,s_U) 
\\
&\qquad
+\sum_{\substack{s\text{ switch}\\\text{of } \Psi}}  \sum_{\substack{a,b\geq 1\\a+b=n}}  \sum_{\substack{a',b'\geq 1\\a'+b'=n}} C^{ab}_{a'b'}  \, d\sigma^{ab}(e_s^{\mathrm{left}})  \wedge d \sigma^{a'b'}(e_s^{\mathrm{right}}) 
\\
&\qquad
-\sum_{\substack{s\text{ switch}\\\text{of } \Psi}}  \sum_{\substack{a,b\geq 1\\a+b=n}}  \sum_{\substack{a',b', c'\geq 1\\a'+b'+c'=n}} C^{ab}_{a'(b'+c')} \, d\sigma^{ab}(e_s^{\mathrm{left}})  \wedge d \tau^{a'b'c'}(U_s, s) 
\\
&\qquad
+\sum_{\substack{s\text{ switch}\\\text{of } \Psi}}  \sum_{\substack{a,b\geq 1\\a+b=n}}  \sum_{\substack{a',b', c'\geq 1\\a'+b'+c'=n}} C^{ab}_{a'(b'+c')} \, d \sigma^{ab}(e_s^{\mathrm{right}})  \wedge d \tau^{a'b'c'}(U_s, s).
\end{align*}
with
\begin{align*}
C^{abc}_{a'b'c'}&=
\begin{cases}
a'b-ab' &\text{if } a\leq a', b\leq b' \text{ and } c\geq c'\\
ac'-a'c &\text{if } a\leq a', b\geq b' \text{ and } c\leq c'\\
b'c-bc' &\text{if } a\leq a', b\geq b' \text{ and } c\geq c'\\
b'c-bc' &\text{if } a\geq a', b\leq b' \text{ and } c\leq c'\\
ac'-a'c &\text{if } a\geq a', b\leq b' \text{ and } c\geq c'\\
a'b -ab'&\text{if } a\geq a', b\geq b' \text{ and } c\leq c'
\end{cases}
\\
 C^{ab}_{a'b'} &=
\begin{cases}
a'b  &\text{if } a\geq a'
\\
 ab' &\text{if } a\leq a'.
\end{cases}
\end{align*}
\end{thm}

\begin{proof} Let $\Phi_\delta$ be obtained from the train track neighborhood $\Phi_0$ corresponding to $\Psi$ by opening zippers up to depth $\delta$. 
Lemma \ref{lem:ApproximateFormula} and \ref{lem:InvarianceZipperOpening} show that
$$\omega = A(\Phi_\delta) + O(\E^{-C\delta})=A(\Phi_0) + O(\E^{-C\delta})$$ 
for some constant $C>0$. Letting $\delta$ tend to infinity  proves that the approximation is actually an equality. 
\end{proof}

\bibliographystyle{amsalpha}
\bibliography{HitchinSymplectic}

\end{document}